\documentclass[a4paper,12pt,toc=flat]{scrartcl}
\usepackage{amsmath,amssymb,latexsym,amsthm,enumerate,mathrsfs}
\usepackage{thmtools}
\usepackage{xcolor}
\usepackage{graphicx}
\usepackage[nocompress]{cite}

\usepackage{sidecap}
\sidecaptionvpos{figure}{t}

\usepackage[T1]{fontenc}
\usepackage[utf8]{inputenc}
\usepackage[english]{babel}

\usepackage[hypertexnames=false]{hyperref}
\usepackage[capitalise,noabbrev]{cleveref}

\newcommand*{\wt}{\widetilde}
\newcommand*{\ol}{\overline}

\newcommand*{\cA}{\mathcal{A}}

\newcommand*{\cF}{\mathcal{F}}

\newcommand*{\cL}{\mathcal{L}}

\newcommand*{\cS}{\mathbb{S}}

\newcommand*{\N}{\mathbb{N}}

\newcommand*{\R}{\mathbb{R}}

\DeclareMathOperator*{\essinf}{ess\,inf}

\newcommand*{\lam}{\lambda}

\newcommand*{\diag}{\operatorname{diag}}

\newcommand*{\sgn}{\operatorname{sgn}}

\newcommand*{\mD}{\widetilde{D}}

\newcommand*{\aphi}{\widetilde{\nu}}

\newcommand*{\scH}{\mathcal{H}}
\newcommand*{\scHhat}{\widehat{\mathcal{H}}}

\newcommand*{\fvJ}{J^{fv}}
\newcommand*{\pmJ}{J^{pm}}
\newcommand*{\LQJ}{J}
\newcommand*{\LQJhat}{\hat{J}}

\newcommand*{\OO}{O}
\newcommand*{\risk}{\Xi}
\newcommand*{\kap}{\kappa}

\newcommand*{\KK}{\mathscr{R}}
\renewcommand*{\lam}{\boldsymbol{\lambda}}
\newcommand*{\diagmu}{\boldsymbol{\mu}}
\newcommand*{\diagsigma}{\boldsymbol{\sigma}}
\newcommand*{\md}{\mathbf{d}}
\newcommand*{\procB}{F}

\newtheorem{theo}{Theorem}[section]

\newtheorem{lemma}[theo]{Lemma}
\newtheorem{propo}[theo]{Proposition}
\newtheorem{corollary}[theo]{Corollary}

\theoremstyle{definition}
\newtheorem{ex}[theo]{Example}

\newtheorem{remark}[theo]{Remark}
\newtheorem{setting}[theo]{Setting}
\newtheorem{assumption}[theo]{Assumption}

\title{Multi-asset optimal trade execution with stochastic cross-effects: An Obizhaeva--Wang-type framework}

\author{Julia Ackermann\thanks{Department of Mathematics \& Informatics, University of Wuppertal, Gaußstr.~20, 42119 Wuppertal, Germany.
		\emph{Email:} jackermann@uni-wuppertal.de}
	\and Thomas Kruse\thanks{Department of Mathematics \& Informatics, University of Wuppertal, Gaußstr.~20, 42119 Wuppertal, Germany.
		\emph{Email:} tkruse@uni-wuppertal.de}
	\and Mikhail Urusov\thanks{Faculty of Mathematics, University of Duisburg-Essen, Thea-Leymann-Str.~9, 45127 Essen, Germany.
		\emph{Email:} mikhail.urusov@uni-due.de
}}

\begin{document}

\maketitle

\begin{abstract}
	We analyze a continuous-time optimal trade execution problem in multiple assets
	where the price impact and the resilience can be matrix-valued stochastic processes that incorporate cross-impact effects. 
	In addition, we allow for stochastic terminal and running targets. 
	Initially, we formulate the optimal trade execution task as a stochastic control problem with a finite-variation control process that acts as an integrator both in the state dynamics and in the cost functional. 
	We then extend this problem continuously to a stochastic control problem with progressively measurable controls. 
	By identifying this extended problem as equivalent to a certain linear-quadratic stochastic control problem, we can 
	use established results in linear-quadratic stochastic control to solve the extended problem.  
	This work generalizes [Ackermann, Kruse, Urusov; FinancStoch'24] 
	from the single-asset setting to the multi-asset case. 
	In particular, we reveal cross-hedging effects, showing that it can be optimal to trade in an asset despite having no initial position. 
	Moreover, as a subsetting we discuss 
	a multi-asset variant of the model in [Obizhaeva, Wang; JFinancMark'13].

	\bigskip
	
	
	\smallskip
	

	\smallskip
	
\end{abstract}

\tableofcontents

\section{Introduction}

Cross-impact describes the phenomenon that trading in a financial asset not only affects its own price but also the prices of other assets. 
The implications of this 
effect have drawn increasing attention in the literature on price manipulation and optimal trade execution in multi-asset models.  
Studies analyzing 
cross-impact in the context of 
multidimensional optimal trade execution 
include, for example, 
Alfonsi et al.~\cite{AlfonsiKloeckSchied2016}, 
Horst \& Xia \cite{horst2019multi},
Tsoukalas et al.~\cite{TsoukalasGieseckeWang2019},
Abi Jaber et al.~\cite{jaber2024optimal},
Bertsimas \& Lo \cite[Section~5]{bertsimas1998optimal},
Bertsimas et al.~\cite{bertsimaslohummel1999},
Almgren \& Chriss \cite[Appendix~A]{almgren2001optimal},
and 
Ma et al.\ \cite{ma2020optimal}.
Further works that investigate trading in cross-impact models include, for instance, 
Bilarev \cite[Chapter~5]{bilarev2018feedback}, 
Schneider \& Lillo \cite{SchneiderLillo2019}, 
G\^arleanu \& Pedersen \cite{garleanu2016dynamic}, 
Huberman \& Stanzl \cite[Section~5]{huberman2004price}, 
Muhle-Karbe \& Tracy \cite{muhlekarbe2024stochastic}, 
and Hey et al.\ \cite{hey2024concave}.

At the same time, a growing body of research analyzes the role of stochastic liquidity in single-asset optimal trade execution problems. Contributions in this area include Almgren \cite{almgren2012optimal},
Schied \cite{schied2013control}, 
Ankirchner et al.\ \cite{ankirchner2014bsdes},
Cheridito \& Sepin \cite{cheridito2014optimal},
Ankirchner \& Kruse \cite{ankirchner2015optimal},
Graewe et al.\ \cite{graewe2015non},
Kruse \& Popier \cite{kruse2016minimal},
Horst et al.\ \cite{horst2016constrained},
Graewe \& Horst \cite{graewe2017optimal},
Bank \& Voß \cite{bank2018linear},
Graewe et al.\ \cite{graewe2018smooth},
Popier \& Zhou \cite{popier2019second}, and
Ankirchner et al.\ \cite{ankirchner2020optimal}. 
These works extend the market impact models of Almgren \& Chriss \cite{almgren2001optimal} and Bertsimas \& Lo \cite{bertsimas1998optimal} 
by incorporating stochastic liquidity parameters. 
Moreover, Ackermann et al. \cite{ackermann2020optimal,ackermann2020cadlag,ackermann2021negativeresilience,ackermann2022reducing} 
analyze the effects of randomly evolving order book depth and resilience 
in extended variants of the models of 
Alfonsi \& Acevedo \cite{alfonsi2014optimal}, 
Bank \& Fruth \cite{bank2014optimal}, 
Fruth et al.\ \cite{fruth2014optimal,fruth2019optimal}, 
Siu et al.\ \cite{siu2019optimal},  
Alfonsi et al.\ \cite{alfonsi2008constrained}, and 
Obizhaeva \& Wang \cite{obizhaeva2013optimal}. 
Furthermore, Horst \& Kivman \cite{HorstKivman2024} 
allow for stochastic resilience while examining the optimal portfolio process for small instantaneous price impact factors.

A paper that analyzes the combined effects of cross-impact and stochastic liquidity characteristics in multi-asset optimal trade execution problems is Horst \& Xia \cite{horst2019multi}. The authors address optimal trade execution with both instantaneous and persistent price impact, thereby incorporating cross-impact, as well as stochastic
resilience, using backward stochastic Riccati differential equations. They restrict attention to deterministic price impacts, allowing only the resilience to evolve stochastically.

Furthermore, Ma et al.\ \cite{ma2020optimal} set up a price impact model in the spirit of Almgren \& Chriss \cite{almgren2001optimal} for trading in multiple assets that incorporates stochastic permanent and instantaneous price impact via a finite-state Markov process while omitting resilience. The authors solve an optimal trade execution problem by analyzing coupled Riccati differential equations.

Moreover, Muhle-Karbe \& Tracy \cite{muhlekarbe2024stochastic}
analyze price manipulation and optimization of the expected profits and losses
in a linear stochastic liquidity model with external order flow. The resilience is a 
constant matrix, whereas the price impact is a matrix-valued stochastic process. 
The authors show that such a model arises as the scaling limit of a discrete-time cross-impact model with deterministic concave price impact and semimartingale order flow.

The objective of the present paper is to investigate 
optimal trade execution in a multi-asset model 
where both cross-impact and resilience vary stochastically.
To achieve this, we extend the model of Ackermann et al.\ \cite{ackermann2022reducing} (except for the fact that we set $\eta=0$ in the notation of \cite{ackermann2022reducing}) to a multi-asset setting.

To formalize the model we consider a time horizon $T\in (0,\infty)$, a number of assets $n\in\N$, and a filtered probability space $(\Omega,\cF_T,(\cF_t)_{t\in[0,T]}, P)$ that satisfies the usual conditions. 
We consider an agent (typically a large institutional trader) whose trading activities are described by an $\R^n$-valued stochastic process~$X$, which tracks the positions in the $n$ assets. More specifically, $X_j(s)$ is the position in the $j$-th asset at the time $s \in [0,T]$ ($X_j(s)>0$ means a long position of $X_j(s)$ shares and $X_j(s)<0$ a short position of $-X_j(s)$ shares in the $j$-th asset at time $s\in [0,T]$). Immediately prior to the initial time $0$ the agent has the initial position $x\in\R^n$.
At the terminal time~$T$ the agent needs to reach a certain target position which is given by an $\R^n$-valued, $\cF_T$-measurable random variable~$\xi$. These two requirements are reflected by the boundary conditions $X(0-)=x$ and $X(T)=\xi$.

As a starting point we consider the finite-variation problem in which the agent chooses $X$ from the set of admissible strategies called 
$\cA^{fv}$, which is the set of $\R^n$-valued c\`adl\`ag finite-variation processes $X=(X(s))_{s\in[0,T]}$ that satisfy $X(0-)=x$, $X(T)=\xi$, and certain integrability conditions (see \cref{sec:fv_problem} for details). 
We assume an additive price impact model where the underlying unaffected price processes of the assets are suitable martingales, and we thus  
focus on the price deviation 
that stems from the trading activities of the agent which we model by
\begin{equation}\label{eq:deviationdynmultivariate}
	dD^X(s) = - \rho(s) D^X(s) ds + \gamma(s) dX(s), \quad s \in [0,T], \quad 
	D^X(0-)=d.
\end{equation}
The initial deviation immediately prior to the time~$0$ is given by $D^X(0-)=d \in \R^n$ (usually $d=0\in\R^n$). Note that according to \eqref{eq:deviationdynmultivariate} a trade in some asset, that is, a change of the position~$X$, has a linear influence on~$D^X$ with the matrix-valued price impact process~$\gamma$ as a factor.
We emphasize that we do not assume $\gamma$ to be diagonal\footnote{see ~\eqref{eq:dyn_lambda} and~\eqref{eq:def_gamma} for the definition of $\gamma$} implying 
that the deviation in a certain asset can also be affected by the trades in the other assets. 

To further discuss the price deviation, 
observe that the term $-\rho(s) D^X(s) ds$, 
involving the stochastic matrix-valued resilience process~$\rho$,
governs the evolution of $D^X$ when the agent is not trading.
If $\rho$ is  a.s.\ symmetric and uniformly positive definite in time,
this term induces an exponential decay of the price deviation to zero.
This reflects that, following a trade,
new orders gradually replenish the order book -- a phenomenon known as the resilience effect. We emphasize, however, that our results do not require $\rho$ to be  a.s.\ symmetric and uniformly positive definite in time. Instead, we impose only weaker assumptions (see, e.g., \cref{assump_convexity} and \cref{rem:sufficient_cond_convex}). In particular, our framework also accommodates models where the price
impact exhibits self-exciting behavior. For a discussion of qualitative effects of such negative resilience in the single-asset situation, we refer to \cite{ackermann2021negativeresilience}. 
Furthermore, note that $\rho$ is not necessarily diagonal.

In addition to reaching the required terminal position $\xi$, the aim of the agent is to incur minimal costs associated with her trading activities.
Let us consider 
the costs of a block trade at the time $s\in[0,T]$ when trading according to a strategy $X \in \cA^{fv}$ with the associated deviation $D^X$.  
Observe that immediately prior to $s$ we have the deviation $D^X(s-)$. By \eqref{eq:deviationdynmultivariate} the block trade $\Delta X(s)$ shifts the deviation $D^X(s-)$ to 
$$D^X(s)=D^X(s-)+\Delta D^X(s) = D^X(s-)+ \gamma(s)\Delta X(s).$$ 
We next take the mid-prices and multiply the mid-prices by the amount of the traded shares. 
The block trade $\Delta X(s)$ is thus assigned the costs  
\begin{equation*}
	\begin{split}
		\big(D^X(s-)+\tfrac12 \gamma(s) \Delta X(s)\big)^\top \Delta X(s)
		& = D^X(s-)\Delta X(s-) + \tfrac12 (\Delta X(s))^\top \gamma(s) \Delta X(s) .
	\end{split}
\end{equation*}
This explanation for the costs of a single block trade motivates the definition of the pathwise costs $C(X)$ over the whole trading interval $[0,T]$ by 
\begin{equation}\label{eq:def_pathwisecosts}
	C(X) = 
	\int_{[0,T]} (D^X(s-))^\top\, dX(s) + \frac12 \int_{[0,T]} (\Delta X(s))^\top \gamma(s)\, dX(s) .
\end{equation}
Then $E[C(X)]$ describes the expected costs due to illiquidity when trading according to the finite-variation execution strategy~$X$.  
This is one of the two summands in the definition of the costs 
\begin{equation}\label{eq:cost_fct_intro}
	\begin{split}
		\fvJ(X) 
		& = E[C(X)] + E\bigg[\int_0^T (X(s)-\zeta(s))^\top  \risk(s) (X(s)-\zeta(s)) ds \bigg]  
	\end{split}
\end{equation}
that are to be minimized. 
The other summand in these costs  
can be used to incorporate some kind of risk preference into the model via the choice of the matrix-valued, not necessarily diagonal, process
$\risk$, which acts as a penalization, and the $\R^n$-valued process~$\zeta$, which acts as a running target. 
We refer to Horst \& Xia \cite[Section~1.1]{horst2019multi} for an illustrative example of a possible choice for 
$\risk$.

To summarize, the stochastic control problem to minimize the costs~\eqref{eq:cost_fct_intro} over all strategies $X \in \cA^{fv}$ 
models the agent's task to reach the terminal position $\xi$ from the initial position $x$ while incurring minimal costs. The framework can incorporate cross-effects between the $n$ assets through the price impact~$\gamma$, the resilience~$\rho$, and the risk preference~$\risk$, 
any of which can produce cross-hedging effects in the optimal strategy (see, e.g., the examples in \cref{sec:examples} and \cref{sec:appendixcrossresilience,sec:example_risk,sec:example_gamma}).

Since optimal strategies in the class $\cA^{fv}$ do not always exist (see also \cref{rem:pmvsfv}), we continuously extend the stochastic control problem from finite-variation
strategies to progressively measurable strategies, using similar techniques as in \cite{ackermann2022reducing}. 
We refer to \cref{sec:pm_problem} for the formulation of the extended problem, which is based on the representation obtained in \cref{propo:rewritten_costs_and_deviation} and \cref{lem:rewritten_costs3} for the finite-variation price deviation and the finite-variation costs. 
The main result with regard to the extension is \cref{thm:contextcostfct}, where we show (i) that admissible progressively measurable strategies can be approximated, in a suitable metric (see~\eqref{eq:def_metric}), by a sequence of admissible finite-variation strategies and (ii) that the costs converge. 

An aspect important for establishing \cref{thm:contextcostfct} and for solving the extended problem is the insight that the extended problem is equivalent to a ``standard'' linear-quadratic stochastic control problem (see \cref{sec:equivLQ}) with a bijection (see~\eqref{eq:def_phi_b} and \cref{lem:bijective}) between the set of admissible progressively measurable trade execution strategies and the set of progressively measurable square-integrable controls. 
In the single-asset case, this has been already observed in \cite{ackermann2022reducing}. 
To solve the extended problem, we first solve the ``standard'' linear-quadratic stochastic control problem by using results from Sun et al.\ \cite{sun2021indefiniteLQ}. 
In \cref{cor:soln_opt_trade_execution} we then conclude that, under appropriate assumptions, there exists a unique optimal progressively measurable trade execution strategy for the extended problem. 
We moreover provide a feedback-type formula for the optimal strategy in terms of the solution of a Riccati-type backward stochastic differential equation.

Moreover, we analyze as a subsetting of our model a multi-asset variant of the model by Obizhaeva \& Wang~\cite{obizhaeva2013optimal} 
(see \cref{sec:examples}, \cref{sec:multiasset_OW}, and \cref{sec:appendixcrossresilience}). 
In the literature, there are further multi-asset models that can be seen as generalizing Obizhaeva \& Wang~\cite{obizhaeva2013optimal} to multiple assets. 
For instance, optimal trade execution in such kinds of models has been investigated in 
Tsoukalas et al.~\cite{TsoukalasGieseckeWang2019}, 
Alfonsi et al.~\cite{AlfonsiKloeckSchied2016},
and Abi Jaber et al.~\cite{jaber2024optimal}. 
Note that, different from our set-up, Tsoukalas et al.~\cite{TsoukalasGieseckeWang2019} consider a discrete-time framework and describe the transient price impact and the resilience by diagonal matrices. 
Cross-impact in their model is introduced via the permanent price impact and correlations in the fundamental prices, which affect the model via risk aversion. 

Alfonsi et al.~\cite{AlfonsiKloeckSchied2016} analyze price manipulation and optimal trade execution in a multi-asset propagator model with deterministic decay kernels.
Whereas 
Alfonsi et al.~\cite{AlfonsiKloeckSchied2016} 
restrict their study to decay kernels of convolution type (and thus constant price impact), Abi Jaber et al.~\cite{jaber2024optimal} allow for more general, yet still deterministic, decay kernels.
Cross-impact in \cite{jaber2024optimal} is implemented via the decay kernel, the instantaneous price impact matrix, and correlations in the semimartingale fundamental prices. 
The authors study multi-asset portfolio choice and in this framework 
consider optimal liquidation by penalizing large terminal inventories (see \cite[Remark~2.10]{jaber2024optimal}), but do not require strict liquidation, in contrast to Tsoukalas et al.~\cite{TsoukalasGieseckeWang2019}, 
Alfonsi et al.~\cite{AlfonsiKloeckSchied2016}, and our work. 

Note that all of the works Tsoukalas et al.~\cite{TsoukalasGieseckeWang2019}, 
Alfonsi et al.~\cite{AlfonsiKloeckSchied2016},
and Abi Jaber et al.~\cite{jaber2024optimal} model liquidity and cross-impact effects by deterministic quantities within nonlinear models, whereas we allow for cross-impact from stochastic resilience and stochastic price impact in a linear-quadratic framework. 
Moreover, we incorporate stochastic terminal and running targets in our optimal trade execution problem.

\section{The multi-asset trade execution problem for finite-variation strategies}
\label{sec:multi_asset_fv}
\label{sec:notation_and_setting}

Throughout the whole article, we let $T\in(0,\infty)$ be the time horizon, $n\in\N$ the number of assets, $x \in \R^n$ the initial position, $d\in\R^n$ the initial price deviation, and $m\in\N$.  
We fix a filtered probability space $(\Omega,\cF_T,(\cF_s)_{s\in[0,T]}, P)$ that satisfies the usual conditions and supports an $m$-dimensional Brownian motion $W=(W_1,\ldots,W_m)^\top$.

\subsection{Notation} 
\label{sec:notation}

In this subsection we introduce some notations 
that we use in this work. 

The set of symmetric matrices in $\R^{n\times n}$ is denoted by $\cS^n$, 
the subset of positive semidefinite $(n\times n)$-matrices is denoted by $\cS^n_{\geq 0}$, and the subset of positive definite $(n\times n)$-matrices is denoted by $\cS^n_{> 0}$.   
The identity matrix in $\R^{n\times n}$ is denoted by $I_{n}$. 
For $n',m'\in\N$ and $A=(A_{i,j})_{(i,j)\in\{1,\ldots,n'\}\times \{1,\ldots,m'\}}\in\R^{n'\times m'}$ let $\lVert A\rVert_F = (\sum_{i=1}^{n'} \sum_{j=1}^{m'} \lvert A_{i,j} \rvert^2)^{\frac12}$ be the Frobenius norm. 

We introduce the following notations for $n',m',l\in\N$. 
For $s\in[0,T]$ and an $\R^{n'\times m'}$-valued c\`adl\`ag semimartingale $L$ we define $\Delta L(s) = L(s) - L(s-)$. 
We follow the convention that, 
for $t\in[0,T]$, $r \in [t,T]$, 
and a c\`adl\`ag semimartingale $L$, jumps of the c\`adl\`ag integrator $L$ at time $t$ contribute to integrals of the form $\int_{[t,r]} \ldots dL(s) \ldots$. 
In contrast, we write $\int_{(t,r]} \ldots dL(s) \ldots$ when we do not include jumps of $L$ at time $t$ into the integral. 
For suitable c\`adl\`ag semimartingales $A$, $B$, where $A$ is $\R^{n'\times m'}$-valued and $B$ is $\R^{m' \times l}$-valued, we denote by $[A,B]$ 
the $\R^{n'\times l}$-valued covariation. 

We define 
\begin{equation*}
	\begin{split}
		& \cL^2 = \bigg\{u \,\bigg\lvert\, u\colon [0,T]\times \Omega \to\R^n \text{ progr.\ measurable and }  \lVert u \rVert_{\cL^2}^2 = E\bigg[\int_0^T \lVert u(s) \rVert_F^2 ds\bigg] < \infty \bigg\} 
	\end{split}
\end{equation*}
and 
$L^2(\Omega,\cF_T,P;\R^n) =\big\{Y \,\big\lvert\, Y\colon\Omega\to\R^n\,\, \cF_T\text{-measurable and } \lVert Y \rVert_{L^2}^2 = E[\lVert Y\rVert_F^2]<\infty  \big\} .$

\subsection{The stochastic liquidity processes}\label{sec:setting}

In the following we introduce the stochastic resilience and price impact processes and some related quantities that we need to set up and solve the optimal trade execution problem. 

Let $\rho=(\rho(s))_{s\in[0,T]}$ be an $\R^{n\times n}$-valued progressively measurable process which is $dP\times ds$-a.e.\ bounded. We call $\rho$ the \textit{resilience process}. 
Let $\nu=(\nu(s))_{s\in[0,T]}$ be the unique solution of 
\begin{equation}\label{eq:defnu}
	d\nu(s) = \nu(s) \rho(s) ds, \quad s \in [0,T], \quad \nu(0)=I_n, 
\end{equation}
and observe that the inverse $\nu^{-1}=(\nu^{-1}(s))_{s\in[0,T]}$ is 
the unique solution of 
\begin{equation}\label{eq:eqnuinv}
	d\nu^{-1}(s) = - \rho(s) \nu^{-1}(s) ds, \quad s\in[0,T], \quad \nu^{-1}(0) =I_n. 
\end{equation}

Next, we construct the price impact process. To this end let $\mu=(\mu(s))_{s\in[0,T]}$ be an $\R^n$-valued progressively measurable process and let $\sigma=(\sigma(s))_{s\in[0,T]}$ be an $\R^{n\times m}$-valued progressively measurable process.  
Assume that $\mu$ and $\sigma$ are $dP\times ds$-a.e.\ bounded. 
For each $j\in\{1,\ldots,n\}$ let 
$\lambda_j=(\lambda_j(s))_{s\in[0,T]}$ 
be the $(0,\infty)$-valued stochastic process given as the unique solution of 
\begin{equation}\label{eq:dyn_lambda}
	d\lambda_j(s) = \lambda_j(s) \mu_{j}(s) ds + \sum_{k=1}^m  \lambda_j(s) \sigma_{j,k}(s) dW_k(s) , \quad s\in[0,T], \quad \lambda_j(0)\in(0,\infty). 
\end{equation} 
For every $\alpha \in\R$ let $\lam^{\alpha}=(\lam^{\alpha}(s))_{s \in [0,T]}$ be the $\cS_{>0}^n$-valued stochastic process given by 
\begin{equation}\label{eq:def_lam}
	\lam^{\alpha}(s) = \diag((\lambda_1(s))^\alpha,\ldots,(\lambda_n(s))^\alpha), 
	\quad s\in[0,T]. 
\end{equation}
Furthermore, let $\diagmu=(\diagmu(s))_{s\in[0,T]}$ 
be defined by 
\begin{equation*}
	\diagmu(s) = \diag(\mu_1(s),\ldots,\mu_n(s)), 
	\quad s \in [0,T], 
\end{equation*}
and for all $k\in\{1,\ldots,m\}$ let 
$\diagsigma_k=(\diagsigma_{k}(s))_{s\in[0,T]}$ 
be defined by 
\begin{equation*}
	\diagsigma_{k}(s) = \diag(\sigma_{1,k}(s),\ldots,\sigma_{n,k}(s)), 
	\quad s\in [0,T] .
\end{equation*}
Let $\OO \in\R^{n\times n}$ be an orthogonal matrix 
(that is, $\OO^\top\OO=I_n=\OO\OO^\top$) 
and let $\gamma=(\gamma(s))_{s\in[0,T]}$
be the $\cS_{>0}^n$-valued continuous semimartingale 
given by 
\begin{equation}\label{eq:def_gamma}
	\gamma(s) = \OO^\top \lam(s) \OO, \quad s\in[0,T]. 
\end{equation}
We call $\gamma$ the \textit{price impact process}. 
For all $\alpha \in\R$ we denote by $\gamma^{\alpha} = (\gamma^{\alpha}(s))_{s\in[0,T]}$
the $\cS_{>0}^n$-valued continuous semimartingale 
given by 
\begin{equation}\label{eq:def_gamma_alpha}
	\gamma^{\alpha}(s) = \OO^\top \lam^{\alpha}(s) \OO, \quad s\in[0,T].
\end{equation}
We assume that 
$\gamma^{-\frac12} \rho \gamma^{\frac12}$ 
is $dP\times ds$-a.e.\ bounded.

\begin{remark}\label{rem:remark_on_setting_assumptions1}
	(i)
	The assumption that $\gamma$ is symmetric is necessary for the well-posedness of our trade execution problem. 
	Indeed, in 
	\cref{ex:gamma_sym} in the appendix  
    we consider a situation with a non-symmetric price impact~$\gamma$ and show that the execution costs can then be 
	pushed towards $-\infty$. 
	Note that for every $\cS^n$-valued process $\wt \gamma$ and for all $\omega\in\Omega$, $s\in[0,T]$ we can find an orthogonal matrix $\wt\OO(s,\omega)\in\R^{n\times n}$ and $\wt\lambda_j(s,\omega)\in \R$, $j\in\{1,\ldots,n\}$, such that $\wt \gamma(s,\omega) = \wt\OO(s,\omega)^\top \diag(\wt\lambda_1(s,\omega),\ldots,\wt\lambda_n(s,\omega)) \wt\OO(s,\omega)$. 
	In our above definition of~$\gamma$ we assume that $\wt \OO = \OO$ is deterministic and constant and that the eigenvalues $\wt\lambda_j=\lambda_j$, $j\in\{1,\ldots,n\}$, of $\gamma$ are strictly positive and follow 
	suitable, possibly stochastic, dynamics 
	(see~\eqref{eq:dyn_lambda}).
	
	(ii)
	In \cite{ackermann2022reducing}, the assumptions on the resilience and the price impact consist of the following: 
	1.\ the resilience is a $dP\times ds$-a.e.\ bounded, progressively measurable process and 
	2.~the price impact is defined by \eqref{eq:dyn_lambda} with $dP\times ds$-a.e.\ bounded, progressively measurable coefficient processes and $m=1$.
	Note that in the one-dimensional case, boundedness of  $\gamma^{-\frac12} \rho \gamma^{\frac12}$ is equivalent to boundedness of $\rho$.
	Hence, the 
	present setting generalizes the one in \cite{ackermann2022reducing} to multiple assets.
	
	(iii)
	A sufficient\footnote{Note that we do not have equivalence. For example, consider $\gamma$ and $\rho$ to be deterministic and time-independent, then clearly $\gamma^{-\frac12} \rho \gamma^{\frac12}$ 
	is $dP\times ds$-a.e.\ bounded, but $\gamma$ and $\rho$ do not need to commute.} 
	condition for 
	$\gamma^{-\frac12} \rho \gamma^{\frac12}$ 
	to be $dP\times ds$-a.e.\ bounded is, for example, that $\gamma$ and $\rho$ commute. 
	To wit, if $\gamma\rho=\rho\gamma$ $dP\times ds$-a.e., 
	then it also holds that
	$\gamma^{\frac12} \rho = \rho\gamma^{\frac12}$ $dP\times ds$-a.e.\ (see, for example, \cite[Theorem~7.2.6]{HornJohnson2013MatrixAnalysis}), and hence 
	$\gamma^{-\frac12} \rho \gamma^{\frac12} = \rho$ $dP\times ds$-a.e. 
	Since $\rho$ is assumed to be $dP\times ds$-a.e.\ bounded, we thus obtain in this situation that $\gamma^{-\frac12} \rho \gamma^{\frac12}$ 
	is $dP\times ds$-a.e.\ bounded. 
\end{remark}

Moreover, we want to include a risk term into the cost functional. 
To this end, let $\risk=(\risk(s))_{s\in[0,T]}$ be an
$\cS^n$-valued progressively measurable process such that
the $\cS^n$-valued progressively measurable process $\mathscr{Q}=(\mathscr{Q}(s))_{s\in[0,T]}$ defined by  
\begin{equation}\label{eq:def_Q}
	\mathscr{Q}(s) 
	= \gamma^{-\frac12}(s) \risk(s) \gamma^{-\frac12}(s), \quad s\in[0,T],
\end{equation}
is $dP\times ds$-a.e.\ bounded.

To conclude this subsection, we introduce auxiliary quantities essential for the subsequent analysis. 
For every $k\in\{1,\ldots,m\}$ let 
the $\cS^n$-valued progressively measurable process  $\mathscr{C}^k=(\mathscr{C}^k(s))_{s\in[0,T]}$ be defined 
by 
\begin{equation}\label{eq:def_Ck}
	\mathscr{C}^k(s)
	= \frac12 \OO^\top \diagsigma_{k}(s) \OO, 
	\quad s\in[0,T].
\end{equation} 
Furthermore, let 
the $\cS^n$-valued progressively measurable process  $\mathscr{A}=(\mathscr{A}(s))_{s\in[0,T]}$ and 
the progressively measurable process  $\mathscr{B}=(\mathscr{B}(s))_{s\in[0,T]}$ be defined for all $s \in [0,T]$ by 
\begin{equation}\label{eq:def_A_B}
	\begin{split}
		& \mathscr{A}(s)
		= \frac12 \OO^\top \diagmu(s) \OO  
		- \frac12 \sum_{k=1}^m \mathscr{C}^k(s)\mathscr{C}^k(s), \\
		& \mathscr{B}(s)
		= -\gamma^{-\frac12}(s) \rho(s) \gamma^{\frac12}(s) 
		- \OO^\top \diagmu(s) \OO 
		+ 2 \sum_{k=1}^m \mathscr{C}^k(s)\mathscr{C}^k(s)
		.
	\end{split}
\end{equation}
Moreover, we introduce the $\cS^n$-valued progressively measurable process $\kap=(\kap(s))_{s\in[0,T]}$ defined for all $s\in[0,T]$ by 
\begin{equation}\label{eq:def_kap}
	\kap(s) = 
	\frac12 \OO^\top \diagmu(s) \OO - 2 \sum_{k=1}^m \mathscr{C}^k(s)\mathscr{C}^k(s) 
	+ \frac12 \gamma^{-\frac12}(s) \rho(s) \gamma^{\frac12}(s) 
	+ \frac12  \gamma^{\frac12}(s) (\rho(s))^\top \gamma^{-\frac12}(s) 
\end{equation}
and the $\cS^n$-valued progressively measurable process $\KK=(\KK(s))_{s\in[0,T]}$ defined by 
\begin{equation}\label{eq:def_KK}
	\KK = \mathscr{Q} + \kap .
\end{equation}

\begin{remark}\label{rem:remark_on_setting_assumptions2}
	(i)
	Note that the fact that $\sigma$ is $dP\times ds$-a.e.\ bounded implies for all $k\in\{1,\ldots,m\}$ that $\mathscr{C}^k$ is $dP\times ds$-a.e.\ bounded. 
	This and the fact that $\mu$ and $\gamma^{-\frac12} \rho \gamma^{\frac12}$ are $dP\times ds$-a.e.\ bounded yield that $\kap$, $\mathscr{A}$, and $\mathscr{B}$ are $dP\times ds$-a.e.\ bounded.
	Moreover, the fact that $\kap$ and $\mathscr{Q}$ are $dP\times ds$-a.e.\ bounded ensures that $\KK$ is $dP\times ds$-a.e.\ bounded.
	
	(ii) 
    If $\rho$ and $\gamma$ commute, then also $\rho^\top$ and $\gamma$ commute, hence
    $\gamma^{-\frac12}\rho\gamma^{\frac12}=\rho$ and
    $\gamma^{\frac12}\rho^\top\gamma^{-\frac12}=\rho^\top$,
    and we thus have 
	$\mathscr{B}
	= -\rho - \OO^\top \diagmu \OO 
	+ 2 \sum_{k=1}^m \mathscr{C}^k\mathscr{C}^k$
	and $\kap 
	= \frac12 (\rho + \rho^\top)
	+ \frac12 \OO^\top \diagmu \OO - 2 \sum_{k=1}^m \mathscr{C}^k\mathscr{C}^k $.
\end{remark}

\subsection{The finite-variation stochastic control problem}\label{sec:fv_problem}

We first introduce the finite-variation strategies. 
To this end we first associate a state process to potential finite-variation strategies. 
For 
an $\R^{n}$-valued c\`adl\`ag finite-variation process $X=(X(s))_{s\in[0,T]}$ with $X(0-)=x$ we define the $\R^n$-valued c\`adl\`ag finite-variation process $D^X=(D^X(s))_{s\in[0,T]}$ by~\eqref{eq:deviationdynmultivariate}.

\begin{remark}\label{rem:soln_for_deviation_SDE}
    Suppose that $X=(X(s))_{s\in[0,T]}$ is an $\R^{n}$-valued c\`adl\`ag finite-variation process with $X(0-)=x$. 
	It then holds that there exists a unique solution 
	$D^X=(D^X(s))_{s\in[0,T]}$ of \eqref{eq:deviationdynmultivariate}, and it is given by 
	\begin{equation*}
	D^X(r) = \nu^{-1}(r) \left( d + \int_{[0,r]} \nu(s) \gamma(s) dX(s) \right),
	\quad r \in [0,T] 
	\end{equation*}
	(cf. \eqref{eq:defnu}, \eqref{eq:eqnuinv}, and, for example, \cite{jacod1982equations}, see also \cite[Theorem 1.2]{duan2008general}).
\end{remark}

The terminal target $\xi$ is an $\R^n$-valued, $\cF_T$-measurable random variable satisfying 
\begin{equation}\label{eq:int_cond_terminal_pos}
	E\big[ \lVert \gamma^{\frac12}(T) \xi \rVert_F^2 \big] < \infty .
\end{equation}
We denote by $\cA^{fv}$ the set of all adapted c\`adl\`ag finite-variation processes $X=(X(s))_{s\in[0,T]}$ that satisfy $X(0-)=x$, $X(T)=\xi$, and 
the integrability conditions 
\begin{align}
& E\bigg[ \int_0^T \big\lVert \gamma^{-\frac12}(s) D^X(s) \big\rVert_F^2 ds \bigg] < \infty,  \label{eq:A1}\\
& \sum_{k=1}^m E\bigg[ \bigg( \int_0^T \big\lvert \big( \gamma^{-\frac12}(s) D^X(s)\big)^\top \OO^\top \diagsigma_{k}(s) \OO \gamma^{-\frac12}(s) D^X(s) \big\rvert^2 ds \bigg)^{\frac12} \bigg] < \infty. \label{eq:A2}
\end{align}
Any element $X\in\cA^{fv}$ is called a \emph{finite-variation (execution) strategy}, and 
the process $D^X$ defined by \eqref{eq:deviationdynmultivariate} is called the associated \emph{deviation (process)}.

We now set up the cost functional. 
In the first step we define 
for all 
$\R^{n}$-valued c\`adl\`ag finite-variation processes $X=(X(s))_{s\in[0,T]}$ with $X(0-)=x$ and with associated process $D^X=(D^X(s))_{s\in[0,T]}$ given  by~\eqref{eq:deviationdynmultivariate} the pathwise costs by \eqref{eq:def_pathwisecosts}.
Note that for all 
$X\in\cA^{fv}$ the 
expectation $E[C(X)]$ is well defined (cf.~\cref{lem:rewritten_costs3}).

We next introduce a running position target. To this end let $\zeta=(\zeta(s))_{s\in[0,T]}$ be an $\R^n$-valued progressively measurable process such that 
\begin{equation}\label{eq:int_cond_zeta}
	\lVert \gamma^{\frac12}\zeta \rVert_{\cL^2}<\infty .
\end{equation}
We remark that 
$$E\bigg[ \Big\lvert \int_0^T (X(s)-\zeta(s))^\top  \risk(s) (X(s)-\zeta(s)) ds \Big\rvert \bigg]$$
is finite for all 
$X\in\cA^{fv}$
(cf.~\cref{lem:risk_term_finite_expectation}). 

Finally, we introduce 
for all 
$X \in \cA^{fv}$ 
the cost functional~$\fvJ$ given by \eqref{eq:cost_fct_intro}, that is, 
\begin{equation}\label{eq:defcostfctmultivariate}
	\begin{split}
		\fvJ(X) 
		& = E\bigg[ \int_{[0,T]} (D^X(s-))^\top\, dX(s) + \frac12 \int_{[0,T]} (\Delta X(s))^\top \gamma(s)\, dX(s) \bigg] \\
        & \quad + E\bigg[\int_0^T (X(s)-\zeta(s))^\top  \risk(s) (X(s)-\zeta(s)) ds \bigg] . 
	\end{split}
\end{equation}
The finite-variation stochastic control problem is to minimize $\fvJ$ over
$\cA^{fv}$. 

\begin{remark}
    Similar to \cite{ackermann2022reducing}, we can consider general initial times $t \in[0,T]$ and replace the expected values 
    by conditional expectations $E[\cdot \vert \cF_t]$. 
    All results in \cref{sec:altrepres} and \cref{sec:cont_ext_to_pm} as well as \cref{cor:linkLQpm} remain valid. 
\end{remark}

In \cref{sec:appendixontheprice} we include some comments and examples regarding our set-up. 
In particular, we show that for non-symmetric price impact the trade execution problem can become ill-posed in the sense that we find a sequence of trading strategies that pushes the execution costs to $-\infty$ (see \cref{ex:gamma_sym}). 
Furthermore, we illustrate in a simple example the role of the matrix-valued resilience process $\rho$ in our model and discuss the effects of non-zero off-diagonal entries of $\rho$ on the price deviation (see \cref{ex:resilience_effect}).

\subsection{Alternative representations for the costs and for the deviation}
\label{sec:altrepres}

In \cref{propo:rewritten_costs_and_deviation} we provide alternative pathwise representations of the costs~\eqref{eq:def_pathwisecosts} and of the deviation~\eqref{eq:deviationdynmultivariate}. 
Note that the right-hand sides of these representations make sense also for strategies that are not necessarily of finite variation. 
The proof of \cref{propo:rewritten_costs_and_deviation}, along with all other proofs, is moved to \cref{sec:proofs}.

\begin{propo}\label{propo:rewritten_costs_and_deviation}
	Let $X=(X(s))_{s\in[0,T]}$ be an $\R^n$-valued c\`adl\`ag finite-variation process with $X(0-)=x$ and with associated process $D^X=(D^X(s))_{s\in[0,T]}$ defined by~\eqref{eq:deviationdynmultivariate}. 
	Then it holds that  
	\begin{equation}\label{eq:costfunctionalpart}
		\begin{split}
			C(X) 
			& = \frac12 (D^X(T))^\top \gamma^{-1}(T) D^X(T) 
			- \frac12 d^\top \gamma^{-1}(0) d \\
			& \quad - \frac12 \int_{0}^T (D^X(s))^\top (\nu(s))^\top d\left((\nu^{-1}(s))^\top \gamma^{-1}(s) \nu^{-1}(s)\right) \nu(s) D^X(s) 
		\end{split}
	\end{equation}
	and 
	\begin{equation}\label{eq:rewritten_deviation}
		D^X(r) = \gamma(r) X(r) + \nu^{-1}(r) \left( d - \gamma(0) x - \int_0^r \bigl(d(\nu(s)\gamma(s))\bigr) X(s) \right) , \quad r\in[0,T].
	\end{equation}
\end{propo}

Next we use the representation~\eqref{eq:costfunctionalpart} to establish 
the expression~\eqref{eq:fvJ_rewritten} for  
$E[C(X)]$. 
Recall that~$\kap$ is defined in~\eqref{eq:def_kap}.

\begin{propo}\label{lem:rewritten_costs3}
	Let $X\in\cA^{fv}$ with associated deviation process $D^X$ defined by~\eqref{eq:deviationdynmultivariate}. 
	Then $E[C(X)]$ is well defined and admits the representation 
	\begin{equation}\label{eq:fvJ_rewritten}
		\begin{split}
			E[C(X)]
			& = \frac12 E\big[ (D^X(T))^\top \gamma^{-1}(T) D^X(T) \big] 
			- \frac12 d^\top \gamma^{-1}(0) d 
			\\
			& \quad + E\bigg[ \int_0^T (D^X(s))^\top \gamma^{-\frac12}(s) 
			\kap(s) 
			\gamma^{-\frac12}(s) D^X(s)  ds \bigg] 
			.
		\end{split}
	\end{equation} 
\end{propo}

\section{Continuous extension to progressively measurable strategies}
\label{sec:cont_ext_to_pm}

The aim in this section is to extend the problem formulation of \cref{sec:fv_problem} to a broader class of admissible strategies. 
The alternative representations of the execution costs and of the deviation process 
that are derived in \cref{sec:altrepres} 
lead to 
our problem formulation for progressively measurable strategies in \cref{sec:pm_problem}.  
We show in the main result \cref{thm:contextcostfct} in \cref{sec:continuousextension} that 
the control problem in \cref{sec:pm_problem} can be interpreted as the continuous extension to progressively measurable strategies of the control problem for finite-variation strategies stated in \cref{sec:fv_problem}.
As helpful tools we introduce the stochastic differential equation (SDE)~\eqref{eq:def_scH_u} 
and the bijection~\eqref{eq:def_phi_b} 
(see moreover \cref{lem:bijective}) 
between our set of progressively measurable execution strategies and the set of square-integrable progressively measurable processes.

\subsection{Progressively measurable execution strategies}
\label{sec:pm_problem}

For an $\R^{n}$-valued progressively measurable process $X=(X(s))_{s\in[0,T]}$ 
which satisfies $\int_0^T \lVert X(s) \rVert_F^2 ds < \infty$ a.s.\ 
we define  
the $\R^n$-valued progressively measurable process $D^X=(D^X(s))_{s\in[0,T]}$ by
\begin{equation}\label{eq:def_deviation_pm}
D^X(r) = \gamma(r) X(r) + \nu^{-1}(r) \bigg( d - \gamma(0) x - \int_0^r \big(d(\nu(s)\gamma(s))\big) X(s) \bigg) , \quad r\in[0,T].
\end{equation} 
We denote by $\cA^{pm}$ the set of
(equivalence classes\footnote{For $X,Y\in\cA^{pm}$, we say that $X\sim Y$ if $X=Y$ $dP\times ds$-a.e.\ on $\Omega\times[0,T]$, $X(0-)=Y(0-) \; (=x)$ and $X(T)=Y(T) \; (=\xi)$.} of)
$\R^n$-valued progressively measurable processes $X=(X(s))_{s\in[0,T]}$ with $X(0-)=x$ and $X(T)=\xi$ that satisfy $\int_0^T \lVert X(s) \rVert_F^2 ds < \infty$ a.s.\ and~\eqref{eq:A1}.
Any element $X\in\cA^{pm}$ is called a \emph{progressively measurable (execution) strategy}, and 
the process $D^X$ now defined via~\eqref{eq:def_deviation_pm} is again called the associated \emph{deviation (process)} (note also \cref{cor:fvsubsetpm} and \cref{lem:deviationcompatibility} below). 

For 
$X \in \cA^{pm}$ with associated deviation $D^X$ defined by~\eqref{eq:def_deviation_pm} 
we define 
the cost functional $\pmJ$ by 
\begin{equation}\label{eq:def_costfct_pm}
	\begin{split}
		\pmJ(X)
		& = \frac12 E\big[ (D^X(T))^\top \gamma^{-1}(T) D^X(T) \big] 
        + E\bigg[ \int_0^T (X(s)-\zeta(s))^\top \risk(s) (X(s) -\zeta(s)) ds \bigg] \\
		& \quad  
		+ E\bigg[ \int_0^T (D^X(s))^\top \gamma^{-\frac12}(s)  
		\kap(s) 
		\gamma^{-\frac12}(s) D^X(s)  ds \bigg] 
		- \frac12 d^\top \gamma^{-1}(0) d .
	\end{split}
\end{equation}
Note that $\pmJ$ is well defined (cf.~\cref{lem:pm_cost_fct_as_LQ}). 
The following result is an immediate consequence of \Cref{propo:rewritten_costs_and_deviation} and \Cref{lem:rewritten_costs3}.

\begin{corollary}\label{cor:fvsubsetpm}
	Suppose that $X\in\cA^{fv}$ with associated deviation process $D^X$ defined by~\eqref{eq:deviationdynmultivariate}. 
	It then holds that $X \in \cA^{pm}$, that $D^X$ satisfies~\eqref{eq:def_deviation_pm}, and that $\fvJ(X)=\pmJ(X)$.  
\end{corollary}

Moreover, we have the following:

\begin{lemma}\label{lem:deviationcompatibility}
	Suppose that $X\in\cA^{pm}$ is a c\`adl\`ag finite-variation process with associated deviation process $D^X$ defined by~\eqref{eq:def_deviation_pm}. 
    Then $D^X$ satisfies~\eqref{eq:deviationdynmultivariate}.
\end{lemma}

\subsection{Continuous extension of the cost functional}
\label{sec:continuousextension}

We consider $\cL^2$, the set of (equivalence classes\footnote{For $u,v\in\cL^2$, we say that $u\sim v$ if $u=v$ $dP\times ds$-a.e.\ on $\Omega\times[0,T]$.} of) $\R^n$-valued progressively measurable processes $u=(u(s))_{s\in[0,T]}$ that satisfy 
$ \lVert u \rVert_{\cL^2}^2 = E[\int_0^T \lVert u(s) \rVert_F^2 ds] < \infty$, and for any 
$u\in\cL^2$ 
we introduce the 
controlled 
SDE 
\begin{equation}\label{eq:def_scH_u}
\begin{split}
d\scH^u(s) & = 
\mathscr{A}(s) \scH^u(s) ds  
+ \mathscr{B}(s) u(s) ds 
+ \sum_{k=1}^m \mathscr{C}^k(s) \scH^u(s) dW_k(s) \\
& \quad 
- 2 \sum_{k=1}^m \mathscr{C}^k(s) u(s) dW_k(s)  , 
 \quad s\in[0,T],\\
\scH^u(0) & = \gamma^{-\frac12}(0) d - \gamma^{\frac12}(0) x, 
\end{split}
\end{equation}
where the coefficient processes $\mathscr{A}$, $\mathscr{B}$, and $\mathscr{C}^k$, $k\in\{1,\ldots,m\}$, are defined in~\eqref{eq:def_Ck} and~\eqref{eq:def_A_B}.

\begin{remark}\label{rem:SDE_hidden_dev_and_new_state}
	Let 
    $u\in\cL^2$.
	Then there exists a unique solution $\scH^u = (\scH^u(s))_{s\in[0,T]}$ of the SDE~\eqref{eq:def_scH_u}
	and it holds that $E[\sup_{s\in[0,T]} \lVert\scH^u(s) \rVert_F^2]<\infty$ (cf., for example, \cite[Theorem 3.2.2 \& Theorem 3.3.1]{zhang}).
\end{remark}

Due to~\eqref{eq:A1}, it 
holds for all 
$X \in \cA^{pm}$ 
that $\gamma^{-\frac12} D^X\in\cL^2$, 
and in the next result
we observe that 
the process $\scH^{\gamma^{-\frac12} D^X}$ is the multivariate version of the so-called scaled hidden deviation in~\cite{ackermann2022reducing}.

\begin{lemma}\label{lem:scH_hidden_dev}
	Let 
    $X \in \cA^{pm}$. 
	It then holds that 
	$\scH^{\gamma^{-\frac12}D^X} = \gamma^{-\frac12} D^X - \gamma^{\frac12} X.$
\end{lemma}

In the following result we provide a reformulation of the cost functional for progressively measurable strategies \eqref{eq:def_costfct_pm} 
in terms of the (scaled) deviation process and the corresponding solution of the SDE~\eqref{eq:def_scH_u} (the multivariate version of the scaled hidden deviation process). 
This representation is helpful in order to establish the continuous extension result \cref{thm:contextcostfct} and it appears again in \cref{sec:equivLQ} in the standard linear-quadratic stochastic control problem.  
Recall that $\mathscr{Q}$ and~$\KK$ are defined in~\eqref{eq:def_Q} and~\eqref{eq:def_KK}. 

\begin{propo}\label{lem:pm_cost_fct_as_LQ}
	Let 
    $X \in \cA^{pm}$.  
	Then 
    $\pmJ(X)$ is well defined and finite 
	and it holds 
	\begin{align}
    \nonumber
	\pmJ(X)	
	& = \frac12 E\big[ \big(\scH^{\gamma^{-\frac12} D^X}(T) + \gamma^{\frac12}(T) \xi \big)^\top \big(\scH^{\gamma^{-\frac12} D^X}(T) + \gamma^{\frac12}(T) \xi \big) \big] 
	- \frac12 d^\top \gamma^{-1}(0) d \\ \nonumber
	& \quad + 
	E\bigg[ \int_0^T (\gamma^{-\frac12}(s) D^X(s))^\top 
	\KK(s) 
	\gamma^{-\frac12}(s) D^X(s)  ds \bigg] 
	 \\ \nonumber
	& \quad  
	+ E\bigg[ \int_0^T \big(\scH^{\gamma^{-\frac12} D^X}(s) + \gamma^{\frac12}(s) \zeta(s) \big)^\top \mathscr{Q}(s) \big(\scH^{\gamma^{-\frac12} D^X}(s) + \gamma^{\frac12}(s) \zeta(s) \big) ds \bigg] \\ 
	& \quad - 2 E\bigg[ \int_0^T (\gamma^{-\frac12}(s) D^X(s))^\top \mathscr{Q}(s) \big(\scH^{\gamma^{-\frac12} D^X}(s) + \gamma^{\frac12}(s) \zeta(s) \big) ds \bigg] 
	. \label{eq:cost_fct_shd}
	\end{align} 
\end{propo}

Next we introduce 
a bijection 
\begin{equation}\label{eq:def_phi_b}
	\varphi\colon \cA^{pm} \to \cL^2, \,
	\varphi(X)=\gamma^{-\frac12}D^X
\end{equation}
between the set of progressively measurable execution strategies~$\cA^{pm}$ and the set of square-integrable controls~$\cL^2$.

\begin{propo}\label{lem:bijective}
	It holds that $\varphi\colon \cA^{pm}\to \cL^2$ defined by~\eqref{eq:def_phi_b} 
	is bijective with inverse
	$\ol{\varphi} \colon \cL^2 \to \cA^{pm}$ 
	defined by 
	\begin{equation}\label{eq:1607}
		\begin{split}
			& (\ol{\varphi}(u))(s) = \gamma^{-\frac12}(s) (u(s)-\scH^u(s)), 
			\quad s\in [0,T),  \\
			& (\ol{\varphi}(u))(0-) = x, \quad (\ol{\varphi}(u))(T) = \xi,
			\quad u \in \cL^2 .
		\end{split}
	\end{equation}
	Moreover, it holds for all $u\in\cL^2$ and $s \in [0,T)$ that 
	$D^{\ol{\varphi}(u)}(s) = \gamma^{\frac12}(s) u(s)$. 
\end{propo}

We equip the set of progressively measurable strategies with a metric. 
For 
$X,\widetilde{X} \in \cA^{pm}$ 
with associated deviation processes $D^X, D^{\widetilde{X}}$ we define  
\begin{equation}\label{eq:def_metric}
	\md(X,\widetilde{X}) = \bigg( E\bigg[ \int_0^T \big(D^X(s)-D^{\widetilde{X}}(s)\big)^\top \gamma^{-1}(s) \big(D^X(s)-D^{\widetilde{X}}(s)\big) ds \bigg] \bigg)^{\frac12} .
\end{equation}

\begin{remark}\label{rem:metric_and_cLt2}
	Observe that $(\cL^2,\md_{\cL^2})$ with $\md_{\cL^2}(u,v) = \lVert u-v \rVert_{\cL^2}$, 
	$u,v \in \cL^2$, is a complete metric space. 
	Therefore, \Cref{lem:bijective} and 
	the fact that for all $X,\widetilde{X} \in \cA^{pm}$ it holds that
	$\md(X,\widetilde{X}) = \md_{\cL^2}(\gamma^{-\frac12}D^X,\gamma^{-\frac12}D^{\widetilde{X}})$ 
	prove that $(\cA^{pm},\md)$ is a complete metric space.
\end{remark}

The following result ensures that the convergence of progressively measurable strategies in the metric $\md$ implies a suitable convergence for the corresponding solutions of the SDE~\eqref{eq:def_scH_u}. 

\begin{lemma}\label{lem:convergence_SDE}
	Suppose that $X \in \cA^{pm}$ and that $(X^N)_{N\in\N}$ is a sequence in $\cA^{pm}$ such that $\lim_{N\to\infty}\md(X,X^N)=0$. 
	Denote the solution of~\eqref{eq:def_scH_u} associated to $\gamma^{-\frac12} D^{X} \in \cL^2$ by $\scH$. 
	For $N\in\N$ denote the solution of~\eqref{eq:def_scH_u} associated to $\gamma^{-\frac12} D^{X^N} \in \cL^2$ by $\scH^N$. 
	Then it holds that 
    $\lim_{N\to\infty} E[\sup_{s\in[0,T]} \lVert \scH(s) - \scH^N(s) \rVert_F^2 ] = 0$.
\end{lemma}

 We conclude this section with the main result \cref{thm:contextcostfct} that the cost functional~$\pmJ$ in~\eqref{eq:def_costfct_pm} can be regarded as a continuous extension of the cost functional~$\fvJ$ in~\eqref{eq:defcostfctmultivariate} from finite-variation strategies to progressively measurable strategies.

\begin{theo}\label{thm:contextcostfct} 
	(i) 
	Suppose that $X\in\cA^{pm}$.
	Then for every sequence $(X^N)_{N\in\N}$ in $\cA^{pm}$ with
	$\lim_{N \to \infty} \md(X^N,X) = 0$ it holds that
	$\lim_{N\to\infty} \lvert \pmJ(X^{N}) - \pmJ(X) \rvert = 0$.
	
	(ii)
	For every $X\in\cA^{pm}$
	there exists a sequence $(X^N)_{N\in\N}$ in $\cA^{fv}$ such that it holds  
    $\lim_{N \to \infty} \md(X^N,X) = 0$.
	
	(iii)
	It holds that 
	\begin{equation}\label{eq:inf_equal}
	\inf_{X\in\cA^{fv}}\fvJ(X)=
	\inf_{X\in\cA^{pm}}\pmJ(X).
	\end{equation}	
\end{theo}

\section{Solution of the trade execution problem and examples}
\label{sec:soln}

 Observe that~\eqref{eq:cost_fct_shd} in \cref{lem:pm_cost_fct_as_LQ} is a representation of the cost functional~\eqref{eq:def_costfct_pm} as a cost functional which is quadratic in $(\gamma^{-\frac12} D^X,\scH^{\gamma^{-\frac12} D^X})$
 and wherein  $\gamma^{-\frac12} D^X$ appears only as an integrand and $\scH^{\gamma^{-\frac12} D^X}$ appears only as an integrand or evaluated at the terminal time. 
 Moreover, the SDE~\eqref{eq:def_scH_u} for $\scH^{\gamma^{-\frac12} D^X}$ is linear in $(\gamma^{-\frac12} D^X,\scH^{\gamma^{-\frac12} D^X})$. 
 Thus, we notice in \cref{sec:equivLQ} that the problem to minimize the right-hand side of~\eqref{eq:cost_fct_shd} over all $\gamma^{-\frac12} D^X \in \cL^2$ subject to the state process $\scH^{\gamma^{-\frac12} D^X}$ has a ``standard'' form 
 so that, under suitable assumptions, we can apply results from the literature
 on linear-quadratic (LQ) stochastic optimal control to solve that problem. 
 More specifically, in \cref{sec:zerotargets} we employ results from Sun et al.~\cite{sun2021indefiniteLQ} to obtain, under \cref{assump_filtration} and \cref{assump_convexity}, a unique solution of the LQ stochastic control problem of \cref{sec:equivLQ} 
 in the case of zero targets\footnote{Non-vanishing targets can, under suitable assumptions, be treated by combining results from Kohlmann \& Tang~\cite{kohlmann2003minimization} and Sun et al.~\cite{sun2021indefiniteLQ}; see also \cref{sec:appendixsolving}.} and a representation in terms of the solution of a Riccati BSDE. 
 A way back to a unique solution of the trade execution problem for progressively measurable strategies  
 in \cref{sec:pm_problem} is provided by \cref{cor:linkLQpm}, which results in \cref{cor:soln_opt_trade_execution}. 
 In \cref{sec:examples} we then use \cref{cor:soln_opt_trade_execution} to study optimal strategies in some examples.

\subsection{An equivalent LQ stochastic control problem}
\label{sec:equivLQ}

For 
$u\in\cL^2$ we define the cost functional $\LQJ$ by 
\begin{equation}\label{eq:def_cost_fct_LQ}
\begin{split}
\LQJ(u) 
& = \frac12 E\big[ \big(\scH^u(T)+\gamma^{\frac12}(T)\xi \big)^\top \big(\scH^u(T)+\gamma^{\frac12}(T) \xi \big) \big] 
+ E\bigg[ \int_0^T (u(s))^\top 
\KK(s) 
u(s)  ds \bigg] \\
& \quad 
+ E\bigg[ \int_0^T \big(\scH^u(s) + \gamma^{\frac12}(s) \zeta(s) \big)^\top \mathscr{Q}(s) \big(\scH^u(s) + \gamma^{\frac12}(s) \zeta(s) \big) ds \bigg] \\
& \quad - E\bigg[ \int_0^T 2 (u(s))^\top \mathscr{Q}(s) \big(\scH^u(s) + \gamma^{\frac12}(s) \zeta(s) \big)  ds \bigg]   
 ,
\end{split}
\end{equation} 
where the state process $\scH^u=(\scH^u(s))_{s\in[0,T]}$ is the solution to the SDE~\eqref{eq:def_scH_u} and~$\mathscr{Q}$ and~$\KK$ are defined in~\eqref{eq:def_Q} and~\eqref{eq:def_KK}.

It is a direct consequence of \Cref{lem:bijective} and \Cref{lem:pm_cost_fct_as_LQ} 
that the control problems pertaining to $\pmJ$ and $\LQJ$ are equivalent in the following sense.

\begin{corollary}\label{cor:linkLQpm}
	(i) It holds a.s.\ that 
	\begin{equation*}
		\begin{split}
			\inf_{X \in \cA^{pm}} \pmJ(X) 
			& = \inf_{u\in\cL^2} \LQJ(u) - \tfrac12 d^\top \gamma^{-1}(0) d .
		\end{split}
	\end{equation*}
	
	(ii)
	Suppose that $X^*=(X^*(s))_{s \in [0,T]} \in \cA^{pm}$  minimizes 
    $\pmJ$ 
    over $\cA^{pm}$ and let $D^{X^*}$ be the associated deviation process.
	Then, $u^*=(u^*(s))_{s \in [0,T]}$ defined by 
	$$u^*(s)=\gamma^{-\frac12}(s) D^{X^*}(s), \quad s \in [0,T],$$ minimizes 
    $\LQJ$
    over $\cL^2$.
	
	(iii) 
	Suppose that $u^*=(u^*(s))_{s \in [0,T]} \in \cL^2$ minimizes 
    $\LQJ$ 
    over $\cL^2$ and let $\scH^{u^*}$ be the associated solution of \eqref{eq:def_scH_u} for $u^*$.
	Then, $X^*=(X^*(s))_{s\in[0,T]}$ defined by 
	\begin{equation*}
		\begin{split}
			& X^*(s)=\gamma^{-\frac12}(s) (u^*(s)-\scH^{u^*}(s)), 
			\quad s \in [0,T), \\
			& X^*(0-) = x, \quad X^*(T)=\xi,
		\end{split}
	\end{equation*}
	minimizes 
    $\pmJ$
    over $\cA^{pm}$. 
	Furthermore, the deviation process $D^{X^*}$ defined in \eqref{eq:def_deviation_pm} satisfies for all $s \in [0,T)$ that $D^{X^*}(s) = \gamma^{\frac12}(s) u^*(s)$. 
	
	(iv)
	There exists a $dP\times ds$-a.e.\ unique minimizer of 
    $\LQJ$
    in $\cL^2$ if and only if there exists a $dP\times ds$-a.e.\ unique minimizer of 
    $\pmJ$ 
    in $\cA^{pm}$. 
\end{corollary}

\subsection{Solution of the trade execution problem for progressively measurable strategies with zero targets}
\label{sec:zerotargets}

We here consider the case of a zero terminal position $\xi=0$ and a zero running target $\zeta=0$, in which case the cost functional in \cref{sec:equivLQ} does not contain inhomogeneities.  
We make the following standard assumption of a Brownian filtration.  

\begin{assumption}\label{assump_filtration}
	Assume that the filtration $(\cF_s)_{s\in[0,T]}$ 
	is the augmented natural filtration of the $m$-dimensional Brownian motion~$W$. 
\end{assumption}

In Sun et al.~\cite{sun2021indefiniteLQ}, which we are about to apply, there is a uniform convexity assumption on the cost functional.
To formulate this assumption, we consider the cost functional $J^0$, which we define as in \eqref{eq:def_cost_fct_LQ} but with $x$ and $d$ satisfying $\gamma^{-\frac12}(0) d - \gamma^{\frac12}(0) x = 0$, that is, the state process $\scH^u$ controlled by $u\in\cL^2$ starts in $0$. This restriction is solely used for the formulation of the following assumption; the results below also apply when $\gamma^{-\frac12}(0) d - \gamma^{\frac12}(0) x \neq 0$.

\begin{assumption}\label{assump_convexity}
	Assume that 
	there exists $\delta \in (0,\infty)$ such that for all $u \in \cL^2$ it holds that 
	\begin{equation}\label{eq:convexity_cond}
        J^0(u)
        \ge \delta E\bigg[ \int_0^T \lVert u(s) \rVert_F^2 ds \bigg] .
	\end{equation}
\end{assumption}

This assumption implies that the LQ problem with zero targets in a Brownian filtration is uniquely solvable for any initial condition $\gamma^{-\frac12}(0) d - \gamma^{\frac12}(0) x \in \R^n$ (cf.~\cite[Equation~(44), Corollary~3.5(ii), and Definition~1.1(ii)]{sun2021indefiniteLQ}).

\begin{remark}\label{rem:sufficient_cond_convex}
	Consider the case $\xi=0$ and $\zeta=0$. 
	For sufficient conditions for \cref{assump_convexity} to hold, 
	we refer to \cite[Section~7]{sun2021indefiniteLQ}. 
	In particular, if $\risk \equiv 0$ and there exists $\delta \in (0,\infty)$ such that $\KK-\delta I_n$ is $\cS_{\ge 0}^n$-valued, 
	then~\eqref{eq:convexity_cond} is satisfied (cf.\ \cite[equation~(5)]{sun2021indefiniteLQ}). 
	Moreover, if $\risk \equiv 0$ and $\KK$ is $\cS_{\ge 0}^n$-valued and there exists $\delta \in (0,\tfrac12)$ such that $4\sum_{k=1}^m \mathscr{C}^k\mathscr{C}^k -\delta I_n$ is $\cS_{\ge 0}^n$-valued, 
	then \eqref{eq:convexity_cond} is satisfied (cf.\ \cite[equation~(8)]{sun2021indefiniteLQ}). 
	A more complex sufficient condition is provided in \cite[Theorem~7.3]{sun2021indefiniteLQ}. 
	Furthermore, note that if $\mathscr{Q}$ is $\cS_{\ge 0}^n$-valued and there exists $\delta \in (0,\infty)$ such that $\kap-\delta I_n$ is $\cS_{\ge 0}^n$-valued, then \eqref{eq:convexity_cond} is satisfied as well.
\end{remark}

We next introduce a matrix-valued BSDE of Riccati type, which is strongly connected to the LQ problem of \cref{sec:equivLQ}:
\begin{equation}\label{eq:BSDE}
	\begin{split}
		d \mathscr{Y}(s)
		& = - g\big(s,\cdot,\mathscr{Y}(s),\mathscr{Z}^1(s),\ldots,\mathscr{Z}^m(s)\big) ds
		+ \sum_{k=1}^m \mathscr{Z}^k(s) dW_k(s),
		\quad s \in [0,T],  \\
		\mathscr{Y}(T) & = \tfrac12 I_n
	\end{split}
\end{equation}
with the driver 
\begin{equation}\label{eq:driver_of_BSDE}
	\begin{split}
		& g\big(s,\omega,\mathscr{Y}(s,\omega),\mathscr{Z}^1(s,\omega),\ldots,\mathscr{Z}^m(s,\omega)\big) \\
		& = \mathscr{Y}\mathscr{A} + \mathscr{A} \mathscr{Y} 
		+ \mathscr{Q}
		+ \sum_{k=1}^m \big( \mathscr{C}^k \mathscr{Y} \mathscr{C}^k + \mathscr{Z}^k \mathscr{C}^k + \mathscr{C}^k \mathscr{Z}^k \big) 
		 \\
		& \quad - \bigg( \mathscr{Y} \mathscr{B}  - \mathscr{Q} - 2 \sum_{k=1}^m \big( \mathscr{C}^k \mathscr{Y} \mathscr{C}^k + \mathscr{Z}^k \mathscr{C}^k \big) \bigg)
		\cdot \bigg( \KK + 4 \sum_{k=1}^m \mathscr{C}^k \mathscr{Y} \mathscr{C}^k  \bigg)^{-1} \\
		& \qquad \cdot 
		\bigg( \mathscr{B}^\top \mathscr{Y}  - \mathscr{Q} - 2 \sum_{k=1}^m \big( \mathscr{C}^k \mathscr{Y} \mathscr{C}^k + \mathscr{C}^k \mathscr{Z}^k \big)  \bigg), 
	\end{split}
\end{equation}
where on the right-hand side of the equation for the driver we suppressed the dependence on $\omega\in\Omega$ and $s \in [0,T]$. 
A pair $(\mathscr{Y},\mathscr{Z})$ with $\mathscr{Z}=(\mathscr{Z}^1,\ldots,\mathscr{Z}^m)$ is called  
a solution of the BSDE~\eqref{eq:BSDE} if 
\begin{itemize}
	\item the process $\mathscr{Y} \colon [0,T]\times \Omega \to \cS^n$ is bounded, adapted, and continuous,
	
	\item for every $k\in\{1,\ldots,m\}$ it holds that the process $\mathscr{Z}^k\colon [0,T]\times \Omega \to \cS^n$ is progressively measurable and satisfies $E[\int_0^T \lVert \mathscr{Z}^k(s) \rVert_F^2 ds] < \infty$, 
	
	\item 
	the process 
	$\KK + 4 \sum_{k=1}^m \mathscr{C}^k \mathscr{Y} \mathscr{C}^k$ is    
	$dP\times ds$-a.e.\ $\cS^n_{> 0}$-valued, 
	and
	
	\item the BSDE~\eqref{eq:BSDE} is satisfied $P$-a.s.
\end{itemize}

Given a solution $(\mathscr{Y},\mathscr{Z})$ of the BSDE~\eqref{eq:BSDE}, we define the matrix-valued progressively measurable process $\theta=(\theta(s))_{s \in [0,T]}$ by, for all $s\in[0,T]$, 
\begin{equation}\label{eq:def_theta}
	\begin{split}
		\theta(s) & = 
		- \bigg( \KK(s) + 4 \sum_{k=1}^m \mathscr{C}^k(s) \mathscr{Y}(s) \mathscr{C}^k(s)  \bigg)^{-1} \\
		& \quad \cdot 
		\bigg( (\mathscr{B}(s))^\top \mathscr{Y}(s) - \mathscr{Q}(s) - 2 \sum_{k=1}^m \big( \mathscr{C}^k(s) \mathscr{Y}(s) \mathscr{C}^k(s) + \mathscr{C}^k(s) \mathscr{Z}^k(s) \big) \bigg) . 
	\end{split}
\end{equation}
In addition, given 
a solution $(\mathscr{Y},\mathscr{Z})$ of the BSDE~\eqref{eq:BSDE},
we consider the $\R^n$-valued SDE 
\begin{equation}\label{eq:SDE_optimal_state}
	\begin{split}
		d\scH^*(s) & = 
		\big[ \mathscr{A}(s) + \mathscr{B}(s) \theta(s) \big] \scH^*(s) ds 
		+ \sum_{k=1}^m \big[ \mathscr{C}^k(s) (I_n - 2 \theta(s)) \big] \scH^*(s) dW_k(s), 
		\; s\in[0,T], \\
		\scH^*(0) & = \gamma^{-\frac12}(0) d - \gamma^{\frac12}(0) x.
	\end{split}
\end{equation}

We now apply results from Sun et al.~\cite{sun2021indefiniteLQ} to our situation.
This leads to \cref{propo:soln_LQ_Sun}, where the LQ stochastic control problem given by \eqref{eq:def_cost_fct_LQ} and \eqref{eq:def_scH_u} is solved 
in the case of zero targets.

\begin{propo}\label{propo:soln_LQ_Sun}
	Assume that $\xi=0$ and $\zeta=0$. 
	Let \cref{assump_filtration} and \cref{assump_convexity} be in force.
	
	(i)
	There exists a unique solution $(\mathscr{Y},\mathscr{Z})$ of the BSDE~\eqref{eq:BSDE}.  
	Moreover, there exists $\varepsilon\in(0,\infty)$ such that 
	$$\KK + 4 \sum_{k=1}^m \mathscr{C}^k \mathscr{Y} \mathscr{C}^k -\varepsilon I_n$$ 
	is    
	$dP\times ds$-a.e.\ $\cS^n_{\ge 0}$-valued.
	
	(ii)
	Let $(\mathscr{Y},\mathscr{Z})$ be the unique solution of the BSDE~\eqref{eq:BSDE}. 
	Then there exists a unique $u^*\in\cL^2$ such that for all $u\in\cL^2$ it holds $P$-a.s.\ that $\LQJ(u^*)\le \LQJ(u)$. 
	Moreover, there exists a unique solution $\scH^*$ of the SDE~\eqref{eq:SDE_optimal_state} (the state process associated to~$u^*$), and the unique optimal control~$u^*$ admits the representation 
	\begin{equation}
		u^*(s)=\theta(s)\scH^*(s), 
		\quad s \in[0,T], 
	\end{equation}
	where $\theta$ is defined in~\eqref{eq:def_theta}.
	
	(iii)
	Let $(\mathscr{Y},\mathscr{Z})$ be the unique solution of the BSDE~\eqref{eq:BSDE}. 
    Then 
	\begin{equation}\label{eq:optcostsLQ}
		\inf_{u\in\cL^2} \LQJ(u)
		= \bigl( \gamma^{-\frac12}(0) d - \gamma^{\frac12}(0) x \bigr)^\top \mathscr{Y}(0) \bigl( \gamma^{-\frac12}(0) d - \gamma^{\frac12}(0) x \bigr)
		.
	\end{equation}
\end{propo}

\begin{remark}\label{rem:necessaryforconvex}
	Suppose that \cref{assump_filtration} and \cref{assump_convexity} are in force and that $\sigma= 0$, $\xi=0$, and $\zeta=0$. 
	Then \cref{propo:soln_LQ_Sun} shows that there must exist $\varepsilon \in (0,\infty)$ such that $\KK-\varepsilon I_n$ is $dP\times ds$-a.e.\ $\cS^n_{\ge 0}$-valued.
	In particular, if $\risk \equiv 0$ then $\kap$ has to be $dP\times ds$-a.e.\ $\cS_{>0}^n$-valued.
\end{remark}

In \cref{cor:soln_opt_trade_execution} we state the solution of the trade execution problem of \cref{sec:pm_problem} 
with zero targets $\xi=0$ and $\zeta=0$. 
This result is obtained by combining \cref{propo:soln_LQ_Sun} and \cref{cor:linkLQpm}. 

\begin{corollary}\label{cor:soln_opt_trade_execution}
	Assume that $\xi=0$ and $\zeta=0$.
	Let \cref{assump_filtration} and \cref{assump_convexity} be in force. 
	Let $(\mathscr{Y},\mathscr{Z})$ be the unique solution of the BSDE~\eqref{eq:BSDE} (cf.\ \cref{propo:soln_LQ_Sun}). 
	Recall the definition~\eqref{eq:def_theta} of $\theta$, 
    and let $\scH^*$ be the unique solution of the SDE~\eqref{eq:SDE_optimal_state} (cf.\ \cref{propo:soln_LQ_Sun}). 
	
	Then there exists a unique (up to $dP\times ds$-null sets) minimizer $X^*$ of 
    $\pmJ$ 
    in $\cA^{pm}$. 
	Moreover, it holds that 
	\begin{equation}\label{eq:def_opt_strat}
		\begin{split}
			& X^*(0-)=x, \quad X^*(T)=0,\\
			& X^*(s)
			= \gamma^{-\frac12}(s)
			\big( \theta(s) - I_n \big) \scH^*(s),
			\quad s\in[0,T). 
		\end{split}
	\end{equation} 
	The deviation process $D^*:= D^{X^*}$ (defined in \eqref{eq:def_deviation_pm}) satisfies 
	that 
	\begin{equation}\label{eq:def_opt_dev}
		D^*(s) 
		= \gamma^{\frac12}(s) \theta(s) \scH^*(s), 
		\quad s \in [0,T). 
	\end{equation}
	The optimal costs are given by 
	\begin{equation}\label{eq:optcosts_trade_execution_soln}
		\begin{split}
			\inf_{X \in \cA^{pm}} \pmJ(X) 
			& = d^\top  \gamma^{-\frac12}(0) \big(   \mathscr{Y}(0) - \tfrac12 I_n  \big) \gamma^{-\frac12}(0) d \\
			& \quad 
			- 2 d^\top \gamma^{-\frac12}(0) \mathscr{Y}(0)
			\gamma^{\frac12}(0) x 
			+ x^\top \gamma^{\frac12}(0) \mathscr{Y}(0) \gamma^{\frac12}(0) x 
			.
		\end{split}
	\end{equation}
\end{corollary}

\begin{remark}\label{rem:pmvsfv}
	If in \cref{cor:soln_opt_trade_execution} the optimal progressively measurable execution strategy $X^*$ is in $\cA^{fv}$, then $X^*$ is also the unique solution of the finite-variation stochastic control problem of \cref{sec:fv_problem} (cf.~\cref{thm:contextcostfct}). 
	However, there does not always exist an optimizer in this smaller class; see, for example, \cite[Example~6.4]{ackermann2020cadlag} and \cite[Section~4.3]{ackermann2022reducing}. These one-dimensional counterexamples demonstrate that the optimizer is not attained even within the class of semimartingales (the extension of the stochastic control problem to semimartingales is provided in \cite{ackermann2020cadlag}).
    In this context, we also note that several further works have considered trading strategies going beyond the finite-variation framework including, for example, 
    \cite{lorenz2013drift}, \cite{HorstKivman2024}, 
    and \cite{becherer2019stability}.
    Moreover, there is also empirical evidence for trading with inventories beyond finite variation (see, for example, \cite{carmona2019selffinancing} and \cite{CarmonaLeal2023}).
\end{remark}

\subsection{Examples}
\label{sec:examples}

We can now use \cref{cor:soln_opt_trade_execution} to study optimal strategies in some examples. 
Note that additional content and figures can be found in \cref{sec:appendix:examples}.

First, we observe that if $d=\gamma(0) x$ or $\rho \equiv 0$, then the naive strategy to close all positions immediately is optimal. 

\begin{lemma}\label{lem:closing_immed}
	Assume that $\xi=0$ and $\zeta=0$. 
	Let \cref{assump_filtration} and \cref{assump_convexity} be in force. 
	If $d=\gamma(0) x$ or $\rho \equiv 0$, then 
	$X^*(0-)=x$, $X^*(s)=0$, $s\in[0,T]$, 
    is the optimal strategy in 
	$\cA^{pm}$ for 
    $\pmJ$ 
    and the optimal strategy in $\cA^{fv}$ for 
    $\fvJ$. 
\end{lemma}

We next turn to a setting with constant $\gamma$ and $\rho$. 
Note that the model in Obizhaeva \& Wang \cite{obizhaeva2013optimal} consists of a single asset with constant price impact and constant resilience. 
Our framework developed in this paper allows to consider  
a multi-asset variant of that model 
where both the 
price impact and the resilience are matrices that do not need to be diagonal, that is, we include possible cross-effects between the assets. 
More specifically, we consider the following 
subsetting 
of \cref{sec:setting}. 

\begin{setting}\label{set:OW}
Let $\xi=0$, $\zeta= 0$, $\sigma= 0$, $\mu= 0$, and $\risk = 0$. 
Hence, we have that $\lambda_j \equiv \lambda_j(0)$ for all $j \in \{1,\ldots,n\}$, and $\gamma=\OO^\top \lam \OO$ is a deterministic matrix in~$\cS_{>0}^n$. 
We choose $\rho \in \R^{n\times n}\backslash\{0\}$ to be a deterministic matrix, too. 
Let \cref{assump_filtration} be in force. 
Furthermore, assume that $d\ne \gamma(0) x$. 
\end{setting}

In \cref{rem:conley} in the appendix we 
provide sufficient conditions which ensure that \cref{assump_convexity} in \cref{set:OW} (i.e., $\kap \in \cS_{>0}^n$) is satisfied, using Conley et al.~\cite[Theorem 2.1]{conley2005elliptic}. 
Observe that in \cref{set:OW} it holds for all $k\in\{1,\ldots,m\}$ that 
$\mathscr{A}\equiv 0$, $\mathscr{B}\equiv -\gamma^{-\frac12} \rho \gamma^{\frac12}$, $\mathscr{C}^k \equiv 0$, $\mathscr{Q}\equiv 0$, and $\KK=\kap$.
By solving the BSDE~\eqref{eq:BSDE} and the SDE~\eqref{eq:SDE_optimal_state}, we obtain from \cref{cor:soln_opt_trade_execution} the optimal strategy in the multi-asset Obizhaeva--Wang model. 

\begin{corollary}\label{cor:soln_OW}
	Assume \cref{set:OW} and that $\kap\in \cS_{>0}^n$ is satisfied. 
	Then,  
	$(\mathscr{Y},\mathscr{Z})$ given by 
	$\mathscr{Z}= 0$ and 
	\begin{equation*}
		\begin{split}
			\mathscr{Y}(s) & = \tfrac12 \bigl( I_n + \tfrac12 (T-s) \mathscr{B} \KK^{-1} \mathscr{B}^\top \bigr)^{-1}, \quad s\in [0,T],
		\end{split}
	\end{equation*}
	is the solution of the BSDE~\eqref{eq:BSDE}, 
	$\theta$ in~\eqref{eq:def_theta} is given by 
	$\theta(s) = -\KK^{-1} \mathscr{B}^\top \mathscr{Y}(s)$, $s\in[0,T]$, 
	and the solution of the SDE~\eqref{eq:SDE_optimal_state} is given by 
	\begin{equation*}
		\begin{split}
			\scH^*(s) 
			& = \scH^*(0) - \tfrac12 \mathscr{B} \KK^{-1} \mathscr{B}^\top
			\bigl(I_n+\tfrac12 T \mathscr{B} \KK^{-1} \mathscr{B}^\top \bigr)^{-1} \scH^*(0) s, \quad s\in[0,T] ,
		\end{split}
	\end{equation*}
	where $\scH^*(0)= \gamma^{-\frac12} d - \gamma^{\frac12} x$.
	Further, there exists a unique optimal strategy $X^* \in \cA^{pm}$ that minimizes 
    $\pmJ$,  
    and it holds that $X^* \in \cA^{fv}$ and 
	\begin{equation}\label{eq:opt_strat_OW}
		\begin{split}
			X^*(0-) & = x, \quad X^*(T) = 0, \\ 
			X^*(s) 
			& = -\tfrac12 \gamma^{-\frac12} \bigl( \KK^{-1} \mathscr{B}^\top
			(I_n+\tfrac12 T \mathscr{B} \KK^{-1} \mathscr{B}^\top )^{-1} 
			+ 2 I_n \bigr)
			\scH^*(0) \\
			& \quad 
			+ \tfrac12 \gamma^{-\frac12} \mathscr{B} \KK^{-1} \mathscr{B}^\top 
			\bigl(I_n+\tfrac12 T \mathscr{B} \KK^{-1} \mathscr{B}^\top \bigr)^{-1} 
			\scH^*(0)  s, \quad s\in[0,T)
			. 
		\end{split}
	\end{equation}
	The deviation process $D^*$ associated to the optimal strategy $X^*$ is given by  
	\begin{equation}\label{eq:opt_dev_OW}
		\begin{split}
			&D^*(0-)=d, 
			\qquad D^*(T)= \gamma^{\frac12} 
			\bigl(I_n+\tfrac12 T \mathscr{B} \KK^{-1} \mathscr{B}^\top \bigr)^{-1} \scH^*(0), \\
			&D^*(s)=-\tfrac12 \gamma^{\frac12} \KK^{-1} \mathscr{B}^\top  \bigl(I_n+\tfrac12 T \mathscr{B} \KK^{-1} \mathscr{B}^\top \bigr)^{-1} \scH^*(0),
			\quad  s\in[0,T).
		\end{split}
	\end{equation}
\end{corollary}

From the representation \eqref{eq:opt_strat_OW} of the optimal strategy we conclude that the optimal strategy in the multi-asset Obizhaeva--Wang model can have block trades only at the time points $0$ and $T$. In between, trading in all assets happens with constant rates.
In addition, the price deviation stays constant on $[0,T)$ when trading according to the optimal strategy (cf.~\eqref{eq:opt_dev_OW}). 
We recall that these features are also characteristic for optimal strategies in the classic single-asset Obizhaeva--Wang model \cite{obizhaeva2013optimal}. 
However, it is in general not optimal to simply take the optimal strategies from the individual single-asset models. 
In particular, we show in \cref{ex:crossingzero} that it can be optimal to also trade in an asset where the initial position is~$0$, whereas in the single-asset model the optimal strategy given the initial position~$0$ (and the initial deviation~$0$) is to stay at the position~$0$ (cf., for example, \cref{lem:closing_immed}).

\begin{ex}\label{ex:crossingzero}
	Within \cref{set:OW} suppose that $n=2$, $x_1\ne 0$, $x_2=0$, $d=0$, and 
	\begin{equation*}
		\rho = 
		\begin{pmatrix}
			\rho_1 & \rho_3\\
			\rho_3 & \rho_2
		\end{pmatrix}
		,
	\end{equation*}
	where $\rho_1,\rho_2,\rho_3\in\R$ are chosen in such a way that $\rho \in \cS_{>0}^2$. Moreover, choose $\gamma$ such that $\gamma$ and $\rho$ commute.  
	Then, \cref{cor:soln_OW} and some matrix computations demonstrate that the optimal strategy $X^* \in \cA^{fv}$ is given by $X^*(0-)=x$, $X^*(T)=0$, and 
	\begin{equation}\label{eq:1002ex}
		\begin{split}
			& X^*_1(s) = 
			\frac{\big[(1+(T-s)\rho_1)(2+T\rho_2)-T(T-s)\rho_3^2\big] x_1}{(2+T\rho_1)(2+T\rho_2)-T^2\rho_3^2 } , \\
			& X^*_2(s) = 
			\frac{(T-2s) \rho_3 x_1}{(2+T\rho_1)(2+T\rho_2)-T^2\rho_3^2} 
			, \quad s \in [0,T).
		\end{split}
	\end{equation}
	In particular, if $\rho_3\ne 0$, then it is optimal to also trade in the second asset although one has a non-zero objective only for the first asset. 
    In \cref{sec:appendixcrossresilience} we provide additional discussion of this example as well as a figure of the optimal strategy.  
\end{ex}

In the case of \cref{ex:crossingzero}, the behavior of the optimal strategy to entail trading in an asset with the initial position~$0$ is due to the off-diagonal entries in the resilience~$\rho$. 
We illustrate numerically in \cref{sec:example_risk} and \cref{sec:example_gamma}  
that also off-diagonal entries in~$\risk$ or in the price impact~$\gamma$, respectively, can lead to such effects.
To this end, we go beyond the multi-asset Obizhaeva--Wang model. 
Furthermore, we, in \cref{ex:opt_strat_sigma} in the appendix, present a simulation of the optimal strategy in a situation with stochastic price impact.

\section{Proofs}
\label{sec:proofs}

\subsection{Proofs for \cref{sec:multi_asset_fv}}

\begin{proof}[Proof of \cref{propo:rewritten_costs_and_deviation}]
	Integration by parts,  
	\eqref{eq:defnu}, 
	\eqref{eq:deviationdynmultivariate},
    and the fact that $\nu$ is continuous and of finite variation 
	imply that
	\begin{equation}\label{eq:2355}
		\begin{split}
			d(\nu(s) D^X(s)) 
			& = \nu(s) dD^X(s) + (d\nu(s) )D^X(s) + d[\nu, D^X](s) \\
			& = -\nu(s) \rho(s) D^X(s) ds 
			+ \nu(s) \gamma(s) dX(s) 
			+ \nu(s) \rho(s) D^X(s) ds \\
			& = \nu(s) \gamma(s) dX(s), \quad s\in[0,T].
		\end{split}
	\end{equation}
	We use this and the fact that $\Delta D^X(s) = \gamma(s) \Delta X(s)$, $s \in [0,T]$, to obtain that 
	\begin{equation}\label{eq:costfunctionalpart001}
		\begin{split}
			&\int_{[0,T]} \left( 2(D^X(s-))^\top + (\Delta X(s))^\top \gamma(s) \right) dX(s) \\
			& = \int_{[0,T]} \left( 2(D^X(s-))^\top + (\Delta D^X(s))^\top \right) \gamma^{-1}(s) \nu^{-1}(s) d(\nu(s) D^X(s)) \\ 
			& = \int_{[0,T]} \left( 2 (\mD(s-))^\top + (\Delta \mD(s))^\top \right) \aphi(s) d\mD(s) ,
		\end{split}
	\end{equation}
	where we introduced the abbreviations $\mD=\nu D^X$ and $\aphi=(\nu^{-1})^\top \gamma^{-1} \nu^{-1}$. 
	Note that, since $\mD=\nu D^X$ has finite variation, it holds for all $r \in [0,T]$ that 
	\begin{equation}\label{eq:2325a}
		\begin{split}
			\int_{[0,r]} (\Delta \mD(s))^\top \aphi(s) d\mD(s) 
			& = \sum_{k=1}^n \int_{[0,r]} \Delta \big((\mD(s))^\top \aphi(s) \big)_k d\mD_{k}(s)  \\
			& = \sum_{k=1}^n \big[\big( \mD^\top \aphi \big)_k, \mD_k \big](r)
			= \big[ \mD^\top \aphi , \mD \big](r) .
		\end{split}
	\end{equation}
	Furthermore, 
	integration by parts  
    and using that $\mD=\nu D^X$ has finite variation and $\aphi=(\nu^\top)^{-1}\gamma^{-1} \nu^{-1}$ is continuous shows 
	for all $s \in[0,T]$ that 
	\begin{equation}\label{eq:2325c}
		d\big( (\mD(s))^\top \aphi(s) \big) 
		= \bigl(d(\mD(s))^\top\bigr) \aphi(s) + (\mD(s))^\top d \aphi(s) . 
	\end{equation}
	Combining \eqref{eq:2325a},  
    integration by parts, and \eqref{eq:2325c} yields 
	for all $s \in [0,T]$ that 
	\begin{equation*}
		\begin{split}
			& \big( 2(\mD(s-))^\top + (\Delta \mD(s))^\top \big) \aphi(s) d\mD(s) \\
			& =  
			(\mD(s-))^\top \aphi(s) d\mD(s) 
			+ \bigl(d(\mD(s))^\top\bigr) \aphi(s) \mD(s-)  
			+ d \big[ \mD^\top \aphi , \mD \big](s)
			\\
			& = d\Big( \big((\mD(s))^\top \aphi(s) \big) \mD(s) \Big) - \Big( d\big((\mD(s))^\top \aphi(s)\big) \Big) \mD(s-) 
			+ \big(d(\mD(s))^\top \big) \aphi(s) \mD(s-) \\
			& = d\big( (\mD(s))^\top \aphi(s)  \mD(s) \big) 
			- (\mD(s))^\top (d\aphi(s)) \mD(s-) .
		\end{split}
	\end{equation*}
	This and \eqref{eq:costfunctionalpart001} together  establish~\eqref{eq:costfunctionalpart}.
	
	To obtain~\eqref{eq:rewritten_deviation}, note that~\eqref{eq:2355}, integration by parts, and the facts that $\nu\gamma$ is continuous and $X$ has finite variation show that	
	\begin{equation*}
		\begin{split}
			d(\nu(s) D^X(s) ) & = \nu(s) \gamma(s) dX(s)
			= - \big(d(\nu(s)\gamma(s))\big) X(s) + d(\nu(s)\gamma(s) X(s)) , \quad s\in[0,T]. 
		\end{split}
	\end{equation*}
	This implies for all $r\in[0,T]$ that 
	\begin{equation*}
		\begin{split}
			\nu(r) D^X(r) & = \nu(0-) D^X(0-) - \int_{0}^r (d(\nu(s)\gamma(s))) X(s) \\
			& \quad 
            + \nu(r) \gamma(r) X(r) - \nu(0-) \gamma(0-) X(0-) \\
			& = \nu(r) \gamma(r) X(r) + d - \gamma(0) x - \int_0^r  \big(d(\nu(s)\gamma(s))\big) X(s) . 
		\end{split}
	\end{equation*}
	It follows that $D^X$ satisfies~\eqref{eq:rewritten_deviation}. 
\end{proof}

We next remove $\nu$ from the integral in~\eqref{eq:costfunctionalpart}.

\begin{lemma}\label{lem:rewritten_costs2}
	Suppose that $X=(X(s))_{s\in[0,T]}$ is an $\R^n$-valued c\`adl\`ag finite-variation process with associated process $D^X=(D^X(s))_{s\in[0,T]}$ defined by~\eqref{eq:deviationdynmultivariate}.
	It then holds that 
    \begin{equation}\label{eq:costfunctionalpart2}
		\begin{split}
			&\int_{0}^T (D^X(s))^\top (\nu(s))^\top d\left((\nu^{-1}(s))^\top \gamma^{-1}(s) \nu^{-1}(s) \right) \nu(s) D^X(s) \\
			& = \int_0^T (D^X(s) )^\top \gamma^{-\frac12}(s) \Big( \gamma^{\frac12}(s) (d\gamma^{-1}(s)) \gamma^{\frac12}(s) \\
			& \qquad \qquad \qquad \qquad \qquad - \big(\gamma^{-\frac12}(s) \rho(s) \gamma^{\frac12}(s) + \gamma^{\frac12}(s) (\rho(s))^\top \gamma^{-\frac12}(s) \big) ds  \Big) 
			\gamma^{-\frac12}(s) D^X(s) . 
		\end{split}
	\end{equation}
\end{lemma}

\begin{proof}
    This follows from computing the dynamics of $(\nu^{-1})^\top \gamma^{-1} \nu^{-1}$ by integration by parts, using the dynamics~\eqref{eq:eqnuinv} of $\nu^{-1}$ and that $\nu^{-1}$ is continuous and of finite variation.
\end{proof}

The following simple lemma, which is an application of It\^o's lemma, provides the dynamics for powers of $\lambda_j$, $j\in\{1,\ldots,n\}$. 
In particular, we further obtain the dynamics of $\lam^{-1}$, $\lam^{\frac12}$, and $\lam^{-\frac12}$, which can be used to compute expressions such as $\gamma^{\frac12} (d\gamma^{-1}) \gamma^{\frac12}$ in the integral for the costs.

\begin{lemma}\label{lemma:ito_to_lambda}
	For all $\alpha \in \R$, $j\in\{1,\ldots,n\}$, $s\in[0,T]$ it holds that  
	\begin{equation}\label{eq:lambda_alpha_dyn}
		\begin{split}
			d(\lambda_j(s))^{\alpha}
			& = \alpha (\lambda_j(s))^{\alpha} 
			\biggl(  
			\Bigl(
			\mu_j(s) + \frac12 (\alpha-1) \sum_{k=1}^m (\sigma_{j,k}(s))^2  
			\Bigr) ds 
			+ \sum_{k=1}^m \sigma_{j,k}(s) dW_k(s) 
			\biggr)
			.
		\end{split}
	\end{equation}
	In particular, it holds for 
	all $s\in[0,T]$ that  
	\begin{equation}\label{eq:funct_of_lambda_dyn}
		\begin{split}
			d\lam^{-1}(s) 
			& = \lam^{-1}(s) \bigg( \Big( -\diagmu(s) + \sum_{k=1}^m \diagsigma_{k}(s)  (\diagsigma_{k}(s) )^\top  \Big) ds - \sum_{k=1}^m \diagsigma_{k}(s) dW_k(s) \bigg), \\
			d \lam^{\frac12}(s) 
			& = \lam^{\frac12}(s) \bigg( \Big( \frac12 \diagmu(s) - \frac18 \sum_{k=1}^m \diagsigma_{k}(s)  (\diagsigma_{k}(s) )^\top  \Big) ds + \frac12 \sum_{k=1}^m \diagsigma_{k}(s)   dW_k(s) \bigg), \\
			d \lam^{-\frac12}(s)
			& = \lam^{-\frac12}(s) \bigg( \Big( -\frac12 \diagmu(s) + \frac38 \sum_{k=1}^m \diagsigma_{k}(s)  (\diagsigma_{k}(s) )^\top  \Big) ds - \frac12 \sum_{k=1}^m \diagsigma_{k}(s)  dW_k(s) \bigg) .
		\end{split}
	\end{equation}
\end{lemma}

\begin{proof}
    It\^o's Lemma applied for all $\alpha \in\mathbb{R}$ to the function $f\colon (0,\infty) \to \mathbb{R}$, $f(x)=x^{\alpha}$, and using~\eqref{eq:dyn_lambda} yields~\eqref{eq:lambda_alpha_dyn}. By inserting $-1$, $\frac12$, and $-\frac12$, respectively, for $\alpha$ into~\eqref{eq:lambda_alpha_dyn} and using~\eqref{eq:def_lam}, we obtain \eqref{eq:funct_of_lambda_dyn}.
\end{proof}

Further, we obtain the following corollary of \cref{lemma:ito_to_lambda}, which we state without proof.

\begin{corollary}\label{cor:dlam_lam}
	It holds for all $s\in[0,T]$ that 
	\begin{equation*}
		\begin{split}
			\big(d\lam^{\frac12}(s)\big) \lam^{-\frac12}(s) 
			& = \bigg(\frac12 \diagmu(s)  - \frac18 \sum_{k=1}^m \diagsigma_{k}(s)  (\diagsigma_{k}(s) )^\top  \bigg) ds  + \frac12 \sum_{k=1}^m \diagsigma_{k}(s) dW_k(s) , \\
			\big(d\lam^{-\frac12}(s)\big) \lam^{\frac12}(s) 
			& = 
			\bigg( -\frac12 \diagmu(s) + \frac38 \sum_{k=1}^m \diagsigma_{k}(s)  (\diagsigma_{k}(s) )^\top  \bigg) ds - \frac12 \sum_{k=1}^m \diagsigma_{k}(s) dW_k(s) ,\\
			\lam^{\frac12}(s) \big( d \lam^{-1}(s) \big) \lam^{\frac12}(s) 
			& = \bigg( \Big( - \diagmu(s) + \sum_{k=1}^m \diagsigma_{k}(s)  (\diagsigma_{k}(s) )^\top  \Big) ds - \sum_{k=1}^m \diagsigma_{k}(s) dW_k(s) \bigg) . 
		\end{split}
	\end{equation*}
\end{corollary}

We now use the preceding results to prove \cref{lem:rewritten_costs3}.

\begin{proof}[Proof of \cref{lem:rewritten_costs3}]
	Observe that \Cref{propo:rewritten_costs_and_deviation} and \Cref{lem:rewritten_costs2} show that 
	\begin{equation}\label{eq:2224a}
		\begin{split}
			C(X)
			& = \frac12 (D^X(T))^\top \gamma^{-1}(T) D^X(T)  
			- \frac12 d^\top \gamma^{-1}(0) d \\
			& \quad - \frac12 \int_0^T ( D^X(s))^\top \gamma^{-\frac12}(s) \Big( \gamma^{\frac12}(s) (d\gamma^{-1}(s))\gamma^{\frac12}(s) \\
			& \qquad \qquad - \big(\gamma^{-\frac12}(s) \rho(s) \gamma^{\frac12}(s) + \gamma^{\frac12}(s) (\rho(s))^\top \gamma^{-\frac12}(s) \big) ds \Big) 
			\gamma^{-\frac12}(s) D^X(s) .
		\end{split}
	\end{equation}
	Moreover, note that by using~\cref{cor:dlam_lam} 
	we obtain for all $s\in[0,T]$ that 
	\begin{align}\label{eq:2224b}
			\gamma^{\frac12}(s) (d\gamma^{-1}(s)) \gamma^{\frac12}(s) 
			& = \OO^\top \lam^{\frac12}(s) \Big( d \lam^{-1}(s) \Big) \lam^{\frac12}(s) \OO \\
			& = \OO^\top \Big( - \diagmu(s) + \sum_{k=1}^m \diagsigma_{k}(s) \diagsigma_{k}(s) \Big) \OO ds 
			- \sum_{k=1}^m \OO^\top \diagsigma_{k}(s) \OO  dW_k(s) . \nonumber
	\end{align}
	Combining this, \eqref{eq:2224a}, and \eqref{eq:def_kap} 
	demonstrates that 
	\begin{equation}\label{eq:1607x}
		\begin{split}
			C(X)
			& = \frac12 (D^X(T))^\top \gamma^{-1}(T) D^X(T)  
			- \frac12 d^\top \gamma^{-1}(0) d \\
			& \quad + \int_0^T ( D^X(s))^\top \gamma^{-\frac12}(s) 
			\kap(s) 
			\gamma^{-\frac12}(s) D^X(s) ds \\
			& \quad + \frac12 \sum_{k=1}^m \int_0^T ( D^X(s))^\top \gamma^{-\frac12}(s) \OO^\top \diagsigma_{k}(s) \OO \gamma^{-\frac12}(s) D^X(s)  dW_k(s) .
		\end{split}
	\end{equation}
	The Burkholder--Davis--Gundy inequality and~\eqref{eq:A2} imply that there exists $c_1\in(0,\infty)$ such that for all $k\in\{1,\ldots,m\}$ it holds that 
	\begin{equation}\label{eq:1544x}
		\begin{split}
			& E\bigg[ \sup_{r \in [0,T]} \bigg\lvert \int_0^r (D^X(s) )^\top \gamma^{-\frac12}(s)  \OO^\top \diagsigma_{k}(s) \OO \gamma^{-\frac12}(s) D^X(s) dW_k(s) \bigg\rvert \bigg] \\
			& \le c_1 
			E\bigg[ \bigg( \int_0^T \Big\lvert (\gamma^{-\frac12}(s) D^X(s))^\top \OO^\top \diagsigma_{k}(s) \OO \gamma^{-\frac12}(s) D^X(s) \Big\rvert^2 ds \bigg)^\frac12 \bigg] 
			< \infty .
		\end{split}
	\end{equation}
	This proves that 
	\begin{equation}\label{eq:2224c}
		\sum_{k=1}^m E\bigg[ \int_0^T ( D^X(s))^\top \gamma^{-\frac12}(s)  \OO^\top \diagsigma_{k}(s) \OO \gamma^{-\frac12}(s) D^X(s) dW_k(s) \bigg] = 0 .
	\end{equation}
	Since $\kap$ is 
	$dP\times ds$-a.e.\ bounded (see~\cref{rem:remark_on_setting_assumptions2}), 
	we obtain from Jensen's inequality and~\eqref{eq:A1} that there exists $c_2\in(0,\infty)$ such that
	\begin{equation}\label{eq:2224d}
		\begin{split}
			E\bigg[ \bigg\lvert  \int_0^T ( D^X(s))^\top \gamma^{-\frac12}(s) \kap(s) 
			\gamma^{-\frac12}(s) D^X(s)  ds \bigg\rvert \bigg] 
			& \le c_2  E\bigg[  \int_0^T \lVert \gamma^{-\frac12}(s) D^X(s) \rVert_F^2 ds \bigg]
			< \infty .
		\end{split}
	\end{equation}
	Moreover, since $\gamma^{-1}$ is 
	$\cS_{\ge 0}^n$-valued, we have
	$(D^X(T))^\top \gamma^{-1}(T) D^X(T) \ge 0.$ 
	This, the fact that $\gamma^{-1}(0)$ is 
    deterministic, 
    \eqref{eq:1607x}, \eqref{eq:1544x}, \eqref{eq:2224c}, and \eqref{eq:2224d} 
	show that $E[C(X)]$ is well defined and admits the representation~\eqref{eq:fvJ_rewritten}.
\end{proof}

\subsection{Proofs for \cref{sec:cont_ext_to_pm}}

\begin{proof}[Proof of \cref{lem:deviationcompatibility}]
    Since $D^X$ here is a semimartingale, we have that 
    integration by parts, \eqref{eq:defnu}, \eqref{eq:eqnuinv}, and \eqref{eq:def_deviation_pm} show for all $r\in[0,T]$ that 
    \begin{equation*}
	\begin{split}
		dD^X(r) & = (d\gamma(r)) X(r) + \gamma(r) dX(r) + (d\nu^{-1}(r)) \bigg( d - \gamma(0) x - \int_0^r \big(d(\nu(s)\gamma(s))\big) X(s) \bigg) \\
		& \quad - \nu^{-1}(r) (d\nu(r)) \gamma(r) X(r) 
		- (d\gamma(r)) X(r)   \\
		& = \gamma(r) dX(r) 
		- \rho(r) \nu^{-1}(r) \bigg( d - \gamma(0) x - \int_0^r \big(d(\nu(s)\gamma(s))\big) X(s) \bigg) dr \\
		& \quad - \rho(r) \gamma(r) X(r) dr \\
		& = \gamma(r) dX(r) - \rho(r) D^X(r) dr.
	\end{split}
    \end{equation*}
    Furthermore, note that
    $D^X(0-)=\gamma(0) X(0-) + \nu^{-1}(0) (d-\gamma(0) x)=d$.
\end{proof}

We continue with a lemma wherein we compute the dynamics of 
$\gamma^{-\frac12}\nu^{-1}\beta$ for a certain process $\beta$. 
In the proofs of \cref{lem:scH_hidden_dev} and \cref{lem:bijective} we use this to show that $\scH^{\gamma^{-\frac12} D^X} = \gamma^{-\frac12} \nu^{-1} \beta = \gamma^{-\frac12} D^X - \gamma^{\frac12} X$.

\begin{lemma}\label{lem:computations_for_SDE}
	Assume that $X=(X(s))_{s\in[0,T]}$ is an $\R^{n}$-valued progressively measurable process such that 
	$\int_0^T \lVert X(s) \rVert_F^2 ds < \infty$ a.s.\
	and let 
    \begin{equation}\label{eq:defbeta}
        \beta(s) = d-\gamma(0) x - \int_0^s (d(\nu(r)\gamma(r))) X(r), \quad s\in[0,T].
    \end{equation}
	It then holds for all $s\in[0,T]$ that 
	\begin{equation*}
		\begin{split}
			d\big(\gamma^{-\frac12}(s) \nu^{-1}(s) \beta(s)\big) 
			& = - \bigg( \gamma^{-\frac12}(s) \rho(s) \gamma^{\frac12}(s) ds 
			+ \OO^\top \Big( \diagmu(s) - \frac12 \sum_{k=1}^m \diagsigma_{k}(s) \diagsigma_{k}(s) \Big) \OO ds \\
			& \qquad \quad + \sum_{k=1}^m \OO^\top \diagsigma_{k}(s) \OO dW_k(s) 
			\bigg) \gamma^{\frac12}(s) X(s) \\
			& \quad - \bigg( \gamma^{-\frac12}(s) \rho(s) \gamma^{\frac12}(s) ds 
			+ \OO^\top \Big( \frac12 \diagmu(s) - \frac38 \sum_{k=1}^m \diagsigma_{k}(s) \diagsigma_{k}(s) \Big) \OO ds \\
			& \qquad \quad + \frac12 \sum_{k=1}^m \OO^\top \diagsigma_{k}(s) \OO dW_k(s)
			\bigg)
			\big(\gamma^{-\frac12}(s)\nu^{-1}(s)\beta(s)\big) ,\\
			\gamma^{-\frac12}(0)\nu^{-1}(0)\beta(0) & = \gamma^{-\frac12}(0) d - \gamma^{\frac12}(0) x .
		\end{split}
	\end{equation*}
\end{lemma}

\begin{proof}
    Integration by parts, \eqref{eq:eqnuinv}, the fact that~$\nu$,~$\nu^{-1}$, and~$\gamma$ are continuous, the fact that~$\nu^{-1}$ and~$\nu$ are of finite variation, \eqref{eq:defnu}, and~\eqref{eq:def_deviation_pm} 
    imply for all $s \in [0,T]$ that 
	\begin{equation*}
		\begin{split}
			d(\nu^{-1}\beta)(s) & = \nu^{-1}(s) d\beta(s) + (d\nu^{-1}(s)) \beta(s) + d[\nu^{-1},\beta](s) \\
			& =  - \nu^{-1}(s) \big(d(\nu(s)\gamma(s))\big) X(s) - \rho(s) \nu^{-1}(s) \beta(s) ds \\
			& = - \nu^{-1}(s) \nu(s) (d\gamma(s)) X(s) 
			- \nu^{-1}(s) (d\nu(s)) \gamma(s) X(s) \\
			& \quad - \nu^{-1}(s) \big(d[\nu,\gamma](s)\big) X(s)
			- \rho(s) \nu^{-1}(s) \beta(s) ds \\
			& = - (d\gamma(s)) X(s) 
			- \nu^{-1}(s) \nu(s) \rho(s) \gamma(s) X(s) ds 
			- \rho(s) \nu^{-1}(s) \beta(s) ds \\
			& = - (d\gamma(s)) X(s) 
			- \rho(s) \big(\gamma(s) X(s) + \nu^{-1}(s) \beta(s)\big) ds .
		\end{split}
	\end{equation*}
	It follows by integration by parts for all $s\in[0,T]$ that 
	\begin{equation}\label{eq:1328a}
		\begin{split}
			& d\big(\gamma^{-\frac12}(s) \nu^{-1}(s) \beta(s) \big) \\
			& = \gamma^{-\frac12}(s) d \big(\nu^{-1}(s) \beta(s)\big) 
			+ \big(d\gamma^{-\frac12}(s)\big) \nu^{-1}(s) \beta(s) + d[\gamma^{-\frac12},\nu^{-1}\beta](s) \\
			& = -\gamma^{-\frac12}(s) (d\gamma(s)) X(s) - \gamma^{-\frac12}(s) \rho(s) \big(\gamma(s) X(s) + \nu^{-1}(s) \beta(s)\big) ds \\
			& \quad + \big(d\gamma^{-\frac12}(s)\big) \nu^{-1}(s) \beta(s) 
			- \big(d[\gamma^{-\frac12},\gamma](s)\big) X(s) .
		\end{split}
	\end{equation}
	Furthermore, it holds by integration by parts for all $s\in [0,T]$ that 
	\begin{equation}\label{eq:1328b}
		\begin{split}
			\big(d[\gamma^{-\frac12},\gamma](s)\big) X(s) 
			& = \big(d(\gamma^{-\frac12}(s) \gamma(s))\big)  X(s) - \gamma^{-\frac12}(s) (d\gamma(s))  X(s)  
            - \big(d\gamma^{-\frac12}(s)\big) \gamma(s)  X(s).
		\end{split}
	\end{equation}
	From~\eqref{eq:1328a} and~\eqref{eq:1328b} we obtain for all $s\in[0,T]$ that 
	\begin{equation}\label{eq:intermedresultcompbeta}
		\begin{split}
			& d\big(\gamma^{-\frac12}(s) \nu^{-1}(s) \beta(s)\big) \\ 
			& =  \big(- \gamma^{-\frac12}(s) \rho(s) ds + d\gamma^{-\frac12}(s) \big) 
			\big(\gamma(s) X(s) + \nu^{-1}(s) \beta(s)\big) 
			- \big(d\gamma^{\frac12}(s) \big)  X(s) \\
			& = 
			\Big(- \gamma^{-\frac12}(s) \rho(s) \gamma^{\frac12}(s) ds  + \big(d\gamma^{-\frac12}(s)\big) \gamma^{\frac12}(s) - \big(d\gamma^{\frac12}(s) \big) \gamma^{-\frac12}(s) \Big) \gamma^{\frac12}(s) X(s) \\
			& \quad +  \Big(- \gamma^{-\frac12}(s) \rho(s) \gamma^{\frac12}(s) ds + \big(d\gamma^{-\frac12}(s)\big) 
			\gamma^{\frac12}(s) \Big) 
			\gamma^{-\frac12}(s) \nu^{-1}(s) \beta(s)
			.
		\end{split}
	\end{equation}
	We have from \cref{cor:dlam_lam} 
    for all $s\in[0,T]$ that 
    \begin{align}\label{eq:1448b}
			\big(d\gamma^{-\frac12}(s)\big) \gamma^{\frac12}(s)
			& = \OO^\top \big(d\lam^{-\frac12}(s)\big) \lam^{\frac12}(s) \OO \\
			& = \OO^\top 
			\bigg(\! -\frac12 \diagmu(s) + \frac38 \sum_{k=1}^m \diagsigma_{k}(s) \diagsigma_{k}(s) \bigg) \OO ds 
			- \frac12 \sum_{k=1}^m \OO^\top \diagsigma_{k}(s) \OO dW_k(s) \nonumber
	\end{align}
    and 
	\begin{equation}\label{eq:1448c}
		\begin{split}
			\big(d\gamma^{\frac12}(s)\big) \gamma^{-\frac12}(s) 
			& = \OO^\top  \bigg(\frac12 \diagmu(s)  - \frac18 \sum_{k=1}^m \diagsigma_{k}(s) \diagsigma_k(s) \bigg) \OO ds  
			+ \frac12 \sum_{k=1}^m \OO^\top \diagsigma_{k}(s)  \OO  dW_k(s) .
		\end{split}
	\end{equation} 
	The claim follows from combining 
    \eqref{eq:intermedresultcompbeta}, 
	\eqref{eq:1448b}, and \eqref{eq:1448c}.
\end{proof}

\begin{proof}[Proof of \cref{lem:scH_hidden_dev}]
    Let 
    $\beta=(\beta(s))_{s\in[0,T]}$ be defined by \eqref{eq:defbeta} 
    and observe 
	that 
	\begin{equation}\label{eq:1633}
		\gamma^{\frac12}X=\gamma^{-\frac12} D^X - \gamma^{-\frac12} \nu^{-1} \beta.  
	\end{equation} 
    Replacing $\gamma^{\frac12}X$ in \Cref{lem:computations_for_SDE} by the right-hand side of \eqref{eq:1633} 
    and combining the result 
    with \eqref{eq:def_scH_u} proves that 
    \begin{equation*}
        \begin{split}
            & d\Big(\scH^{\gamma^{-\frac12} D^X}(s) - \gamma^{-\frac12}(s) \nu^{-1}(s) \beta(s) \Big) \\
            & = 
			\bigg( \OO^\top \Big( \frac12 \diagmu(s) - \frac18 \sum_{k=1}^m \diagsigma_{k}(s) \diagsigma_k(s) \Big) \OO ds 
            + \frac12 \sum_{k=1}^m \OO^\top \diagsigma_{k}(s) \OO dW_k(s)
			\bigg) \\ 
			& \quad \cdot \Big(\scH^{\gamma^{-\frac12} D^X}(s) - \gamma^{-\frac12}(s) \nu^{-1}(s) \beta(s)\Big), \quad s\in[0,T] ,
        \end{split}
    \end{equation*}
    and $\scH^{\gamma^{-\frac12} D^X}(0) - \gamma^{-\frac12}(0) \nu^{-1}(0) \beta(0) = 0$. 
    Since this is a linear SDE with bounded coefficients and start in $0$, 
    it follows that the process $\scH^{\gamma^{-\frac12} D^X} - \gamma^{-\frac12} \nu^{-1} \beta$
    is indistinguishable from the zero process (cf., for example, \cite[Theorem~3.17]{pardouxBook}).
	This and~\eqref{eq:1633} yield that 
	\begin{equation*}
		\scH^{\gamma^{-\frac12} D^X} = \gamma^{-\frac12} \nu^{-1} \beta = \gamma^{-\frac12} D^X - \gamma^{\frac12} X
		.
	\end{equation*}
	This completes the proof. 
\end{proof}

In the following lemma we clarify that the risk term 
is finite. 

\begin{lemma}\label{lem:risk_term_finite_expectation}
	Let 
    $X \in \cA^{pm}$. It then holds that 
	\begin{equation*}
		E\bigg[ \Big\lvert \int_0^T (X(s)-\zeta(s))^\top \risk(s) (X(s)-\zeta(s)) ds \Big\rvert \bigg] < \infty.
	\end{equation*}
\end{lemma}

\begin{proof}
	Recall from~\eqref{eq:def_Q} that $\mathscr{Q}= \gamma^{-\frac12} \risk \gamma^{-\frac12}$ and that $\mathscr{Q}$ is $dP\times ds$-a.e.\ bounded. 
	Therefore, there exists $c \in (0,\infty)$ such that 
	$dP\times ds$-a.e.\ we have that 
	\begin{equation}\label{eq:1513}
		\begin{split}
			\lvert (X-\zeta)^\top \risk (X-\zeta) \rvert
			& = \lvert (\gamma^{\frac12} X - \gamma^{\frac12}\zeta)^\top \mathscr{Q} (\gamma^{\frac12}X-\gamma^{\frac12}\zeta) \rvert
			\le c \lVert \gamma^{\frac12}X-\gamma^{\frac12}\zeta \rVert_F^2 .
		\end{split}
	\end{equation}
	Moreover, note that~\cref{lem:scH_hidden_dev} shows that 
		$\gamma^{\frac12} X = \gamma^{-\frac12} D^X - \scH^{\gamma^{-\frac12}D^X}$. 
	We thus obtain from \eqref{eq:1513} that 
	\begin{equation}\label{eq:1514}
		\begin{split}
			& E\bigg[ \Big\lvert \int_0^T (X(s)-\zeta(s))^\top \risk(s) (X(s)-\zeta(s)) ds \Big\rvert \bigg] \\
			& \le c \,
			E\bigg[ \int_0^T \lVert \gamma^{-\frac12}(s) D^X(s) - \scH^{\gamma^{-\frac12}D^X}(s) - \gamma^{\frac12}(s)\zeta(s) \rVert_F^2 \, ds \bigg] .
		\end{split}
	\end{equation}
	From \cref{rem:SDE_hidden_dev_and_new_state} we have that 
	\begin{equation*}
		E\bigg[ \int_0^T \big\lVert  \scH^{\gamma^{-\frac12}D^X}(s) \big\rVert_F^2 \, ds \bigg] 
		\le T \, E\bigg[ \sup_{s\in[0,T]} \big\lVert  \scH^{\gamma^{-\frac12}D^X}(s) \big\rVert_F^2 \bigg] < \infty .
	\end{equation*}
	Combining this, \eqref{eq:int_cond_zeta}, \eqref{eq:A1}, the triangle inequality, and \eqref{eq:1514} proves the claim. 
\end{proof}

\begin{proof}[Proof of \cref{lem:pm_cost_fct_as_LQ}]	
    Since $\kap$ is $dP\times ds$-a.e.\ bounded (cf.~\cref{rem:remark_on_setting_assumptions2}), we have from~\eqref{eq:A1} that  
	\begin{equation}\label{eq:1710}
		E\bigg[ \Big\lvert \int_0^T \big(\gamma^{-\frac12}(s) D^X(s)\big)^\top \kap(s) \gamma^{-\frac12}(s) D^X(s) ds \Big\rvert \bigg] < \infty .
	\end{equation}
    Furthermore, it follows from $X(T)=\xi$ and \Cref{lem:scH_hidden_dev} that 
	\begin{equation}\label{eq:1712a}
		\begin{split}
		(D^X(T))^\top \gamma^{-1}(T) D^X(T)
		& = (\gamma^{-\frac12}(T) D^X(T) )^\top \gamma^{-\frac12}(T) D^X(T) \\
		& = \big(\scH^{\gamma^{-\frac12} D^X }(T) + \gamma^{\frac12}(T) \xi\big)^\top \big(\scH^{\gamma^{-\frac12} D^X }(T) + \gamma^{\frac12}(T) \xi \big).
		\end{split}
	\end{equation}
	Moreover, we have that 
	$E[ \lvert (\scH^{\gamma^{-\frac12} D^X }(T))^\top \scH^{\gamma^{-\frac12} D^X }(T) \rvert]<\infty$
	due to \cref{rem:SDE_hidden_dev_and_new_state}. 
    Therefore, 
    the Cauchy--Schwarz inequality and the fact that $E[(\gamma^{\frac12}(T) \xi)^{\top} \gamma^{\frac12}(T) \xi]<\infty$ (cf.~\eqref{eq:int_cond_terminal_pos})
	ensure that 
	\begin{equation*}
		E\big[ \big\lvert \big(\scH^{\gamma^{-\frac12} D^X }(T) + \gamma^{\frac12}(T) \xi\big)^\top \big(\scH^{\gamma^{-\frac12} D^X }(T) + \gamma^{\frac12}(T) \xi \big) \big\rvert\big]<\infty .
	\end{equation*}
	This, \eqref{eq:1710}, \eqref{eq:1712a}, 
    \cref{lem:risk_term_finite_expectation}, and~\eqref{eq:def_costfct_pm} establish that 
    $\pmJ(X)$ is well defined and finite.  
	
	We have from 
	\eqref{eq:def_Q} and 
	\Cref{lem:scH_hidden_dev} 
	that 
	\begin{equation}\label{eq:1158}
		\begin{split}
			(X-\zeta)^\top \risk (X-\zeta) 
			& = \big( \gamma^{\frac12} X - \gamma^{\frac12} \zeta \big)^\top \mathscr{Q} \big( \gamma^{\frac12} X - \gamma^{\frac12} \zeta  \big)
			\\
			& = \big( \gamma^{-\frac12} D^X - \scH^{\gamma^{-\frac12} D^X} - \gamma^{\frac12} \zeta \big)^\top \mathscr{Q} \big( \gamma^{-\frac12} D^X - \scH^{\gamma^{-\frac12} D^X} - \gamma^{\frac12} \zeta \big)
			\\
			& = (\gamma^{-\frac12} D^X)^\top \mathscr{Q}  \gamma^{-\frac12} D^X
			- 2 (\gamma^{-\frac12} D^X)^\top \mathscr{Q} \big(\scH^{\gamma^{-\frac12} D^X} + \gamma^{\frac12} \zeta \big) \\
			& \quad + \big(\scH^{\gamma^{-\frac12} D^X} + \gamma^{\frac12} \zeta \big)^\top \mathscr{Q} \big(\scH^{\gamma^{-\frac12} D^X} + \gamma^{\frac12} \zeta \big) .
		\end{split}
	\end{equation}
	The assumption that $\mathscr{Q}$ is $dP\times ds$-a.e.\ bounded, \eqref{eq:int_cond_zeta}, \eqref{eq:A1}, 
    \cref{rem:SDE_hidden_dev_and_new_state}, and 
    \cref{lem:risk_term_finite_expectation} imply that 
	\begin{equation*}
		\begin{split}
			& E\bigg[ \bigg\lvert \int_0^T (\gamma^{-\frac12}(s) D^X(s))^\top 
			\mathscr{Q}(s) \gamma^{-\frac12}(s) D^X(s) ds \bigg\rvert \bigg] < \infty, \\ 
			& 
			E\bigg[ \bigg\lvert \int_0^T (\gamma^{-\frac12}(s) D^X(s))^\top \mathscr{Q}(s) \big(\scH^{\gamma^{-\frac12} D^X}(s) + \gamma^{\frac12}(s) \zeta(s) \big) ds \bigg\rvert  \bigg] 
			< \infty , \\
			& \text{and} \quad 
			E\bigg[ \bigg\lvert \int_0^T \big(\scH^{\gamma^{-\frac12} D^X}(s) + \gamma^{\frac12}(s) \zeta(s) \big)^\top \mathscr{Q}(s) \big(\scH^{\gamma^{-\frac12} D^X}(s) + \gamma^{\frac12}(s) \zeta(s) \big) ds \bigg\rvert \bigg] < \infty .
		\end{split}
	\end{equation*}
	The fact that $\KK = \mathscr{Q} +  \kap$, \eqref{eq:1710}, \eqref{eq:1712a}, \eqref{eq:1158},  and~\eqref{eq:def_costfct_pm} 
	hence show~\eqref{eq:cost_fct_shd}. 
\end{proof}

\begin{proof}[Proof of \cref{lem:bijective}]
	1. First, let $u\in\cL^2$, and define $X^u$ by $X^u(0-)=x$, $X^u(T)=\xi$, and 
	\begin{equation}\label{eq:1607b}
		X^u(s) = \gamma^{-\frac12}(s) (u(s)-\scH^u(s)), 
		\quad s\in [0,T).
	\end{equation}
    To show that $X^u \in \cA^{pm}$, note first that $u-\scH^u \in \cL^2$ due to $u\in\cL^2$ and \cref{rem:SDE_hidden_dev_and_new_state}. 
	Further, this, \eqref{eq:1607b}, and the fact that $\gamma^{-\frac12}$ has a.s.\ continuous paths yield that 
	\begin{equation*}
		\begin{split}
		&\int_0^T \lVert X^u(s) \rVert_F^2 ds  
        \leq \big(\textstyle{\sup_{s\in[0,T]}} \lVert \gamma^{-\frac12}(s) \rVert_F^2 \big) \int_0^T \lVert \gamma^{\frac12}(s) X^u(s) \rVert_F^2 ds 
		< \infty \text{ a.s.}
		\end{split}
	\end{equation*} 
    Let $\beta=(\beta(s))_{s\in[0,T]}$ be defined by \eqref{eq:defbeta} (using $X^u$).  
    From \cref{lem:computations_for_SDE},  \eqref{eq:1607b}, and \eqref{eq:def_scH_u} we obtain that 
    \begin{equation*}
		\begin{split}
			& d\big(\scH^u(s) - \gamma^{-\frac12}(s) \nu^{-1}(s) \beta(s) \big) \\
			& = 
			- \bigg( 
			\gamma^{-\frac12}(s) \rho(s) \gamma^{\frac12}(s) ds 
			+ 
			\OO^\top \Big( \frac12 \diagmu(s) - \frac38 \sum_{k=1}^m \diagsigma_{k}(s) \diagsigma_k(s) \Big) \OO ds \\
			& \qquad \quad + \frac12 \sum_{k=1}^m \OO^\top \diagsigma_{k}(s) \OO dW_k(s)
			\bigg) 
			\big(\scH^u(s) - \gamma^{-\frac12}(s) \nu^{-1}(s) \beta(s)\big), \quad s\in[0,T] ,
		\end{split}
	\end{equation*}
    and $\scH^u(0) - \gamma^{-\frac12}(0) \nu^{-1}(0) \beta(0) = 0$. 
    It follows that the process $\scH^{u} - \gamma^{-\frac12} \nu^{-1} \beta$
    is indistinguishable from the zero process (cf., for example, \cite[Theorem~3.17]{pardouxBook}). 
	We conclude that 
	\begin{equation}\label{eq:huequalsprodbeta}
		\scH^u(s) = \gamma^{-\frac12}(s) \nu^{-1}(s) 
		\beta(s), 
		\quad s \in [0,T]. 
	\end{equation}
    Let $D^{X^u}$ be defined by \eqref{eq:def_deviation_pm},    
	and note that \eqref{eq:defbeta} and \eqref{eq:huequalsprodbeta} demonstrate that 
	\begin{equation*}
    \begin{split}
		\gamma^{-\frac12}(s) D^{X^u}(s)
        & =\gamma^{\frac12}(s) X^u(s) + \gamma^{-\frac12}(s) \nu^{-1}(s) \beta(s)\\
        & = \gamma^{\frac12}(s) X^u(s) + \scH^u(s), \quad s\in[0,T].
    \end{split}
	\end{equation*} 
	Hence, \eqref{eq:1607b} shows 
    for all $s \in [0,T)$ that 
	$\gamma^{-\frac12}(s) D^{X^u}(s)= u(s)$. 
	Since $u\in\cL^2$, we conclude that~\eqref{eq:A1} holds and thus that $X^u\in \cA^{pm}$. 
	In particular, $\ol{\varphi}$ is well defined.
	Moreover, we have for all $s\in[0,T)$ that 
	$$(\varphi(\ol{\varphi}(u)))(s)=(\varphi(X^u))(s)=\gamma^{-\frac12}(s) D^{X^u}(s)=u(s).$$ 
	Hence, it holds $dP\times ds$-a.e.\ that 
	$\varphi(\ol{\varphi}(u))=u$. 		
	
	\smallskip
	
	2. Now, let $X \in \cA^{pm}$. 
	We want to show that $X=\ol{\varphi}(\varphi(X))$. 
	Since $X \in \cA^{pm}$, 
	it holds that $X(0-)=x=(\ol{\varphi}(\varphi(X)))(0-)$ and $X(T)=\xi=(\ol{\varphi}(\varphi(X)))(T)$. 
	Moreover, recall that $\varphi(X)=\gamma^{-\frac12} D^X$ and note that \Cref{lem:scH_hidden_dev} ensures 
    for all $s\in[0,T)$ that 
	\begin{equation*}
		\begin{split}
			(\ol{\varphi}(\varphi(X)))(s) 
			& = \gamma^{-\frac12}(s) ({\varphi(X)}(s) -\scH^{\varphi(X)}(s)) 
			= \gamma^{-\frac12}(s) \big(\gamma^{-\frac12}(s) D^X(s) - \scH^{\gamma^{-\frac12} D^X}(s) \big) \\
			& = \gamma^{-\frac12}(s) (\gamma^{\frac12}(s) X(s) ) 
			= X(s) .
		\end{split}
	\end{equation*}
	This completes the proof.	
\end{proof}

\begin{proof}[Proof of \cref{lem:convergence_SDE}]
	Let $\delta \scH^N = \scH^N - \scH$, $N\in\N$, and denote 
	$D^N = D^{X^N}$, $N\in\N$, $D=D^X$.  
	Moreover, for each $N\in\N$ let $\wt b^N\colon [0,T]\times \Omega \times \R^n \to \R^n$ and $\wt \sigma^N \colon [0,T] \times \Omega \times \R^n \to \R^{n\times m}$ be defined by 
	\begin{equation*}
		\begin{split}
			\wt b^N(s,y) & = \bigg( - \gamma^{-\frac12}(s) \rho(s) \gamma^{\frac12}(s) 
			- \OO^\top \Big( \diagmu(s) - \frac12 \sum_{k=1}^m \diagsigma_{k}(s) \diagsigma_k(s) \Big) 
			\OO \bigg) \\
			& \quad \cdot \Big(\gamma^{-\frac12}(s) D^N(s) - \gamma^{-\frac12}(s) D(s) \Big) 
			+ \OO^\top \Big( \frac12 \diagmu(s) - \frac18 \sum_{k=1}^m \diagsigma_{k} \diagsigma_{k}(s) \Big) \OO y, \\
			\wt \sigma^N(s,y) & = - \Big( \sum_{l=1}^n (\OO^\top \diagsigma_{k}(s) \OO)_{i,l} \big(\gamma^{-\frac12}(s) D^N(s) - \gamma^{-\frac12}(s) D(s) \big)_{l} \Big)_{(i,k)\in\{1,\ldots,n\}\times\{1,\ldots,m\}} \\
			& \quad + \frac12 \Big( \sum_{l=1}^n (\OO^\top \diagsigma_{k}(s) \OO)_{i,l} y_l \Big)_{(i,k)\in\{1,\ldots,n\}\times\{1,\ldots,m\}}
		\end{split}
	\end{equation*}
	for $s\in[0,T]$, $y\in\R^n$.
	We then have for all $N\in\N$ that 
	\begin{equation*}
		d (\delta \scH^N(s)) = \wt b^N(s,\delta \scH^N(s)) ds + \wt\sigma^N(s,\delta \scH^N(s)) dW_s, \quad s\in[0,T], \quad \delta\scH^N(0) = 0.
	\end{equation*}
	Recall that $\mu$ and $\sigma$ are $dP\times ds$-a.e.\ bounded. Therefore, there exists $c_1\in(0,\infty)$ such that 
	for all $N\in\N$ and $y,z\in\R^n$ it holds $dP\times ds$-a.e.\  that
	\begin{equation*}
		\begin{split}
			\lVert \wt b^N(s,y) - \wt b^N(s,z) \rVert_F + \lVert \wt \sigma^N(s,y) - \wt \sigma^N(s,z) \rVert_F
			& \le c_1 \lVert y-z \rVert_F .
		\end{split}
	\end{equation*}
	Note furthermore that boundedness of $\mu$, $\sigma$, and $\gamma^{-\frac12} \rho \gamma^{\frac12}$, the fact that $\gamma^{-\frac12} D^N - \gamma^{-\frac12} D \in\cL^2$, and Jensen's inequality establish that there exists $c_2\in(0,\infty)$ such that for all $N\in\N$ it holds that 
	\begin{equation}\label{eq:estimatetildebtildesigmazero}
		\begin{split}
			& E\bigg[ \bigg( \int_0^T \lVert \wt b^N(s,0) \rVert_F ds \bigg)^2 \bigg]
			+ 	E\bigg[ \int_0^T \lVert \wt \sigma^N(s,0) \rVert_F^2 ds \bigg] \\
			& \le c_2 E\bigg[ \int_0^T \big\lVert \gamma^{-\frac12}(s) D^N(s) - \gamma^{-\frac12}(s) D(s) \big\rVert_F^2 ds \bigg]
			< \infty.
		\end{split}
	\end{equation}
	It now follows from, for instance, \cite[Theorem~3.2.2]{zhang}, that there exists $c_3\in(0,\infty)$ such that for all $N\in\N$ it holds that 
	\begin{equation*}
		\begin{split}
			E\Big[ \sup_{s\in[0,T]} \lVert \delta \scH(s) \rVert_F^2 \Big]
			& \le c_3 E\bigg[ \bigg( \int_0^T \lVert \wt b^N(s,0) \rVert_F ds \bigg)^2 
			+  \int_0^T \lVert \wt \sigma^N(s,0) \rVert_F^2 ds \bigg] 
            .
		\end{split}
	\end{equation*}
	This, \eqref{eq:estimatetildebtildesigmazero}, and the assumption that 
    $\lim_{N\to\infty} \md(X,X^N)  = 0$ imply the claim. 
\end{proof}

The following technical lemma is used in the proof of \cref{thm:contextcostfct}.

\begin{lemma}\label{lem:convergence_helper}
    (i)
	Let $A=(A(s))_{s\in[0,T]}$ be an $\R^{n\times n}$-valued progressively measurable process such that $A$ is $dP\times ds$-a.e.\ bounded.
	Let $Y=(Y(s))_{s\in[0,T]} \in \cL^2$, $\wt Y=(\wt Y(s))_{s\in[0,T]} \in \cL^2$,  $Y^N=(Y^N(s))_{s\in[0,T]} \in \cL^2$, $N\in\N$, and $\wt Y^N=(\wt Y^N(s))_{s\in[0,T]} \in \cL^2$, $N\in\N$, and assume that $\lim_{N \to \infty} \lVert Y^N - Y \rVert_{\cL^2} = 0$ and $\lim_{N \to \infty} \lVert \wt Y^N - \wt Y \rVert_{\cL^2} = 0$.
	It then holds that 
	\begin{equation}\label{eq:1749}
		\lim_{N\to\infty} E\bigg[ \int_0^T \big\lvert  (Y^N(s))^\top A(s) \wt Y^N(s) - (Y(s))^\top A(s) \wt Y(s) \big\rvert ds \bigg] = 0 .
	\end{equation}

	(ii)
	Let $Y \in L^2(\Omega,\cF_T,P;\R^n)$, $Y^N \in  L^2(\Omega,\cF_T,P;\R^n)$, $N\in\N$, and assume that 
	$\lim_{N \to \infty} \lVert Y^N - Y \rVert_{L^2} = 0$. 
	It then holds that 
	$\lim_{N\to\infty} E[\lvert  (Y^N)^\top Y^N - Y^\top Y \rvert ] = 0$.
\end{lemma}

\begin{proof}
    (i) 
    Note that for all 
    $y,\wt y,z,\wt z\in\R^n$ and $\wt A \in \R^{n\times n}$ it holds that 
    \begin{align}\label{eq:1924new}
            \lvert  z^\top \wt A \wt z - y^\top \wt A \wt y \rvert 
            & \le \sum_{i=1}^n \sum_{j=1}^n \lvert \wt A_{i,j} \rvert 
            \lvert z_i \wt z_j - y_i \wt y_j \lvert \\ 
            & \le \Big(\max_{i,j} \lvert \wt A_{i,j} \rvert \Big)
            \bigg[ \bigg(  
            \sum_{i=1}^n \lvert z_i \rvert \bigg) \bigg(
            \sum_{j=1}^n \lvert \wt z_j - \wt y_j \rvert \bigg)
            + \bigg( \sum_{j=1}^n \lvert \wt y_j \rvert \bigg) \bigg(
            \sum_{i=1}^n \lvert z_i - y_i \rvert \bigg)
            \bigg] . \nonumber
    \end{align}
    Furthermore, observe that the Cauchy--Schwarz inequality and Jensen's inequality show that 
    \begin{equation*}
        \begin{split}
            & \lim_{N\to\infty}  E\bigg[ \int_0^T \bigg(  
            \sum_{i=1}^n \lvert (Y^N(s))_i \rvert \bigg) \bigg(
            \sum_{j=1}^n \lvert (\wt Y^N(s))_j - (\wt Y(s))_j \rvert \bigg) ds \bigg] \\
            & \le 
            \lim_{N\to\infty}
            n
            \bigg( E\bigg[ \int_0^T \lVert Y^N(s) \rVert_F^2  ds \bigg] \bigg)^{\frac12}
            \bigg(
            E\bigg[ \int_0^T \lVert \wt  Y^N(s) - \wt Y(s) \rVert_F^2  ds \bigg]
            \bigg)^{\frac12} 
            = 0,
        \end{split}
    \end{equation*}
    where we have used the assumption $\lim_{N\to\infty} E[\int_0^T \lVert \wt Y(s) - \wt Y^N(s) \rVert_F^2 ds] = 0$ and that 
    $\sup_{N\in\N} E[ \int_0^T \lVert Y^N(s) \rVert_F^2 ds ] <\infty$ due to $\lim_{N\to\infty} E[\int_0^T \lVert Y(s) - Y^N(s) \rVert_F^2 ds] = 0$. 
    The second term arising from \eqref{eq:1924new} can be treated similarly.
    Since $A$ is $dP\times ds$-a.e.\ bounded, we thus obtain~\eqref{eq:1749}.
    
    (ii)
	This can be proven similar to part (i).
\end{proof}

The following lemma is a variant of \cite[Lemma 6.4]{ackermann2022reducing} and thus relies on Lemma~2.7 in Section~3.2 of Karatzas \& Shreve~\cite{karatzasshreve}.
In the form of \cref{cor:KS_approx_seq_finite_var_O} below it is used to construct an approximating sequence of finite-variation strategies in the proof of \cref{thm:contextcostfct}.

\begin{lemma}\label{lem:KS_approx_seq_finite_var}
	Let $L=(L(s))_{s\in[0,T]}$ be the $\R^{n\times n}$-valued process defined by 
	\begin{equation}\label{eq:def_L_for_helper_process_Z}
	\begin{split}
	L(r) & = 
	\int_0^r \frac18 \sum_{k=1}^m \diagsigma_{k}(s) \diagsigma_{k}(s) ds 
	- \sum_{k=1}^m \int_0^r \frac12 \diagsigma_{k}(s) dW_k(s), 
	\quad r \in [0,T], 
	\end{split}
	\end{equation}
	and let $K=(K(s))_{s\in[0,T]}$ be the unique solution of the SDE 
	\begin{equation}\label{eq:def_helper_process_Z}
	dK(s) = (dL(s)) K(s), \quad s\in[0,T], \quad K(0)=I_n .
	\end{equation}
	Let 
    $u=(u(s))_{s\in[0,T]} \in \cL^2$. 
	Then there exists a sequence of $\R^n$-valued bounded c\`adl\`ag finite-variation processes $(v^N)_{N\in\N}$ such that 
	\begin{equation*}
	   \lim_{N\to\infty} E\bigg[ \int_0^T \lVert u(s) - K(s) v^N(s) \rVert_F^2 \, ds \bigg]
	   = 0 .
	\end{equation*}
	Moreover, it holds for all $N\in\N$ that $Kv^N$ is an $\R^n$-valued c\`adl\`ag semimartingale and $E[\sup_{s \in [0,T]} \lVert K(s) v^N(s) \rVert_F^2 ] < \infty$. 
\end{lemma}

\begin{proof}
	Note that the process $L$ takes values in the set of $n\times n$ diagonal matrices.
	This and~\eqref{eq:def_helper_process_Z} show that for all $i,j\in\{1,\ldots,n\}$, $r \in [0,T]$ we have that
	\begin{equation*}
	K_{j,i}(r) = 
	\delta_{j,i}
	+ \int_0^r K_{j,i}(s) dL_{j,j}(s) ,
	\end{equation*} 
	where $\delta_{j,i}=1$ if $i=j$ and $\delta_{i,j}=0$ else. 
	We conclude that $K$ takes values in the set of $n\times n$ diagonal matrices, too, 
	and that for all $j\in\{1,\ldots,n\}$ we have that 
	$K_{j,j}$ is the stochastic exponential of $L_{j,j}$.  
	By a straightforward adaptation of the proof of \cite[Lemma~6.4]{ackermann2022reducing}, we can show that for every $j\in\{1,\ldots,n\}$ there exists a sequence of one-dimensional bounded c\`adl\`ag finite-variation processes $(v_j^N)_{N\in\N}$ such that 
	\begin{equation*}
	\lim_{N\to\infty} E\bigg[ \int_0^T \big( u_j(s) - K_{j,j}(s) v^N_j(s) \big)^2 \, ds \bigg]
	= 0
	\end{equation*}
	and such that for all $N\in \N$ it holds that 
	$K_{j,j}v_j^N$ is a c\`adl\`ag semimartingale with $E[\sup_{s \in [0,T]} \lvert K_{j,j}(s) v_j^N(s) \rvert^2 ] < \infty$.
	Finally, let $v^N=(v_1^N,\ldots,v_n^N)^\top$, $N\in\N$, and observe that the sequence $(v^N)_{N\in\N}$ has the desired properties. 
\end{proof}

By applying \cref{lem:KS_approx_seq_finite_var} to the process $\OO u \in \cL^2$ and using the invariance of the norm $\lVert \cdot \rVert_F$ with respect to orthogonal transformations, we obtain the following corollary of \cref{lem:KS_approx_seq_finite_var}.

\begin{corollary}\label{cor:KS_approx_seq_finite_var_O}
	Let $L=(L(s))_{s\in[0,T]}$ be defined by \eqref{eq:def_L_for_helper_process_Z} and let 
	$K=(K(s))_{s\in[0,T]}$ be defined by \eqref{eq:def_helper_process_Z}. 
	Let 
    $u=(u(s))_{s\in[0,T]} \in \cL^2$. 
	Then there exists a sequence of $\R^n$-valued bounded c\`adl\`ag finite-variation processes $(v^N)_{N\in\N}$ such that for the sequence of $\R^n$-valued c\`adl\`ag semimartingales $(u^N)_{N\in\N}$ defined by $u^N = \OO^\top K v^N$, $N\in\N$, it holds that 
	\begin{equation}\label{eq:1349a}
	\lim_{N\to\infty} E\bigg[ \int_0^T \lVert u(s) - u^N(s) \rVert_F^2 \, ds \bigg]
	= 0 
	\end{equation}
	and 
    \begin{equation}\label{eq:1349b}
     E\Big[\sup_{s \in [0,T]} \lVert u^N(s) \rVert_F^2 \Big] < \infty, \quad N\in\N.   
    \end{equation}
\end{corollary}

The next lemma is employed in the proof of \cref{thm:contextcostfct} to establish that the constructed sequence of strategies is a sequence of finite-variation strategies.

\begin{lemma}\label{lem:dynamics_of_X_when_u_semimart}
		Assume that $u \in\cL^2$ is a c\`adl\`ag semimartingale and let $X^u=\ol{\varphi}(u) \in \cA^{pm}$ (cf.\ \cref{lem:bijective}).  
		Then $X^u$ is a c\`adl\`ag semimartingale and for all $s \in [0,T)$ it holds 
		\begin{equation*}
		\begin{split}
		dX^u(s) & = \gamma^{-\frac12}(s) du(s)
		+ \frac12 \sum_{k=1}^m \gamma^{-\frac12}(s) \OO^\top \diagsigma_k(s) \OO u(s) dW_k(s) \\
		& \quad 
		+ d[\gamma^{-\frac12},u-\scH^u](s)
		- \frac14 \sum_{k=1}^m \gamma^{-\frac12}(s) 
		\OO^\top \diagsigma_{k}(s) \diagsigma_{k}(s) \OO \scH^u(s) ds  
		 \\
		& \quad 
		+  \gamma^{-\frac12}(s) 
		\bigg(\gamma^{-\frac12}(s) \rho(s) \gamma^{\frac12}(s)
		+ \OO^\top 
		\Big( \frac12 \diagmu(s) - \frac18 \sum_{k=1}^m \diagsigma_{k}(s) \diagsigma_{k}(s) \Big)
		\OO \bigg) u(s) ds 
		.
		\end{split}
		\end{equation*}
\end{lemma}

\begin{proof}
	Let $\scH^u$ be the solution of the SDE~\eqref{eq:def_scH_u} associated to $u$ and recall that $X^u(0-)=x$, $X^u(T)=\xi$, and 
	$X^u(s)=\gamma^{-\frac12}(s) (u(s)-\scH^u(s))$, $s\in[0,T)$
	(cf.\ \cref{lem:bijective}). 
	Since $u$ is a c\`adl\`ag semimartingale and $\gamma^{-\frac12}$ and $\scH^u$ are continuous semimartingales, we have that also $X^u$ is a c\`adl\`ag semimartingale. 
	Integration by parts implies for all $s\in[0,T)$ that 
	\begin{equation}\label{eq:1213a}
	\begin{split}
	dX^{u}(s) & = \gamma^{-\frac12}(s) d(u(s)-\scH^{u}(s))
	+ (d\gamma^{-\frac12}(s)) (u(s) - \scH^{u}(s)) 
	   + d[\gamma^{-\frac12},u-\scH^u](s) .
	\end{split}
	\end{equation}
	Further, \eqref{eq:def_scH_u} demonstrates 
    for all $s \in [0,T]$ that 
	\begin{equation}\label{eq:1213b}
	\begin{split}
	& \gamma^{-\frac12}(s) d(u(s) - \scH^{u}(s)) \\
	& = \gamma^{-\frac12}(s) du(s) 
	- \gamma^{-\frac12}(s) 
	\OO^\top \bigg(\frac12 \diagmu(s) - \frac18 \sum_{k=1}^m \diagsigma_{k}(s) \diagsigma_{k}(s) \bigg) \OO \scH^u(s) ds  \\
	& \quad 
	+  \gamma^{-\frac12}(s) 
	\bigg(\gamma^{-\frac12}(s) \rho(s) \gamma^{\frac12}(s)
	+ \OO^\top 
	\Big( \diagmu(s) - \frac12 \sum_{k=1}^m \diagsigma_{k}(s) \diagsigma_{k}(s) \Big)
	\OO \bigg) u(s) ds \\
	& \quad 
	+ \sum_{k=1}^m \gamma^{-\frac12}(s) \OO^\top \diagsigma_{k}(s) \OO u(s) dW_k(s) 
	- \frac12 \sum_{k=1}^m \gamma^{-\frac12}(s) \OO^\top \diagsigma_{k}(s) \OO \scH^{u}(s) dW_k(s)
	.
	\end{split}
	\end{equation}
	Next, by using $\gamma^{-\frac12}=\OO^\top \lam^{-\frac12} \OO$, the dynamics in~\eqref{eq:funct_of_lambda_dyn}, and the fact that $\OO\OO^\top=I_n$,  
	we obtain 
    for all $s \in [0,T]$ that 
	\begin{equation}\label{eq:1213c}
	\begin{split}
	(d\gamma^{-\frac12}(s)) (u(s) & - \scH^{u}(s)) 
	= \OO^\top \big( d\lam^{-\frac12}(s) \big) \OO  (u(s) - \scH^{u}(s) ) \\
	& = 
	\OO^\top \lam^{-\frac12}(s) \bigg( -\frac12 \diagmu(s) + \frac38 \sum_{k=1}^m \diagsigma_{k}(s)  (\diagsigma_{k}(s) )^\top  \bigg) \OO (u(s) - \scH^{u}(s) ) ds \\
	& \quad 
	- \frac12 \sum_{k=1}^m \OO^\top \lam^{-\frac12}(s) \diagsigma_{k}(s) \OO (u(s) - \scH^{u}(s) ) dW_k(s)  \\
	& = 
	\gamma^{-\frac12}(s) \OO^\top \bigg( -\frac12 \diagmu(s) + \frac38 \sum_{k=1}^m \diagsigma_{k}(s)  (\diagsigma_{k}(s) )^\top  \bigg) \OO (u(s) - \scH^{u}(s) ) ds \\
	& \quad 
	- \frac12 \sum_{k=1}^m \gamma^{-\frac12}(s) \OO^\top \diagsigma_{k}(s) \OO (u(s) - \scH^{u}(s)) dW_k(s) .
	\end{split}
	\end{equation}
	Combining \eqref{eq:1213a}, \eqref{eq:1213b}, and \eqref{eq:1213c} 
    establishes the claim. 
\end{proof}

We now use the preceding results to prove \cref{thm:contextcostfct}. 

\begin{proof}[Proof of \cref{thm:contextcostfct}]
	(i)
	For $N\in\N$ we denote by $D^N$ the deviation process that is associated to $X^N$ via~\eqref{eq:def_deviation_pm} and we denote by $\scH^N$ the solution of the SDE~\eqref{eq:def_scH_u} that is associated to $\gamma^{-\frac12} D^N$. 
	Moreover, we denote the deviation associated to $X$ by $D$ and the solution of the SDE~\eqref{eq:def_scH_u} associated to $\gamma^{-\frac12} D$ by $\scH$. 
	It follows from \Cref{lem:pm_cost_fct_as_LQ} 
    for all $N\in\N$ that 
	\begin{equation*}
		\begin{split}
			& \lvert \pmJ(X^N) - \pmJ(X) \rvert \\
			& = \bigg\lvert
			\frac12 E\big[ \big(\scH^{N}(T) + \gamma^{\frac12}(T) \xi \big)^{\!\top} \! \big(\scH^{N}(T) + \gamma^{\frac12}(T) \xi \big) \! - \! \big(\scH(T) + \gamma^{\frac12}(T) \xi \big)^{\!\top}\! \big(\scH(T) + \gamma^{\frac12}(T) \xi \big) \big] \\
			& \quad + 
			E\bigg[ \int_0^T (\gamma^{-\frac12}(s) D^N(s))^\top 
			\KK(s) 
			\gamma^{-\frac12}(s) D^N(s) - (\gamma(s)^{-\frac12} D(s))^\top 
			\KK(s) 
			\gamma^{-\frac12}(s) D(s) ds \bigg] \\
			& \quad  
			+ E\bigg[ \int_0^T \big(\scH^{N}(s) + \gamma^{\frac12}(s) \zeta(s) \big)^\top \mathscr{Q}(s) \big(\scH^{N}(s) + \gamma^{\frac12}(s) \zeta(s) \big) ds \bigg] \\
			& \quad  
			- E\bigg[ \int_0^T \big(\scH(s) + \gamma^{\frac12}(s) \zeta(s) \big)^\top \mathscr{Q}(s) \big(\scH(s) + \gamma^{\frac12}(s) \zeta(s) \big) ds \bigg] \\
			& \quad - 2 E\bigg[ \int_0^T (\gamma^{-\frac12}(s) D^N(s))^\top \mathscr{Q}(s) \big(\scH^{N}(s) + \gamma^{\frac12}(s) \zeta(s) \big) ds \bigg] \\
			& \quad + 2 E\bigg[ \int_0^T (\gamma^{-\frac12}(s) D(s))^\top \mathscr{Q}(s) \big(\scH(s) + \gamma^{\frac12}(s) \zeta(s) \big) ds \bigg] 
			\bigg\rvert
			.
		\end{split}
	\end{equation*} 	
	This, the triangle inequality, and Jensen's inequality imply that there exists $c_1\in(0,\infty)$ such that for all $N\in\N$ it holds  
	\begin{equation}\label{eq:1434}
		\begin{split}
			& c_1 \lvert \pmJ(X^N) - \pmJ(X) \rvert \\
			& \leq E\big[ \lvert \big(\scH^{N}(T) + \gamma^{\frac12}(T) \xi \big)^{\!\top} \! \big(\scH^{N}(T) + \gamma^{\frac12}(T) \xi \big) \! - \! \big(\scH(T) + \gamma^{\frac12}(T) \xi \big)^{\!\top}\! \big(\scH(T) + \gamma^{\frac12}(T) \xi \big) \rvert \big] \\
			& \quad +  E\bigg[ \int_0^T \big\lvert (\gamma^{-\frac12}(s) D^N(s))^\top 
			\KK(s) 
			\gamma^{-\frac12}(s) D^N(s) - (\gamma^{-\frac12}(s) D(s))^\top 
			\KK(s) 
			\gamma^{-\frac12}(s) D(s) \big\rvert ds \bigg] \\
			& \quad + E\bigg[ \int_0^T \big\lvert 
			\big(\scH^{N}(s) + \gamma^{\frac12}(s) \zeta(s) \big)^\top \mathscr{Q}(s) \big(\scH^{N}(s) + \gamma^{\frac12}(s) \zeta(s) \big) \\
			& \qquad \qquad \qquad - \big(\scH(s) + \gamma^{\frac12}(s) \zeta(s) \big)^\top \mathscr{Q}(s) \big(\scH(s) + \gamma^{\frac12}(s) \zeta(s) \big)
			\big\rvert \, ds \bigg] \\
			& \quad + E\bigg[ \int_0^T \big\lvert 
			(\gamma^{-\frac12}(s) D^N(s))^\top \mathscr{Q}(s) \big(\scH^{N}(s) + \gamma^{\frac12}(s) \zeta(s) \big) \\
			&  \qquad \qquad \qquad - (\gamma^{-\frac12}(s) D(s))^\top \mathscr{Q}(s) \big(\scH(s) + \gamma^{\frac12}(s) \zeta(s) \big)
			\big\rvert \, ds \bigg] 
			.
		\end{split}
	\end{equation}
	Note that for all $N\in\N$ it holds that 
	$\gamma^{-\frac12} D^N$, $\gamma^{-\frac12} D$, $\scH^N$, $\scH$, and $\gamma^{\frac12} \zeta$ are in $\cL^2$ (cf.~\eqref{eq:A1}, \cref{rem:SDE_hidden_dev_and_new_state}, and \eqref{eq:int_cond_zeta}) and that $\scH^N(T)$, $\scH(T)$, and $\gamma^{\frac12}\xi$ are in $L^2(\Omega ,\cF_T,P;\R^n)$ (cf.~\cref{rem:SDE_hidden_dev_and_new_state} and \eqref{eq:int_cond_terminal_pos}). 
	Moreover, note that $\KK$ and $\mathscr{Q}$ both are 
    $dP\times ds$-a.e.\ bounded, 
	that $\lim_{N\to\infty} \lVert \gamma^{-\frac12}D^N - \gamma^{-\frac12} D \rVert_{\cL^2} = 0$ by assumption, and that \Cref{lem:convergence_SDE} implies that 
	\begin{equation*}
		\begin{split}
		\lim_{N\to\infty} \lVert \big(\scH^{N} + \gamma^{\frac12}\zeta \big) 
		- \big(\scH + \gamma^{\frac12} \zeta \big) \rVert_{\cL^2}^2 
		& \leq T \lim_{N\to\infty}  E[\sup_{s\in[0,T]} \lVert \scH^N(s) - \scH(s) \rVert_F^2 ] = 0 
		\end{split}
	\end{equation*}
	and $\lim_{N\to\infty} \lVert \big(\scH^{N}(T) + \gamma^{\frac12}(T)\xi \big) 
	- \big(\scH(T) + \gamma^{\frac12}(T) \xi \big) \rVert_{L^2}^2 = 0$. 
	Therefore, \Cref{lem:convergence_helper} and \eqref{eq:1434} establish the convergence 
    $\lim_{N\to\infty} \lvert \pmJ(X^N) - \pmJ(X) \rvert = 0$ 
	in part (i).
	
	\smallskip
	
	(ii) 
	Let $X \in \cA^{pm}$ and $u=\varphi(X)=\gamma^{-\frac12}D^X$, 
    where $D^X$ is defined in \eqref{eq:def_deviation_pm}.
	Let $L=(L(s))_{s\in[0,T]}$ be defined by~\eqref{eq:def_L_for_helper_process_Z} and let $K=(K(s))_{s\in[0,T]}$ be defined by~\eqref{eq:def_helper_process_Z}. 
	Then, \cref{cor:KS_approx_seq_finite_var_O}
	ensures that there exists a sequence of $\R^n$-valued bounded c\`adl\`ag finite-variation processes $(v^N)_{N\in\N}$ 
    such that the sequence 
    of $\R^n$-valued c\`adl\`ag semimartingales $(u^N)_{N\in\N}$ defined by 
	\begin{equation}\label{eq:defuN}
		u^N=\OO^\top K v^N, \quad N\in\N, 
	\end{equation}
    satisfies \eqref{eq:1349a} and \eqref{eq:1349b}. 
	In particular, it holds for all $N\in\N$ that $u^N\in\cL^2$. 
	Let 
	$$X^{N}=\ol{\varphi}(u^N), \quad N \in\N.$$ 
	Note that by \cref{lem:bijective} it holds for all $N\in\N$ that $X^N$ is an element of $\cA^{pm}$ and that 
    $u^N=\varphi(X^N)=\gamma^{-\frac12}D^{X^N}$, where $D^{X^N}$ is defined in \eqref{eq:def_deviation_pm}. 
	Observe that this, the fact that $u=\gamma^{-\frac12}D^X$, \cref{rem:metric_and_cLt2}, and \eqref{eq:1349a} imply that 
	\begin{equation*}
		\lim_{N\to\infty}\md(X,X^N)
		= \lim_{N\to\infty} \lVert u - u^N \rVert_{\cL^2}
		= 0 .
	\end{equation*}
	Next, note that \eqref{eq:1349b}, the fact that for all $N\in\N$ it holds 
    $u^N=\gamma^{-\frac12}D^{X^N}$, and the fact that $\sigma$ is $dP\times ds$-a.e.\ bounded prove that 
	the integrability condition~\eqref{eq:A2} is satisfied for all $N\in\N$. 
	To show that for all $N\in\N$ the process $X^N$ has finite variation, note first that 	
	\cref{lem:dynamics_of_X_when_u_semimart}
	ensures that for all $N\in\N$ there exists a process $A^N=(A^N(s))_{s\in[0,T]}$ of finite variation such that for all $s\in[0,T)$ it holds that 
	\begin{equation}\label{eq:1213d}
		\begin{split}
			dX^{N}(s) & = dA^N(s) + \gamma^{-\frac12}(s) du^N(s) 
			+ \frac12 \sum_{k=1}^m \gamma^{-\frac12}(s) \OO^\top \diagsigma_{k}(s) \OO u^N(s) dW_k(s) .
		\end{split}
	\end{equation}
	Second, integration by parts
    and~\eqref{eq:defuN}
    with the continuous process~$K$ given by \eqref{eq:def_helper_process_Z} and \eqref{eq:def_L_for_helper_process_Z} and with the c\`adl\`ag finite-variation process~$v^N$  
	imply 
	for all $N\in\N$, $s\in[0,T]$ that 
	\begin{equation*}
		\begin{split}
			du^N(s) 
			& = \OO^\top K(s) dv^N(s)  + \OO^\top (dL(s)) K(s) v^N(s) \\ 
			& = \OO^\top K(s) dv^N(s) 
			\! + \! \sum_{k=1}^m \! \frac18 \OO^\top \diagsigma_k(s) \diagsigma_{k}(s) \OO u^N(s) ds 
            - \!\sum_{k=1}^m\! \frac12 \OO^\top \diagsigma_{k}(s) \OO u^N(s) dW_k(s) .
		\end{split}
	\end{equation*}
	We combine this, \eqref{eq:1213d}, and the fact that for all $N\in\N$ it holds that $v^N$ has finite variation 
	to conclude that for all $N\in\N$ the c\`adl\`ag semimartingale $X^N$ has finite variation. 
	This proves for all $N\in\N$ that $X^N \in \cA^{fv}$.
	
	\smallskip
	
	(iii)
	Note that \eqref{eq:inf_equal} follows from (i), (ii), and \cref{cor:fvsubsetpm}. 
\end{proof}

\subsection{Proofs for \cref{sec:soln}}

\begin{proof}[Proof of \cref{propo:soln_LQ_Sun}]
	First, note that \cref{assump_filtration} ensures that $(\cF_s)_{s\in[0,T]}$ is the augmented natural filtration of the Brownian motion $W$. 
	Furthermore, \cref{rem:remark_on_setting_assumptions2} 
	and the fact that the terminal value of the BSDE, $\tfrac12 I_n$, is symmetric, show that the assumptions (A1)' and (A2) in Sun et al.\ \cite{sun2021indefiniteLQ} are satisfied. 
	Since we moreover assume that \cref{assump_convexity} is met, we can apply \cite[Theorem 9.1]{sun2021indefiniteLQ}. 
	This proves (i) and (ii). 
	Furthermore, we can show that the optimal costs are given by~\eqref{eq:optcostsLQ} (cf.\ Section~5 and the introduction of Section~6 in~\cite{sun2021indefiniteLQ}), which proves (iii).
\end{proof}

\begin{proof}[Proof of \cref{lem:closing_immed}]
	Note that all assumptions of \cref{cor:soln_opt_trade_execution} are satisfied. 
	Hence, there exist a unique solution $(\mathscr{Y},\mathscr{Z})$ of the BSDE~\eqref{eq:BSDE}, a unique solution $\scH^*$ of the SDE~\eqref{eq:SDE_optimal_state}, and a unique optimal strategy $X^*$ in 
    $\cA^{pm}$ for  
    $\pmJ$, which is given by \eqref{eq:def_opt_strat} with associated deviation $D^*$ satisfying \eqref{eq:def_opt_dev}. 
	
	1. 
	Suppose first that $d=\gamma(0) x$.  
	It then holds that $\scH^*(0)=\gamma^{-\frac12}(0)d - \gamma^{\frac12}(0)x = 0$. This and~\eqref{eq:SDE_optimal_state} show that $\scH^*\equiv 0$. 
	It thus follows from \eqref{eq:def_opt_strat} that 
	$X^*(0-)=x$, $X^*(s)=0$, $s\in[0,T]$.
	Moreover, the fact that $\scH^*\equiv 0$ and \eqref{eq:def_opt_dev} imply that $D^*(s)=0$, $s\in[0,T)$. We obtain that \eqref{eq:A2} is satisfied. Combined with the fact that $X^* \in \cA^{pm}$ has finite variation, this shows that $X^* \in \cA^{fv}$. 
	
	2. 
	Suppose now that $\rho\equiv 0$. 
	Observe that if $\mathscr{Y}\equiv \frac12 I_n$ and $\mathscr{Z}\equiv 0$, then the driver~$g$ in~\eqref{eq:driver_of_BSDE} of the BSDE~\eqref{eq:BSDE} equals $0$. 
	Hence, $(\mathscr{Y},\mathscr{Z})=(\tfrac12 I_n,0)$ is a solution of the BSDE~\eqref{eq:BSDE}. 
	Therefore, we have that $\theta\equiv I_n$ in~\eqref{eq:def_theta}. 
	We hence obtain from \eqref{eq:def_opt_strat} that 
	$X^*(0-)=x$, $X^*(s)=0$, $s\in[0,T]$. 
	Further, \eqref{eq:def_opt_dev} shows that 
	$D^*(s)=\gamma^{\frac12}(s) \scH^*(s)$, $s\in[0,T)$. 
	This and the fact that $\sigma$ is $dP\times ds$-a.e.\ bounded ensure that there exists $c\in (0,\infty)$ such that 
	\begin{equation*}
		\begin{split}
			& \sum_{k=1}^m E\bigg[ \bigg( \int^T_0 \Big\lvert \big( \gamma^{-\frac12}(s) D^X(s)\big)^\top \OO^\top \diagsigma_{k}(s) \OO \gamma^{-\frac12}(s) D^X(s) \Big\rvert^2 ds \bigg)^{\frac12} \bigg]
			\\
			& \le 
			\sum_{k=1}^m E\bigg[ \bigg( \int^T_0 c^2 \lVert \scH^*(s) \rVert_F^4 ds \bigg)^{\frac12} \bigg]
			\le  mc T^{\frac12}
			E\Big[ \sup_{s\in[0,T]} \lVert \scH^*(s) \rVert_F^2 \Big] 
			.
		\end{split}
	\end{equation*}
	Therefore, \cref{rem:SDE_hidden_dev_and_new_state} guarantees that~\eqref{eq:A2} is satisfied. Since $X^*$ moreover has finite variation, we conclude that  $X^* \in \cA^{fv}$. 
\end{proof}

\vspace{-1cm}

\section*{}
\addcontentsline{toc}{section}{References}
\bibliographystyle{abbrv}
\bibliography{literature}

\appendix

\setcounter{equation}{0}
\renewcommand{\theequation}{\thesection.\arabic{equation}}

\section{On the price impact and the resilience}
\label{sec:appendixontheprice}

In this section we comment on our set-up in \cref{sec:multi_asset_fv}. 
We first justify our assumption of a symmetric price impact by the following example. 

\begin{ex}\label{ex:gamma_sym}
	Let $\wt\gamma \in \R^{n\times n}$ be non-symmetric; without loss of generality, assume that $\wt\gamma_{1,2}\ne \wt\gamma_{2,1}$. 	
	Moreover, let 
	$d=0=x$.
	We show that in this setting we can produce arbitrarily large negative costs in~\eqref{eq:def_pathwisecosts}. 
	The approach is to, for $N\in\N$, quickly execute the following sequence of trades: 
	\begin{itemize}
		\item[1.]
		Buy~$N$ shares in the first asset.
		
		\item[2.] 
		Buy (or sell) one share in the second asset.
		
		\item[3.] 
		Sell~$N$ shares in the first asset.
		
		\item[4.] 
		Sell (or buy) one share in the second asset. 
	\end{itemize}
	Denote for $j\in\{1,\ldots,n\}$ by $e_j \in \R^n$ the $j$-th unit vector in $\R^n$. 
	Let 
	$$a=\sgn( e_1^\top (\wt\gamma - \wt\gamma^\top)  e_2 ).$$ 
	For every $N\in\N$, $h \in (0,T/3]$ let $X^{h,N}=(X^{h,N}(s))_{s\in[0,T]}$ 
	be defined by $X^{h,N}(0-)=0$ and   
	\begin{equation*}
		\begin{split}
			X^{h,N}(s) 
			& = 
			N e_1 
			1_{[0,T]}(s)
			+ 
			a e_2
			1_{[h,T]}(s) 
			- N e_1 
			1_{[2h,T]}(s)
			-a e_2
			1_{[3h,T]}(s), \quad s \in [0,T] .
		\end{split}
	\end{equation*}
	Note that these are c\`adl\`ag functions of finite variation. 
	For every $N\in\N$, $h \in (0,T/3]$ it follows from, for example, \cref{rem:soln_for_deviation_SDE} (note that \cref{rem:soln_for_deviation_SDE} is still valid with $\gamma$ replaced by $\wt\gamma$) that the process $D^{h,N}=(D^{h,N}(s))_{s\in[0,T]}$ associated to $X^{h,N}$ via~\eqref{eq:deviationdynmultivariate}  
	is given by $D^{h,N}(0-)=0$ and  
	\begin{equation*}
		\begin{split}
			D^{h,N}(s) & = \nu^{-1}(s) \sum_{k=0}^3 \nu(kh) \wt\gamma \Delta X^{h,N}(kh) 1_{[kh,T]}(s) , 
			\quad s\in[0,T].
		\end{split}
	\end{equation*}
	Next, note that it holds for all $N\in\N$, $h \in (0,T/3]$ that 
	\begin{equation}\label{eq:gammasym_1}
		\begin{split}
			& \int_{[0,T]} (D^{h,N}(s-))^\top dX^{h,N}(s) + \frac12 \int_{[0,T]} (\Delta X^{h,N}(s))^\top \wt\gamma dX^{h,N}(s) \\
			& = \sum_{k=0}^3 \Big( (D^{h,N}(kh-))^\top \Delta X^{h,N}(kh) + \frac12 (\Delta X^{h,N}(kh))^\top \wt\gamma \Delta X^{h,N}(kh)  \Big) \\
			& =  \sum_{j=1}^3 \sum_{k=0}^{j-1} (\Delta (X^{h,N}(kh))^{\top} \wt\gamma^\top (\nu(kh))^{\top} (\nu^{-1}(jh))^{\top} \Delta X^{h,N}(jh) \\
			& \quad + \frac12 \sum_{j=0}^3 (\Delta X^{h,N}(jh))^\top \wt\gamma \Delta X^{h,N}(jh) .
		\end{split}
	\end{equation}
	Observe for all $N\in\N$, $h \in (0,T/3]$ that 
	\begin{equation}\label{eq:gammasym_2}
		\begin{split}
			\sum_{j=0}^3 (\Delta X^{h,N}(jh))^\top \wt\gamma \Delta X^{h,N}(jh) 
			& = 2 N^2 e_1^\top \wt\gamma e_1 
			+ 2 a^2 e_2^\top \wt\gamma e_2 .
		\end{split}
	\end{equation}
	We moreover have for all $N\in\N$, $h \in (0,T/3]$ that
	\begin{equation}\label{eq:gammasym_4}
		\begin{split}
			& \sum_{j=1}^3 \sum_{k=0}^{j-1} (\Delta X^{h,N}(kh))^{\top} \wt\gamma^\top (\nu(kh))^{\top} (\nu^{-1}(jh))^{\top} \Delta X^{h,N}(jh) \\
			& = N a 
			\bigg( e_1^\top \wt\gamma^\top (\nu(0))^\top (\nu^{-1}(h))^\top e_2 
			- e_2^\top \wt\gamma^\top (\nu(h))^\top (\nu^{-1}(2h))^\top e_1 \\
			& \qquad \qquad - e_1^\top \wt\gamma^\top (\nu(0))^\top (\nu^{-1}(3h))^\top e_2 
			+ e_1^\top \wt\gamma^\top (\nu(2h))^\top (\nu^{-1}(3h))^\top e_2 
			\bigg) \\
			& \quad - N^2 e_1^\top \wt\gamma^\top (\nu(0))^\top (\nu^{-1}(2h))^\top e_1
			- a^2 e_2^\top \wt\gamma^\top (\nu(h))^\top (\nu^{-1}(3h))^\top e_2 .
		\end{split}
	\end{equation}
	Combining \eqref{eq:gammasym_1}, \eqref{eq:gammasym_2},  and \eqref{eq:gammasym_4} yields 
	for all $N\in\N$, $h \in (0,T/3]$ that
	\begin{equation}\label{eq:1601}
		\begin{split}
			& \int_{[0,T]} (D^{h,N}(s-))^\top dX^{h,N}(s) 
			+ \frac12 \int_{[0,T]} (\Delta X^{h,N}(s))^\top \wt\gamma dX^{h,N}(s) \\
			& = 
			N^2 \left( e_1^\top \wt\gamma e_1 
			- e_1^\top \wt\gamma^\top (\nu(0))^\top (\nu^{-1}(2h))^\top e_1 \right) 
			+  
			a^2 \left( e_2^\top \wt\gamma e_2
			- e_2^\top \wt\gamma^\top (\nu(h))^\top (\nu^{-1}(3h))^\top e_2 \right) \\
			& \quad + 
			N a
			\bigg( e_1^\top \wt\gamma^\top (\nu(0))^\top (\nu^{-1}(h))^\top e_2 
			- e_2^\top \wt\gamma^\top (\nu(h))^\top (\nu^{-1}(2h))^\top e_1 \\
			& \quad \qquad \qquad - e_1^\top \wt\gamma^\top (\nu(0))^\top (\nu^{-1}(3h))^\top e_2 
			+ e_1^\top \wt\gamma^\top (\nu(2h))^\top (\nu^{-1}(3h))^\top e_2 
			\bigg) .
		\end{split}
	\end{equation}
	Since $\nu$ and $\nu^{-1}$ are continuous, it follows for all $N\in\N$ that $P$-a.s. 
	\begin{equation}\label{eq:2228}
		\begin{split}
			& \lim_{h\downarrow 0} \int_{[0,T]} (D^{h,N}(s-))^\top dX^{h,N}(s) 
			+ \frac12 \int_{[0,T]} (\Delta X^{h,N}(s))^\top \wt\gamma dX^{h,N}(s) \\
			& = 
			N a 
			\left(- e_2^\top \wt\gamma^\top e_1 
			+ e_1^\top \wt\gamma^\top e_2 
			\right) 
			= - Na
			\left( e_1^\top (\wt\gamma-\wt\gamma^\top) e_2 
			\right) 
			.
		\end{split}
	\end{equation}
	Note that for all $p\in[1,\infty)$ it holds $E[\sup_{s\in[0,T]} \lVert \nu_s \rVert_F^p]  <\infty$ and $E[\sup_{s\in[0,T]} \lVert \nu_s^{-1} \rVert_F^p]<\infty$. 
	This, \eqref{eq:1601}, \eqref{eq:2228}, Hölder's inequality, and the dominated convergence theorem
	show for all $N\in\N$ that 
	\begin{equation*}
		\begin{split}
			& \lim_{h\downarrow 0} \!
			E_0\!\bigg[\! \int_{[0,T]} \! \! (D^{h,N}(s-))^\top dX^{h,N}(s) 
			+ \frac12 \! \int_{[0,T]} \!\! (\Delta X^{h,N}(s))^\top \wt\gamma dX^{h,N}(s) \bigg] 
			\! 
			= \! - N  a e_1^\top (\wt\gamma-\wt\gamma^\top) e_2 .
		\end{split}
	\end{equation*}
	The fact that $a e_1^\top (\wt\gamma - \wt\gamma^\top)  e_2 >0$ demonstrates that 
	the right-hand side of this equation tends to $-\infty$ as $N\to\infty$.
	Thus, by trading faster ($h\downarrow0$) and larger volumes ($N\to\infty$), we in this example can generate arbitrarily large negative costs. 
\end{ex}

\begin{remark}
	Note that there are observations similar to \cref{ex:gamma_sym}  in the literature on price impact models.
	For example, Schneider \& Lillo show in  \cite[Lemma~3.9]{SchneiderLillo2019} 
	that in their multivariate transient impact model with linear price impact function, symmetry of the price impact is a necessary condition for the absence of dynamic arbitrage. 
	Moreover, if in the non-linear multivariate transient impact model of Hey et al.\ \cite{hey2024concave} we choose $L$ to be the identity matrix and $h$ to be the identity function (linear impact), then  \cite[Lemma~5.2]{hey2024concave} 
	yields that to avoid price manipulation their deterministic invertible matrix~$\Lambda$, which corresponds to our price impact~$\gamma$, needs to be symmetric. 
	Huberman \& Stanzl, for example,  
	illustrate in \cite[Section~5]{huberman2004price} 
	that asymmetric price impact in their multivariate (non-transient) price impact setting can lead to price manipulation and they show 
	that it is necessary and sufficient for the absence of price manipulation in their model that their price impact function is linear and represented by a symmetric positive semidefinite matrix. 
	Alfonsi et al., for instance, derive in \cite{AlfonsiKloeckSchied2016} 
	that their decay kernel, which incorporates price impact and resilience in their multi-asset model, should be, among others, symmetric positive semidefinite to guarantee desirable properties of the model (see \cite[Section~4]{AlfonsiKloeckSchied2016}). 
	To summarize, the assumption that the price impact (in the sense of the respective model) is symmetric is a usual one in the literature, and it is also rather common to make a definiteness assumption.
	In our model we consider a stochastically evolving price impact $\gamma$ and we further specify its dynamics. Our definition of $\gamma$ implies symmetry and positive definiteness of $\gamma$; for details we refer to \cref{sec:setting}, and to \cref{rem:remark_on_setting_assumptions1} in particular.  
\end{remark}

In the next example we illustrate how in our model the resilience affects the price deviation. 

\begin{ex}\label{ex:resilience_effect}
	Let $n=2$, 
	$t_1\in(0,T)$, $x\in\R^n$, $d=0\in\R^n$, and 
	\begin{equation*}
		\rho = 
		\begin{pmatrix}
			\rho_1 & \rho_3\\
			\rho_3 & \rho_1
		\end{pmatrix}
	\end{equation*}
	where $\rho_1 \in (0,\infty)$ and $\rho_3 \in \R$. 
	We assume that $\rho_1^2>\rho_3^2$, which ensures that $\rho$ is positive definite. 
	In order to explore the resilience effect in our model, we suppose that there is a block trade at the time\footnote{The choice $t_0=0$ is just to lighten notation. For general $t_0\in [0,t_1)$, replace $r$ by $r-t_0$ where appropriate.}~$t_0=0$ and we study the development of the deviation after this block trade in a period $(t_0,t_1]$ without new trades.   
	More specifically, we consider an $\R^n$-valued c\`adl\`ag finite-variation process $X=(X(s))_{s\in[0,T]}$ with $X(0-)=x$, $\Delta X(0)\ne 0$, and $X(s)=X(0)$, $s\in (0,t_1]$. 
	The associated process $D^X=(D^X(s))_{s\in[0,T]}$ defined by~\eqref{eq:deviationdynmultivariate} 
	satisfies for all $r \in [0,t_1]$ that 
	$D^X(r) = \nu^{-1}(r) \gamma(0) \Delta X(0)$  
	(cf.~\cref{rem:soln_for_deviation_SDE}). 
	Note that it holds for all $r \in [0,T]$ that $\nu^{-1}(r) =e^{-\rho r}$. 
	By diagonalizing $\rho$ we obtain for all $r \in [0,t_1]$ that 
	\begin{equation}\label{eq:15351}
		\begin{split}
			D^X(r) & = 
			\frac12 e^{-\rho_1 r}
			\begin{pmatrix}
				e^{-\rho_3 r} + e^{\rho_3 r} \\
				e^{-\rho_3 r} - e^{ \rho_3r}
			\end{pmatrix}
			D^X_1(0) 
			+ \frac12 e^{-\rho_1 r}
			\begin{pmatrix}
				e^{-\rho_3 r} - e^{\rho_3 r} \\
				e^{-\rho_3 r} + e^{\rho_3 r}
			\end{pmatrix}
			D^X_2(0) 
			. 
		\end{split}
	\end{equation}
	Hence, if $\rho_3\ne 0$ and $D^X_2(0)\ne 0$, then the deviation $D^X_1$ in the first asset is computed not only from the deviation $D_1^X(0)$ in the first asset immediately after the block trade, but also depends on the deviation $D_2^X(0)$ in the second asset immediately after the block trade. 
	
	Let us focus on the first component\footnote{Analogous comments hold for the second component by symmetry.} of~\eqref{eq:15351}, that is, $(D_1^X(r))_{r\in[0,t_1]}$, and assume that\footnote{Note that changing the sign of $\rho_3$ has the same effect on $D^X_1$ as changing the sign of $D_2^X(0)$.} $\rho_3>0$. 
	Observe that the function $[0,\infty) \ni r \mapsto e^{-\rho_3 r} + e^{\rho_3 r} \in [2,\infty)$ is increasing. 
	This means that the usual resilience effect (described by $[0,t_1] \ni r\mapsto e^{-\rho_1 r}$) that acts on the deviation $D_1^X(0)$ in the first asset immediately after the block trade is decelerated 
	with increasing time $r\in[0,t_1]$. 
	At the same time, we have the additional term 
	\begin{equation}\label{eq:1632}
		\frac12 e^{-\rho_1 r} (e^{-\rho_3 r} - e^{\rho_3 r})D_2^X(0), \quad r\in[0,t_1]. 
	\end{equation}
	It holds for all $r\in[0,t_1]$ that $\lvert e^{-\rho_3 r} - e^{ \rho_3 r} \rvert \le e^{-\rho_3 r} + e^{ \rho_3 r}$.  
	Therefore, the factor in front of $D_2^X(0)$ in~\eqref{eq:1632} has at most the magnitude of the factor in front of $D_1^X(0)$.
	We can show that $[0,\infty) \ni r \mapsto e^{-\rho_1 r} (e^{-\rho_3 r} - e^{\rho_3 r})$ is $0$ at $r=0$, is non-positive, has a minimum at $r_0=\frac{1}{2\rho_3}\ln(\frac{\rho_1+\rho_3}{\rho_1-\rho_3})$, is decreasing on $(0,r_0)$, is increasing on $(r_0,\infty)$, and tends to $0$ as $r\to\infty$.  
	In particular, the additional term~\eqref{eq:1632} has the opposite sign of~$D_2^X(0)$. 
	The shape of the function $[0,\infty) \ni r \mapsto e^{-\rho_1 r} (e^{-\rho_3 r} - e^{\rho_3 r})$ (see also the orange 
	curve in \cref{fig:resilience_effect_trade_in_first})
	\begin{SCfigure}
		\caption{The deviation in the first and the second asset after a block buy trade of size $10$ at the time $0$ in the first asset. The setting is $t_1=5$, $\rho_1=2$, $\rho_3=1$, $\gamma(0)=I_2$, and $\Delta X(0)=(10,0)^\top$ within \cref{ex:resilience_effect}.}
		\label{fig:resilience_effect_trade_in_first}
		\includegraphics[scale=0.75]{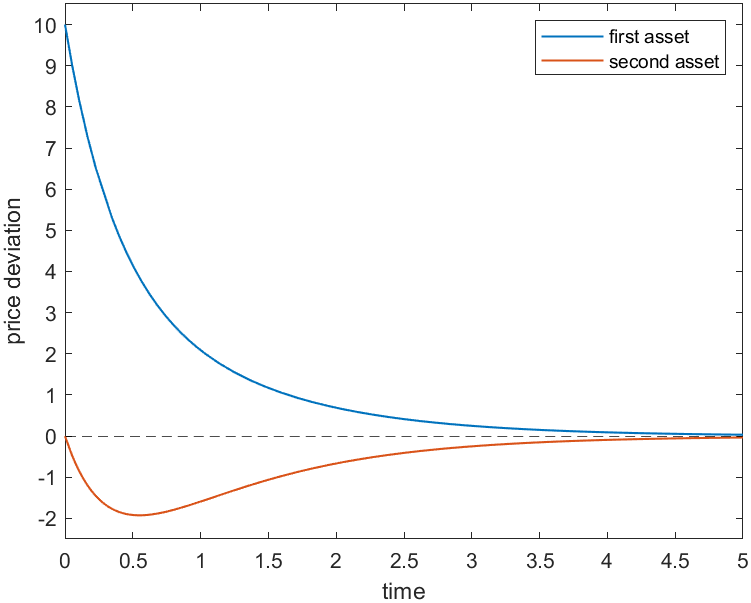}
	\end{SCfigure}
	indicates that 
	the influence of $D_2^X(0)$ on the deviation~$D_1^X$ in the first asset first has to build up and 
	from the time~$r_0$ on then reverts to~$0$.
	
	In \cref{fig:resilience_effect_trade_in_both} 
	\begin{figure}
		\caption{The deviation in the first asset after a block trade $\Delta X(0)=(3,1)^\top$ (topleft), $\Delta X(0)=(1,3)^\top$ (topright), $\Delta X(0)=(3,-1)^\top$ (bottomleft), and $\Delta X(0)=(1,-3)^\top$ (bottomright) at the time $0$. 
			The setting is $t_1=5$, $\rho_1=2$, $\rho_3=1$, and $\gamma(0)=I_2$ within \cref{ex:resilience_effect}. The brown lines indicate the deviation in the first asset without cross-resilience.}
		\label{fig:resilience_effect_trade_in_both}
		\includegraphics[scale=0.7]{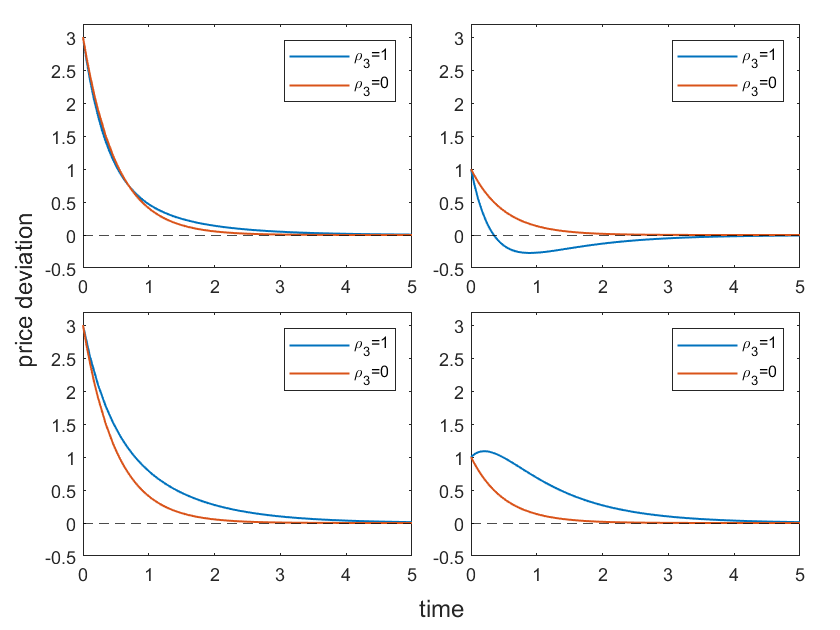}
	\end{figure}
	we visualize the deviation $(D_1^X(r))_{r\in[0,t_1]}$ in the first asset when $\rho_1=2$, $\rho_3\in\{0,1\}$, $\gamma(0)=I_2$, and $\Delta X(0)=(3,1)^\top$, $\Delta X(0)=(1,3)^\top$, $\Delta X(0)=(3,-1)^\top$, or $\Delta X(0)=(1,-3)^\top$, respectively.
	We see from this figure that in the case $\Delta X(0)=(3,1)^\top$ the decay to zero of the deviation in the first asset is slightly accelerated in the beginning and slowed down afterwards in comparison to the situation without cross-resilience. 
	If $\Delta X(0)=(3,-1)^\top$, then the speed of the decay to zero is decreased.
	The shape of the deviation in the first asset is more strongly affected in the case when the size of the trade in the second asset dominates the one in the first asset. 
	In the example with $\Delta X(0)=(1,3)^\top$, we see that the deviation in the first asset becomes negative after a short period of time and then approaches zero from below, that means, the price in the first asset essentially drops below the price that was valid prior to the buy trade $\Delta X_1(0)=1$. 
	In contrast, if $\Delta X(0)=(1,-3)^\top$, then the deviation in the first asset increases for a short period of time before the decay to zero (from above) begins.
	
	The influence of a trade in only the first asset on the deviation $D^X$ in both assets, assuming that $\rho_1=2$, $\rho_3=1$, $\gamma(0)=I_2$, and $\Delta X(0)=(10,0)^\top$, is shown in \cref{fig:resilience_effect_trade_in_first}. 
\end{ex}

\section{Solving the trade execution problem in the case of non-vanishing targets}
\label{sec:appendixsolving}

In \cref{sec:zerotargets} we provide a solution to the LQ stochastic control problem of \cref{sec:equivLQ} in the case of zero targets $\xi=0$ and $\zeta=0$. 
In this appendix, we consider general terminal targets~$\xi$ and general running targets~$\zeta$ satisfying \eqref{eq:int_cond_terminal_pos} and \eqref{eq:int_cond_zeta}. 
We develop a rigorous solution theory which requires a transformation of the control problem removing the cross-terms (products between the control $u$ and the state $\scH^{u}$) and relies on results from Kohlmann \& Tang \cite{kohlmann2003minimization}
combined with results from Sun et al.~\cite{sun2021indefiniteLQ}. 
As an alternative, we outline in the next remark a solution approach that 
uses results from \cite{sun2021indefiniteLQ} and further arguments to be developed. 

\begin{remark}\label{rem:gen_targets_outline}
	We start with the ansatz that for all $h\in \R^n$ and $t\in [0,T]$ it holds almost surely that  
	\begin{equation}\label{eq:gen_quad_form_gen_targets}
		\essinf_{u\in \cL_t^2} \LQJ_t(h,u) = \langle h, \mathscr{Y}(t) h\rangle + \langle \psi(t), h\rangle + V^0(t)
	\end{equation}
	with progressively measurable processes $\mathscr{Y}$, $\psi$, and $V^0$ of appropriate dimensions (looking at a dynamic version of the cost functional).
	It can then be shown that $\mathscr{Y}$ is characterized as the first component of the solution  $(\mathscr{Y}, \mathscr{Z})$ of the Riccati-BSDE~\eqref{eq:BSDE} with the driver given by~\eqref{eq:driver_of_BSDE}. Similarly to \cref{sec:zerotargets}, the results of \cite{sun2021indefiniteLQ} can be used to ensure the existence of $(\mathscr{Y}, \mathscr{Z})$. Next, one can establish that $\psi$ is characterized by a linear BSDE (whose driver depends nonlinearly on $(\mathscr{Y}, \mathscr{Z})$). The next goal is to identify appropriate conditions that ensure the existence of $\psi$. Finally, $V^0$ is determined as the conditional expectation of a time integral (involving $\mathscr{Y}$ and $\psi$ as integrands). Once the existence of $\mathscr{Y}$, $\psi$, and $V^0$ is verified, one can proceed with the classical verification argument in LQ stochastic optimal control, employing Itô’s formula and the completion of squares, to establish the representation \eqref{eq:gen_quad_form_gen_targets}.
\end{remark}

Next, we develop the approach primarily based on \cite{kohlmann2003minimization}. It basically follows the procedure outlined in \cref{rem:gen_targets_outline}, but first transforms the LQ stochastic control problem defined by \eqref{eq:def_cost_fct_LQ} and \eqref{eq:def_scH_u} into an equivalent formulation without cross-terms. This transformation is performed under the following assumption.

\begin{assumption}\label{ass:procB}
	Assume 
	that there exists an $\R^{n\times n}$-valued progressively measurable, $dP\times ds$-a.e.\ bounded process $\procB=(\procB(s))_{s\in[0,T]}$ such that 
	$\KK \procB = \mathscr{Q}$ $dP\times ds$-a.e. 
\end{assumption}

\begin{remark}
	Note that if there exists $\varepsilon\in (0,\infty)$ such that $\KK-\varepsilon I_n$ is $\cS_{\ge 0}^n$-valued, then \cref{ass:procB} holds with $\procB=\KK^{-1} \mathscr{Q}$. 
	In general, the process in \cref{ass:procB}, if it exists, is not necessarily unique (up to $dP\times ds$-null sets). 
	When we assume that \cref{ass:procB} holds, we choose and fix a process $\procB$ with the properties detailed in \cref{ass:procB} and use it for all objects involved.
	Observe that in the case without a risk term, that is, $\mathscr{Q}=0$, we have that \cref{ass:procB} is satisfied with $\procB=0$.
\end{remark}

\subsection{An equivalent LQ stochastic control problem without cross-terms}
\label{sec:equivLQnocross}

Under \cref{ass:procB}, and given $u\in\cL^2$, we make the linear transformation $u \mapsto \hat{u}:= u - \procB (\scH^{u} + \gamma^{\frac12} \zeta)$ and take $\scH^{u}$ to be the state process also for $\hat{u}$. 
Therefore, given $\hat{u}\in\cL^2$, we introduce (under \cref{ass:procB}) the controlled SDE 
\begin{equation}\label{eq:def_scHhat}
	\begin{split}
		d\scHhat^{\hat{u}}(s) & = 
		[\mathscr{A}(s) + \mathscr{B}(s) \procB(s) ] \scHhat^{\hat{u}}(s) ds 
		+ \mathscr{B}(s) \hat{u}(s) ds 
		+ \mathscr{B}(s) \procB(s) \gamma^{\frac12}(s) \zeta(s) ds 
		\\
		& \quad 
		+ \sum_{k=1}^m \mathscr{C}^k(s) (I_n - 2 \procB(s)) \scHhat^{\hat{u}}(s) dW_k(s) 
		- 2 \sum_{k=1}^m \mathscr{C}^k(s) \hat{u}(s) dW_k(s)  \\
		& \quad 
		- 2 \sum_{k=1}^m \mathscr{C}^k(s) \procB(s) \gamma^{\frac12}(s) \zeta(s) dW_k(s), \quad s\in[0,T],\\
		\scHhat^{\hat{u}}(0) & = \gamma^{-\frac12}(0) d - \gamma^{\frac12}(0) x. 
	\end{split}
\end{equation}

\begin{remark}\label{rem:SDE_state_transformed}
	Assume that \cref{ass:procB} holds and 
	let $\hat{u}\in\cL^2$.
	Observe that the coefficients in \eqref{eq:def_scHhat} are $dP\times ds$-a.e.\ bounded and that $\mathscr{B} \procB \gamma^{\frac12} \zeta$ and $\mathscr{C}^k \procB \gamma^{\frac12} \zeta$, $k\in\{1,\ldots,m\}$, are in $\cL^2$ due to \eqref{eq:int_cond_zeta}. Hence, 
	there exists a unique solution $\scHhat^{\hat{u}} = (\scHhat^{\hat{u}}(s))_{s\in[0,T]}$ of the SDE~\eqref{eq:def_scHhat} 
	and it holds that $E[\sup_{s\in[0,T]} \lVert\scHhat^{\hat{u}}(s) \rVert_F^2]<\infty$ (cf., for example, \cite[Theorem 3.2.2 \& Theorem 3.3.1]{zhang}).
\end{remark}

Furthermore, under \cref{ass:procB} and for $\hat{u}\in\cL^2$ and $\scHhat^{\hat{u}}$ given by \eqref{eq:def_scHhat}, we define 
\begin{equation}\label{eq:def_cost_fct_LQ_hat}
	\begin{split}
		\LQJhat( \hat{u} ) 
		& = \frac12 E\big[ \big(\scHhat^{\hat{u}}(T)+\gamma^{\frac12}(T)\xi \big)^\top \big(\scHhat^{\hat{u}}(T)+\gamma^{\frac12}(T) \xi \big) \big] 
		+ E\bigg[ \int_0^T (\hat{u}(s))^\top 
		\KK(s) 
		\hat{u}(s)  ds \bigg] \\
		& \quad 
		+ E\bigg[ \int_0^T \big(\scHhat^{\hat{u}}(s) + \gamma^{\frac12}(s) \zeta(s) \big)^\top \mathscr{Q}(s) ( I_n - \procB(s)) \big(\scHhat^{\hat{u}}(s) + \gamma^{\frac12}(s) \zeta(s) \big) ds \bigg] 
		. 
	\end{split}
\end{equation} 
Note that the cost functional~\eqref{eq:def_cost_fct_LQ_hat} does not contain cross-terms. 
It is a reformulation of the cost functional~\eqref{eq:def_cost_fct_LQ}:

\begin{lemma}\label{lem:trafouvsuhat}
	Let \cref{ass:procB} be in force. 
	
	(i) 
	Let $u \in\cL^2$ and let $\scH^{u}$ be the associated solution of \eqref{eq:def_scH_u} for $u$. 
	Then, $\hat u =(\hat{u}(s))_{s \in [0,T]}$ defined by 
	$\hat{u}(s) = u(s) - \procB(s) (\scH^u(s) + \gamma^{\frac12}(s) \zeta(s) )$, $s\in[0,T]$, 
	is in $\cL^2$, and it holds that $\scHhat^{\hat u}=\scH^u$ and 
	$\LQJ(u ) = \LQJhat(\hat u )$.
	
	(ii) 
	Let $\hat{u} \in\cL^2$ and let $\scHhat^{\hat{u}}$ be the associated solution of \eqref{eq:def_scHhat} for $\hat{u}$.
	Then, $u =(u(s))_{s \in [0,T]}$ defined by 
	$u(s) = \hat{u}(s) + \procB(s) (\scHhat^{\hat{u}}(s) + \gamma^{\frac12}(s) \zeta(s) )$, $s\in[0,T]$, 
	is in $\cL^2$, and it holds that $\scHhat^{\hat u}=\scH^u$ and 
	$\LQJ( u ) = \LQJhat( \hat u )$.
\end{lemma}

\begin{proof}
	(i) 
	The assumption that $F$ is bounded, the fact that $u \in \cL^2$, \eqref{eq:int_cond_zeta}, and \cref{rem:SDE_hidden_dev_and_new_state} establish that $\hat{u} \in \cL^2$. 
	Moreover, we obtain from \eqref{eq:def_scH_u} and $u = \hat{u} + \procB (\scH^u + \gamma^{\frac12} \zeta )$ that 
	\begin{equation*}
		\begin{split}
			d\scH^u(s) & = 
			\mathscr{A}(s) \scH^u(s) ds 
			+ \mathscr{B}(s) \hat{u}(s) ds 
			+ \mathscr{B}(s) \procB(s) \scH^u(s) ds 
			+ \mathscr{B}(s) \procB(s) \gamma^{\frac12}(s) \zeta(s) ds \\
			& \quad 
			+ \sum_{k=1}^m \mathscr{C}^k(s) \scH^u(s) dW_k(s) 
			- 2 \sum_{k=1}^m \mathscr{C}^k(s) \hat{u}(s) dW_k(s) \\
			& \quad
			- 2 \sum_{k=1}^m \mathscr{C}^k(s) \procB(s) \scH^u(s) dW_k(s)  
			- 2 \sum_{k=1}^m \mathscr{C}^k(s) \procB(s) \gamma^{\frac12}(s) \zeta(s) dW_k(s) , 
			\, s\in[0,T].
		\end{split}
	\end{equation*}
	Hence, $\scH^u=\scHhat^{\hat{u}}$ (cf.~\cref{rem:SDE_state_transformed}).
	Furthermore, combining this, $\hat{u} = u - \procB (\scH^u + \gamma^{\frac12} \zeta)$, and \cref{ass:procB} we have that 
	\begin{equation*}
		\begin{split}
			& \hat{u}^\top \KK \hat{u}
			+ \big(\scHhat^{\hat{u}} + \gamma^{\frac12} \zeta \big)^\top \mathscr{Q} ( I_n - \procB) \big(\scHhat^{\hat{u}} + \gamma^{\frac12} \zeta \big) \\
			& = u^\top \KK u
			- 2 u^\top \KK \procB (\scH^u + \gamma^{\frac12} \zeta) 
			+ (\scH^u + \gamma^{\frac12} \zeta)^\top \procB^\top \KK \procB (\scH^u + \gamma^{\frac12} \zeta) \\
			& \quad + \big(\scHhat^{\hat{u}} + \gamma^{\frac12} \zeta \big)^\top \mathscr{Q} ( I_n - \procB) \big(\scHhat^{\hat{u}} + \gamma^{\frac12} \zeta \big) \\
			& = u^\top \KK u
			- 2 u^\top \mathscr{Q} (\scH^u + \gamma^{\frac12} \zeta) 
			+ (\scH^u + \gamma^{\frac12} \zeta)^\top \procB^\top \mathscr{Q} (\scH^u + \gamma^{\frac12} \zeta) \\
			& \quad + \big(\scH^{u} + \gamma^{\frac12} \zeta \big)^\top \mathscr{Q} \big(\scH^{u} + \gamma^{\frac12} \zeta \big)
			- \big(\scH^{u} + \gamma^{\frac12} \zeta \big)^\top \mathscr{Q} \procB \big(\scH^{u} + \gamma^{\frac12} \zeta \big) \\
			& = u^\top \KK u
			+ \big(\scH^u + \gamma^{\frac12} \zeta \big)^\top \mathscr{Q} \big(\scH^u + \gamma^{\frac12} \zeta \big) 
			- 2 u^\top \mathscr{Q} \big(\scH^u + \gamma^{\frac12} \zeta \big) .
		\end{split}
	\end{equation*}
	Therefore, \eqref{eq:def_cost_fct_LQ},  \eqref{eq:def_cost_fct_LQ_hat}, and the fact that $\scH^u=\scHhat^{\hat{u}}$ yield that 
	$\LQJ(u ) = \LQJhat(\hat u ).$
	
	(ii) This claim follows from similar arguments as in (i). 
\end{proof}

From \cref{lem:trafouvsuhat} we obtain that the control problems pertaining to $\LQJhat$ and $\LQJ$ are equivalent in the following sense. 

\begin{corollary}\label{cor:uoptimaliffuhatoptimal}
	Let \cref{ass:procB} be in force. 
	
	(i) It holds a.s.\ that 
	\begin{equation*}
		\begin{split}
			\essinf_{u\in\cL^2} \LQJ(u)
			& = \essinf_{\hat{u}\in\cL^2} \LQJhat( \hat{u}) .
		\end{split}
	\end{equation*}
	
	(ii)
	Suppose that $u^*=(u^*(s))_{s \in [0,T]} \in \cL^2$ minimizes $\LQJ$ over $\cL^2$ and let $\scH^{u^*}$ be the associated solution of \eqref{eq:def_scH_u} for $u^*$.
	Then, $\hat{u}^* =(\hat{u}^*(s))_{s \in [0,T]}$ defined by 
	$$\hat{u}^*(s) = u^*(s) - \procB(s) (\scH^{u^*}(s) + \gamma^{\frac12}(s) \zeta(s) ), \quad s\in[0,T],$$ 
	minimizes $\LQJhat$ over $\cL^2$. 
	
	(iii)
	Suppose that $\hat{u}^*=(\hat{u}^*(s))_{s \in [0,T]} \in \cL^2$ minimizes $\LQJhat$ over $\cL^2$ and let $\scHhat^{\hat{u}^*}$ be the associated solution of \eqref{eq:def_scHhat} for $\hat{u}^*$.
	Then, $u^* =(u^*(s))_{s \in [0,T]}$ defined by 
	$$u^*(s) = \hat{u}^*(s) + \procB(s) (\scHhat^{\hat{u}^*}(s) + \gamma^{\frac12}(s) \zeta(s) ), \quad s\in[0,T],$$ 
	minimizes $\LQJ$ over $\cL^2$. 
	
	(iv)
	There exists a $dP\times ds$-a.e.\ unique minimizer of $\LQJ$ in $\cL^2$ if and only if there exists a $dP\times ds$-a.e.\ unique minimizer of $\LQJhat$ in $\cL^2$. 
\end{corollary}

\subsection{Solution of the LQ stochastic control problem without cross-terms}\label{sec:general_targets_soln}

We introduce a matrix-valued BSDE of Riccati type for the LQ problem of \cref{sec:equivLQnocross} (assuming \cref{assump_filtration} and \cref{ass:procB}):
\begin{equation}\label{eq:BSDE_Riccati_KT}
	\begin{split}
		d \widehat{\mathscr{Y}}(s) 
		& = - \hat{g}\big(s,\cdot,\widehat{\mathscr{Y}}(s),\widehat{\mathscr{Z}}^1(s),\ldots,\widehat{\mathscr{Z}}^m(s)\big) ds
		+ \sum_{k=1}^m \widehat{\mathscr{Z}}^k(s) dW_k(s),   
		\quad s \in [0,T],  \\
		\widehat{\mathscr{Y}}(T) & = \tfrac12 I_n
	\end{split}
\end{equation}
with the driver 
\begin{equation*}
	\begin{split}
		& \hat{g}\big(s,\omega,\widehat{\mathscr{Y}}(s,\omega),\widehat{\mathscr{Z}}^1(s,\omega),\ldots,\widehat{\mathscr{Z}}^m(s,\omega)\big) \\
		& = \widehat{\mathscr{Y}} (\mathscr{A}+\mathscr{B}\procB) + (\mathscr{A}+\mathscr{B}\procB)^\top \widehat{\mathscr{Y}} 
		+ \mathscr{Q} (I_n-\procB) \\
		& \quad + \sum_{k=1}^m \big( (I_n-2\procB^\top) \mathscr{C}^k \widehat{\mathscr{Y}} \mathscr{C}^k (I_n-2\procB) 
		+ \widehat{\mathscr{Z}}^k \mathscr{C}^k (I_n-2\procB) 
		+ (I_n-2\procB^\top) \mathscr{C}^k \widehat{\mathscr{Z}}^k \big) \\
		& \quad - \bigg( \widehat{\mathscr{Y}} \mathscr{B} - 2 \sum_{k=1}^m \big( (I_n-2\procB^\top) \mathscr{C}^k \widehat{\mathscr{Y}} \mathscr{C}^k 
		+ \widehat{\mathscr{Z}}^k \mathscr{C}^k \big) \bigg)
		\cdot \bigg( \KK + 4 \sum_{k=1}^m \mathscr{C}^k \widehat{\mathscr{Y}} \mathscr{C}^k  \bigg)^{-1} \\
		& \qquad \cdot 
		\bigg( \mathscr{B}^\top \widehat{\mathscr{Y}} 
		- 2 \sum_{k=1}^m \big( \mathscr{C}^k \widehat{\mathscr{Y}} \mathscr{C}^k (I_n-2\procB)
		+ \mathscr{C}^k \widehat{\mathscr{Z}}^k \big) \bigg), 
	\end{split}
\end{equation*}
where on the right-hand side of the equation for the driver we suppressed the dependence on $\omega\in\Omega$ and $s \in [0,T]$. 
A pair $(\widehat{\mathscr{Y}},\widehat{\mathscr{Z}})$ with $\widehat{\mathscr{Z}}=(\widehat{\mathscr{Z}}^1,\ldots,\widehat{\mathscr{Z}}^m)$ is called  
a solution of the BSDE~\eqref{eq:BSDE_Riccati_KT} if 
\begin{itemize}
	\item the process $\widehat{\mathscr{Y}} \colon [0,T]\times \Omega \to \cS^n$ is bounded, adapted, and continuous,
	
	\item for every $k\in\{1,\ldots,m\}$ it holds that the process $\widehat{\mathscr{Z}}^k\colon [0,T]\times \Omega \to \cS^n$ is progressively measurable and satisfies $E[\int_0^T \lVert \widehat{\mathscr{Z}}^k(s) \rVert_F^2 ds] < \infty$, 
	
	\item 
	the process 
	$\KK + 4 \sum_{k=1}^m \mathscr{C}^k \widehat{\mathscr{Y}} \mathscr{C}^k$ is    
	$dP\times ds$-a.e.\ $\cS^n_{> 0}$-valued, 
	and
	
	\item the BSDE~\eqref{eq:BSDE_Riccati_KT} is satisfied $P$-a.s.
\end{itemize}

\begin{propo}\label{propo:existenceRiccatiBSDE}
	Let \cref{assump_filtration} and \cref{ass:procB} be in force. 
	Assume that	$\mathscr{Q}(I_n-\procB)$ and $\KK$ are $dP\times ds$-a.e.\ $\cS_{\ge 0}^n$-valued. 
	
	(i) 
	Suppose that there exists $\delta \in (0,\infty)$ such that $\KK - \delta I_n$ is $dP\times ds$-a.e.\  $\cS_{\ge 0}^n$-valued. 
	Then there exists a unique solution $(\widehat{\mathscr{Y}},\widehat{\mathscr{Z}})$ of the BSDE~\eqref{eq:BSDE_Riccati_KT}.
	Moreover, $\widehat{\mathscr{Y}}$ is $P$-a.s.\ $\cS_{\ge 0}^n$-valued and there exists $\varepsilon\in(0,\infty)$ such that 
	$\KK + 4 \sum_{k=1}^m \mathscr{C}^k \widehat{\mathscr{Y}} \mathscr{C}^k -\varepsilon I_n$ 
	is    
	$dP\times ds$-a.e.\ $\cS^n_{\ge 0}$-valued.
	
	(ii) 
	Suppose that there exists $\delta \in (0,\infty)$ such that $\sum_{k=1}^m \diagsigma_{k} \diagsigma_k - \delta I_n$ is $dP\times ds$-a.e.\ $\cS_{\ge 0}^n$-valued.  
	Then there exists a unique solution $(\widehat{\mathscr{Y}},\widehat{\mathscr{Z}})$ of the BSDE~\eqref{eq:BSDE_Riccati_KT}.
	Moreover, there exists $\varepsilon\in(0,\infty)$ such that $\widehat{\mathscr{Y}} -\varepsilon I_n$ is $P$-a.s.\ $\cS^n_{\ge 0}$-valued.
\end{propo}

\begin{proof}
	For $y \in \R^n$, $t\in [0,T]$, and $u\in\cL_t^2$, where 
	\begin{equation*}
		\begin{split}
			\cL_t^2 = & \bigg\{v \,\bigg\lvert\, v\colon [t,T]\times \Omega \to\R^n \text{ progr.\ measurable and }  \lVert v \rVert_{\cL_t^2}^2 = E\bigg[\int_t^T \lVert v(s) \rVert_F^2 ds\bigg] < \infty \bigg\} , 
		\end{split}
	\end{equation*}
	let 
	\begin{equation*}
		\begin{split}
			dH^u(s) & = 
			(\mathscr{A}(s) + \mathscr{B}(s) \procB(s) ) H^u(s) ds 
			+ \mathscr{B}(s) u(s) ds
			\\
			& \quad 
			+ \sum_{k=1}^m \mathscr{C}^k(s) (I_n - 2 \procB(s)) H^u(s) dW_k(s) 
			- 2 \sum_{k=1}^m \mathscr{C}^k(s) u(s) dW_k(s)  , \quad s\in[t,T],\\
			H^u(t) & = y, 
		\end{split}
	\end{equation*}
	and 
	\begin{equation*}
		\begin{split}
			\widetilde{J}_t(y,u) 
			& = \frac12 E[ (H^u(T))^\top H^u \,\vert\, \cF_t ] 
			+ E\bigg[ \int_t^T (u(s))^\top 
			\KK(s) u(s) ds  \,\Big\vert\, \cF_t \bigg] \\
			& \quad 
			+ E\bigg[ \int_t^T (H^u(s))^\top \mathscr{Q}(s) ( I_n - \procB(s)) H^u(s) ds  \,\Big\vert\, \cF_t \bigg] 
			. 
		\end{split}
	\end{equation*}
	Note that \cref{assump_filtration} ensures that $(\cF_s)_{s\in[0,T]}$ is the augmented natural filtration of the Brownian motion $W$. 
	Moreover, observe that the progressively measurable processes $\mathscr{A} - \mathscr{B}\procB$, $\mathscr{B}$, $\mathscr{C}^k (I_n - 2 \procB)$,  $k\in\{1,\ldots,m\}$, $-2\mathscr{C}^k$,  $k\in\{1,\ldots,m\}$, $\KK$, and $\mathscr{Q}(I_n-\procB)$ 
	are $dP\times ds$-a.e.\ bounded (cf.\ \cref{rem:remark_on_setting_assumptions2} and \cref{ass:procB}).
	This and the fact that $\tfrac12 I_n$ is symmetric ensure that the assumptions (A1)' and (A2) in Sun et al.\ \cite{sun2021indefiniteLQ} are satisfied. 
	
	\smallskip
	
	(i) 
	Note that (5) in \cite{sun2021indefiniteLQ} is satisfied and thus \cite[Proposition~7.1]{sun2021indefiniteLQ} shows that the cost functional $\widetilde{J}$ meets the uniform convexity condition in \cite{sun2021indefiniteLQ}.  
	We can therefore apply \cite[Theorem~9.1]{sun2021indefiniteLQ} to obtain the existence of a unique solution $(\widehat{\mathscr{Y}},\widehat{\mathscr{Z}})$ of the BSDE~\eqref{eq:BSDE_Riccati_KT} 
	and the existence of an $\varepsilon\in(0,\infty)$ such that 
	$\KK + 4 \sum_{k=1}^m \mathscr{C}^k \widehat{\mathscr{Y}} \mathscr{C}^k -\varepsilon I_n$ 
	is    
	$dP\times ds$-a.e.\ $\cS^n_{\ge 0}$-valued.
	Moreover, it follows from Section~5 and the introduction of Section~6 in~\cite{sun2021indefiniteLQ} 
	for all $y \in \R^n$ and $t\in[0,T]$ that 
	$y^\top \widehat{\mathscr{Y}}(t) y = \essinf_{u\in\cL_t^2} \widetilde{J}_t(y,u)$. 
	Since $\mathscr{Q}(I_n-\procB)$ and $\KK$ are $dP\times ds$-a.e.\ $\cS_{\ge 0}^n$-valued, we hence conclude that $\widehat{\mathscr{Y}}$ is $P$-a.s.\ $\cS_{\ge 0}^n$-valued.
	
	\smallskip
	
	(ii) 
	Note that (8) in \cite{sun2021indefiniteLQ} is satisfied and thus \cite[Proposition~7.1]{sun2021indefiniteLQ} shows that the cost functional $\widetilde{J}$ meets the uniform convexity condition in \cite{sun2021indefiniteLQ}. 
	We can therefore apply \cite[Theorem~9.1]{sun2021indefiniteLQ} to obtain the existence of a unique solution $(\widehat{\mathscr{Y}},\widehat{\mathscr{Z}})$ of the BSDE~\eqref{eq:BSDE_Riccati_KT}.  
	Moreover, it follows from Section~5 and the introduction of Section~6 in~\cite{sun2021indefiniteLQ} 
	for all $y \in \R^n$ and $t\in[0,T]$ that 
	$y^\top \widehat{\mathscr{Y}}(t) y = \essinf_{u\in\cL_t^2} \widetilde{J}_t(y,u)$. 
	Note that also the assumptions (A1), (A2), (A5), (A6)', and (A7) in Kohlmann \& Tang  \cite{kohlmann2003minimization} are satisfied. Therefore, \cite[Lemma~3.2]{kohlmann2003minimization} ensures that there exists $\varepsilon \in (0,\infty)$ such that for all $y \in \R^n$, $t\in[0,T]$, and $u \in \cL_t^2$ it holds that 
	$\widetilde{J}_t(y,u) \ge \varepsilon y^\top y$. 
	Hence, it holds that  $\widehat{\mathscr{Y}} -\varepsilon I_n$ is $P$-a.s.\ $\cS_{\ge 0}^n$-valued.  
\end{proof}

Given a solution $(\widehat{\mathscr{Y}},\widehat{\mathscr{Z}})$ of the BSDE~\eqref{eq:BSDE_Riccati_KT}, we define the matrix-valued progressively measurable process $\widehat\theta=(\widehat\theta(s))_{s \in [0,T]}$ by, for all $s\in[0,T]$,  
\begin{equation}\label{eq:def_theta_hat}
	\begin{split}
		\widehat\theta(s) & = 
		- \bigg( \KK(s) + 4 \sum_{k=1}^m \mathscr{C}^k(s) \widehat{\mathscr{Y}}(s) \mathscr{C}^k(s)  \bigg)^{-1} \\
		& \quad \cdot 
		\bigg( (\mathscr{B}(s))^\top \widehat{\mathscr{Y}}(s) 
		- 2 \sum_{k=1}^m \big( \mathscr{C}^k(s) \widehat{\mathscr{Y}}(s) \mathscr{C}^k(s) (I_n - 2 \procB(s)) + \mathscr{C}^k(s) \widehat{\mathscr{Z}}^k(s) \big) \bigg). 
	\end{split}
\end{equation}

Next, given a solution $(\widehat{\mathscr{Y}},\widehat{\mathscr{Z}})$ of the BSDE~\eqref{eq:BSDE_Riccati_KT}, we introduce a matrix-valued linear BSDE: 
\begin{equation}\label{eq:BSDE_linear_KT}
	\begin{split}
		d \widehat{\psi}(s) 
		& = - \hat{f}\big(s,\cdot,\widehat{\mathscr{Y}}(s),\widehat{\psi}^1(s),\ldots,\widehat{\psi}^m(s)\big) ds
		+ \sum_{k=1}^m \widehat{\phi}^k(s) dW_k(s),  
		\quad s \in [0,T],  \\
		\widehat{\psi}(T) & = -\tfrac12 \gamma^{\frac12}(T) \xi  
	\end{split}
\end{equation}
with the driver 
\begin{equation}\label{eq:driver_of_linear_BSDE_KT}
	\begin{split}
		& \hat{f}\big(s,\omega,\widehat{\mathscr{Y}}(s,\omega),\widehat{\psi}^1(s,\omega),\ldots,\widehat{\psi}^m(s,\omega)\big) \\
		& = \big(\mathscr{A}+\mathscr{B} (\procB + \widehat\theta) \big)^\top \widehat{\psi}  
		+ \sum_{k=1}^m \big( I_n-2 (\procB 
		+ \widehat{\theta}) \big)^\top \mathscr{C}^k 
		\big( \widehat{\phi}^k + 2  \widehat{\mathscr{Y}} \mathscr{C}^k \procB \gamma^{\frac12} \zeta \big) \\
		& \quad 
		-  \mathscr{Q} (I_n-\procB) \gamma^{\frac12} \zeta 
		- \widehat{\mathscr{Y}} \mathscr{B} \procB \gamma^{\frac12} \zeta 
		+ 2 \sum_{k=1}^m \widehat{\mathscr{Z}}^k \mathscr{C}^k \procB \gamma^{\frac12} \zeta  
		, 
	\end{split}
\end{equation}
where on the right-hand side of the equation for the driver we suppressed the dependence on $\omega\in\Omega$ and $s \in [0,T]$. 
A pair $(\widehat{\psi},\widehat{\phi})$ with $\widehat{\phi}=(\widehat{\phi}^1,\ldots,\widehat{\phi}^m)$ is called 
a solution of the BSDE~\eqref{eq:BSDE_linear_KT} if 
\begin{itemize}
	\item the process $\widehat{\psi} \colon [0,T]\times \Omega \to \R^n$ is adapted and continuous and satisfies that $E[\sup_{s \in [0,T]} \lVert \widehat{\psi}(s) \rVert_F^2]<\infty$, 
	
	\item for every $k\in\{1,\ldots,m\}$ it holds that the process $\widehat{\phi}^k\colon [0,T]\times \Omega \to \R^n$ is progressively measurable and satisfies $E[\int_0^T \lVert \widehat{\phi}^k(s) \rVert_F^2 ds] < \infty$,
	
	\item the BSDE~\eqref{eq:BSDE_linear_KT} is satisfied $P$-a.s.
\end{itemize}

Note that the BSDE~\eqref{eq:BSDE_linear_KT} with the driver \eqref{eq:driver_of_linear_BSDE_KT} is linear, but it is not assured that the coefficients are bounded. 

\begin{propo}\label{propo:existenceLinearBSDE}
	Let \cref{assump_filtration} and \cref{ass:procB} be in force. 
	Assume that	$\mathscr{Q}(I_n-\procB)$ and $\KK$ are $dP\times ds$-a.e.\ $\cS_{\ge 0}^n$-valued.
	
	(i) 
	Suppose that there exists $\delta \in (0,\infty)$ such that $\KK - \delta I_n$ is $dP\times ds$-a.e.\  $\cS_{\ge 0}^n$-valued. 
	Then there exists a unique solution $(\widehat{\psi},\widehat{\phi})$ of the BSDE~\eqref{eq:BSDE_linear_KT}.
	
	(ii) 
	Suppose that there exists $\delta \in (0,\infty)$ such that $\sum_{k=1}^m \diagsigma_{k} \diagsigma_k - \delta I_n$ is $dP\times ds$-a.e.\ $\cS_{\ge 0}^n$-valued.  
	Then there exists a unique solution $(\widehat{\psi},\widehat{\phi})$ of the BSDE~\eqref{eq:BSDE_linear_KT}.
\end{propo}

\begin{proof}
	Observe that \cref{assump_filtration} ensures that $(\cF_s)_{s\in[0,T]}$ is the augmented natural filtration of the Brownian motion $W$. 
	
	(i) 
	First, note that \cref{propo:existenceRiccatiBSDE} implies that there exists a unique solution $(\widehat{\mathscr{Y}},\widehat{\mathscr{Z}})$ of the Riccati BSDE~\eqref{eq:BSDE_Riccati_KT} and that $\widehat{\mathscr{Y}}$ is $\cS_{\ge 0}^n$-valued. 
	Furthermore, \cref{ass:procB}, \cref{rem:remark_on_setting_assumptions2}, the fact that $\mathscr{Q}(I_n-\procB)$ and $\KK - \delta I_n$ are $dP\times ds$-a.e.\ $\cS_{\ge 0}^n$-valued, \eqref{eq:int_cond_terminal_pos}, and  \eqref{eq:int_cond_zeta} 
	ensure that 
	the assumptions (A1), (A5), (A6), and (A7) in Kohlmann \& Tang \cite{kohlmann2003minimization} are satisfied.
	Noting \cite[Section~4]{kohlmann2003minimization}, we can thus apply \cite[Theorem~3.11]{kohlmann2003minimization}, which yields the existence of a unique solution of the linear BSDE~\eqref{eq:BSDE_linear_KT}.
	
	(ii)
	\cref{propo:existenceRiccatiBSDE} ensures that there exists a unique solution $(\widehat{\mathscr{Y}},\widehat{\mathscr{Z}})$ of the Riccati BSDE~\eqref{eq:BSDE_Riccati_KT} and that there exists $\varepsilon \in (0,\infty)$ such that  $\widehat{\mathscr{Y}}-\varepsilon I_n$ is $\cS_{\ge 0}^n$-valued. 
	Furthermore, \cref{ass:procB}, \cref{rem:remark_on_setting_assumptions2}, the fact that $\mathscr{Q}(I_n-\procB)$, $\sum_{k=1}^m \diagsigma_{k} \diagsigma_k - \delta I_n$, and $\KK$ are $dP\times ds$-a.e.\ $\cS_{\ge 0}^n$-valued, 
	\eqref{eq:int_cond_terminal_pos}, and  \eqref{eq:int_cond_zeta} 
	guarantee that 
	the assumptions (A1), (A2), (A5), (A6)', and (A7) in \cite{kohlmann2003minimization} are satisfied.
	Therefore, \cite[Theorem~3.11]{kohlmann2003minimization} implies the existence of a unique solution of the linear BSDE~\eqref{eq:BSDE_linear_KT}.
\end{proof}

Given a solution $(\widehat{\mathscr{Y}},\widehat{\mathscr{Z}})$ of the BSDE~\eqref{eq:BSDE_Riccati_KT} and a solution $(\widehat{\psi},\widehat{\phi})$ of the BSDE~\eqref{eq:BSDE_linear_KT}, we define the $\R^n$-valued progressively measurable process $\widehat\theta^0=(\widehat\theta^0(s))_{s \in [0,T]}$ by, for all $s\in[0,T]$,  
\begin{equation}\label{eq:def_theta_hat0}
	\begin{split}
		\widehat\theta^0(s) & = 
		- \bigg( \KK(s) + 4 \sum_{k=1}^m \mathscr{C}^k(s) \widehat{\mathscr{Y}}(s) \mathscr{C}^k(s)  \bigg)^{-1} \\
		& \quad \cdot 
		\bigg( (\mathscr{B}(s))^\top \widehat{\psi}(s) 
		- 2 \sum_{k=1}^m  \mathscr{C}^k(s)  \big[ \widehat{\phi}^k(s) + 2  \widehat{\mathscr{Y}}(s) \mathscr{C}^k(s) \procB(s) \gamma^{\frac12}(s) \zeta(s) \big] \bigg). 
	\end{split}
\end{equation}

Further, given a solution $(\widehat{\mathscr{Y}},\widehat{\mathscr{Z}})$ of the BSDE~\eqref{eq:BSDE_Riccati_KT} and a solution $(\widehat{\psi},\widehat{\phi})$ of the BSDE~\eqref{eq:BSDE_linear_KT} we consider the $\R^n$-valued SDE 
\begin{equation}\label{eq:def_scHhat_opt}
	\begin{split}
		d\scHhat^{*}(s) & = 
		\big[\mathscr{A}(s) + \mathscr{B}(s) \big(\procB(s) + \widehat{\theta}(s) \big) \big] 
		\scHhat^{*}(s) ds 
		+ \mathscr{B}(s) \big[\procB(s) \gamma^{\frac12}(s) \zeta(s) - \widehat{\theta}^0(s) \big] ds 
		\\
		& \quad 
		+ \sum_{k=1}^m \mathscr{C}^k(s) \big(I_n - 2 \big(\procB(s) + \widehat{\theta}(s) \big) \big) \scHhat^{*}(s) dW_k(s)  \\
		& \quad + 2 \sum_{k=1}^m \mathscr{C}^k(s) \big( \widehat{\theta}^0(s) - \procB(s) \gamma^{\frac12}(s) \zeta(s) \big) dW_k(s), \quad s\in[0,T],\\
		\scHhat^{*}(0) & = \gamma^{-\frac12}(0) d - \gamma^{\frac12}(0) x. 
	\end{split}
\end{equation}

We can now solve under appropriate conditions the LQ stochastic control problem given by \eqref{eq:def_cost_fct_LQ_hat} and \eqref{eq:def_scHhat}. 
For this, we rely on Kohlmann \& Tang~\cite{kohlmann2003minimization}. 

\begin{propo}\label{propo:soln_LQ_KT}
	Let \cref{assump_filtration} and \cref{ass:procB} be in force. 
	Assume that	$\mathscr{Q}(I_n-\procB)$ and $\KK$ are $dP\times ds$-a.e.\ $\cS_{\ge 0}^n$-valued.
	Suppose that there exists $\delta \in (0,\infty)$ 
	such that $\KK - \delta I_n$ is $dP\times ds$-a.e.\  $\cS_{\ge 0}^n$-valued, 
	or such that $\sum_{k=1}^m \diagsigma_{k} \diagsigma_k - \delta I_n$ is $dP\times ds$-a.e.\ $\cS_{\ge 0}^n$-valued. 
	Let $(\widehat{\mathscr{Y}},\widehat{\mathscr{Z}})$ be the unique solution of the Riccati BSDE~\eqref{eq:BSDE_Riccati_KT} (cf.\ \cref{propo:existenceRiccatiBSDE}). 
	Recall the definition~\eqref{eq:def_theta_hat} of $\widehat{\theta}$ and let $(\widehat{\psi},\widehat{\phi})$ be the unique solution of the linear BSDE~\eqref{eq:BSDE_linear_KT} (cf.\ \cref{propo:existenceLinearBSDE}). 
	Recall the definition~\eqref{eq:def_theta_hat0} of $\widehat{\theta}^0$. 
	
	(i)
	There exists a unique $\hat{u}^*\in\cL^2$ such that for all $u\in\cL^2$ it holds $P$-a.s.\ that $\LQJhat(\hat{u}^*)\le \LQJhat(u)$. 
	Moreover, there exists a unique solution $\scHhat^*$ of the SDE~\eqref{eq:def_scHhat_opt} (the state process associated to~$\hat{u}^*$), and the unique optimal control~$\hat{u}^*$ admits the representation 
	\begin{equation*}
		\hat{u}^*(s) = \widehat{\theta}(s) \scHhat^*(s) - \widehat{\theta}^0(s), 
		\quad s \in[0,T]. 
	\end{equation*}
	
	(ii)
	Let 
	\begin{equation}\label{eq:extracosts}
		\begin{split}
			\widehat{V}^0 
			& = \frac12 E[ \lVert \gamma^{\frac12}(T) \xi \rVert_F^2] 
			+ E\bigg[ \int_0^T (\gamma^{\frac12}(s) \zeta(s))^\top \mathscr{Q}(s) (I_n - \procB(s)) (\gamma^{\frac12}(s) \zeta(s)) \, ds \bigg] \\
			& \quad - 2 E\bigg[ \int_0^T (\mathscr{B}(s) \procB(s) \gamma^{\frac12}(s) \zeta(s))^\top \widehat{\psi}(s) \, ds \bigg] \\
			& \quad + 4 E\bigg[ \int_0^T \sum_{k=1}^m (\mathscr{C}^k(s) \procB(s) \gamma^{\frac12}(s) \zeta(s))^\top 
			\widehat{\mathscr{Y}}(s)
			\mathscr{C}^k(s) \procB(s) \gamma^{\frac12}(s) \zeta(s) \,  ds \bigg] \\
			& \quad - 2 E\bigg[ \int_0^T \sum_{k=1}^m (\mathscr{C}^k(s) \procB(s) \gamma^{\frac12}(s) \zeta(s))^\top \widehat{\phi}^k(s) \,  ds \bigg] \\
			& \quad - E\bigg[ \int_0^T (\widehat{\theta}^0(s))^\top \bigg( \KK(s) + 4 \sum_{k=1}^m \mathscr{C}^k(s) \widehat{\mathscr{Y}}(s) \mathscr{C}^k(s) \bigg) \widehat{\theta}^0(s)  \, ds \bigg] .
		\end{split}
	\end{equation}
	Then 
	\begin{equation*}
		\begin{split}
			& \inf_{\hat{u}\in\cL^2} \LQJhat(\hat{u} ) \\
			& = \bigl( \gamma^{-\frac12}(0) d - \gamma^{\frac12}(0) x \bigr)^\top \widehat{\mathscr{Y}}(0) \bigl( \gamma^{-\frac12}(0) d - \gamma^{\frac12}(0) x \bigr) 
			- 2 \bigl( \gamma^{-\frac12}(0) d - \gamma^{\frac12}(0) x \bigr)^\top \widehat{\psi}(0) 
			+ \widehat{V}^0 .
		\end{split}
	\end{equation*}
\end{propo}

\begin{proof}
	(i) 
	It follows from Kohlmann \& Tang  \cite[Theorem~3.8]{kohlmann2003minimization} (note also \cite[Section~4]{kohlmann2003minimization} concerning the assumptions)
	that there exists a unique optimal control $\hat{u}^* \in \cL^2$ for the LQ stochastic control problem given by \eqref{eq:def_cost_fct_LQ_hat} and \eqref{eq:def_scHhat} and that $\hat{u}^*$ is given by 
	\begin{equation}\label{eq:1143}
		\hat{u}^* = \widehat{\theta} \scHhat^{\hat{u}^*} - \widehat{\theta}^0 .
	\end{equation}
	From \eqref{eq:def_scHhat}, \eqref{eq:def_scHhat_opt}, and \eqref{eq:1143} we conclude that $\scHhat^{\hat{u}^*}$ is a solution of the SDE~\eqref{eq:def_scHhat_opt}, which is unique by, for example, \cite{galchuk1979} (see also \cite[Lemma~7.1]{tang2003generalLQ}).
	
	(ii)
	The formula for the optimal costs also follows from \cite[Theorem~3.8]{kohlmann2003minimization} (note also \cite[Section~4]{kohlmann2003minimization} concerning the assumptions).
\end{proof}

\subsection{Solution of the trade execution problem for progressively measurable strategies}\label{sec:general_targets_soln_trade_execution}

In \cref{cor:soln_opt_trade_execution_KT} we state the solution of the trade execution problem of \cref{sec:pm_problem} with general targets $\xi$ and $\zeta$. 
This result is a consequence of \cref{propo:soln_LQ_KT}, \cref{cor:uoptimaliffuhatoptimal}, \cref{lem:trafouvsuhat},  and \cref{cor:linkLQpm}.

\begin{corollary}\label{cor:soln_opt_trade_execution_KT}
	Let \cref{assump_filtration} and \cref{ass:procB} be in force. 
	Assume that $\mathscr{Q}(I_n-\procB)$ and $\KK$ are $dP\times ds$-a.e.\ $\cS_{\ge 0}^n$-valued.
	Suppose that there exists $\delta \in (0,\infty)$ 
	such that $\KK - \delta I_n$ is $dP\times ds$-a.e.\  $\cS_{\ge 0}^n$-valued, 
	or such that $\sum_{k=1}^m \diagsigma_{k} \diagsigma_k - \delta I_n$ is $dP\times ds$-a.e.\ $\cS_{\ge 0}^n$-valued. 
	Let $(\widehat{\mathscr{Y}},\widehat{\mathscr{Z}})$ be the unique solution of the Riccati BSDE~\eqref{eq:BSDE_Riccati_KT} (cf.\ \cref{propo:existenceRiccatiBSDE}). 
	Recall the definition~\eqref{eq:def_theta_hat} of $\widehat{\theta}$ and let $(\widehat{\psi},\widehat{\phi})$ be the unique solution of the linear BSDE~\eqref{eq:BSDE_linear_KT} (cf.\ \cref{propo:existenceLinearBSDE}). 
	Recall the definition~\eqref{eq:def_theta_hat0} of $\widehat{\theta}^0$ and let $\scH^*$ be the unique solution of the SDE~\eqref{eq:def_scHhat_opt} (cf.\ \cref{propo:soln_LQ_KT}).  
	
	Then 
	there exists a unique (up to $dP\times ds$-null sets) minimizer $X^*$ of $\pmJ$ in $\cA^{pm}$. 
	Moreover, it holds that 
	\begin{equation*}
		\begin{split}
			& X^*(0-)=x, \quad X^*(T)=\xi,\\
			& X^*(s)
			= \gamma^{-\frac12}(s)
			\big( \big( \widehat{\theta}(s) + \procB(s) - I_n \big) \scHhat^*(s) 
			+ \procB(s) \gamma^{\frac12}(s) \zeta(s) - \widehat{\theta}^0(s) \big),
			\quad s\in[0,T) . 
		\end{split}
	\end{equation*}
	The deviation process $D^*:= D^{X^*}$ (defined in \eqref{eq:def_deviation_pm}) satisfies 
	that 
	\begin{equation*}
		D^*(s) 
		= \gamma^{\frac12}(s) \big( \big( \widehat{\theta}(s) + \procB(s) \big) \scHhat^*(s) 
		+ \procB(s) \gamma^{\frac12}(s) \zeta(s) - \widehat{\theta}^0(s) \big) , 
		\quad s \in [0,T). 
	\end{equation*}
	The optimal costs are given by 
	\begin{equation*}
		\begin{split}
			\inf_{X \in \cA^{pm}} \pmJ(X) 
			& = \bigl( \gamma^{-\frac12}(0) d - \gamma^{\frac12}(0) x \bigr)^\top \widehat{\mathscr{Y}}(0) \bigl( \gamma^{-\frac12}(0) d - \gamma^{\frac12}(0) x \bigr) \\
			& \quad - 2 \bigl( \gamma^{-\frac12}(0) d - \gamma^{\frac12}(0) x \bigr)^\top \widehat{\psi}(0) 
			+ \widehat{V}^0
			- \tfrac12 d^\top \gamma^{-1}(0) d 
			,
		\end{split}
	\end{equation*}
	with $\widehat{V}^0$ from \eqref{eq:extracosts}.
\end{corollary}

\begin{remark}\label{rem:ontheoptimalsoln}
	Suppose that $\xi=0$ and $\zeta=0$ in the setting of 
	\cref{cor:soln_opt_trade_execution_KT}. 
	Then $(\widehat{\psi},\widehat{\phi})=(0,0)$ is the unique solution of the linear BSDE~\eqref{eq:BSDE_linear_KT} and it follows that $\widehat{\theta}^0=0$ in \eqref{eq:def_theta_hat0} and $\widehat{V}^0=0$ in \eqref{eq:extracosts}. 
\end{remark}

\section{Further examples}
\label{sec:appendix:examples}

In this section we complement \cref{sec:examples} by figures and further examples.

\subsection{On the multi-asset variant of the Obizhaeva--Wang model}
\label{sec:multiasset_OW}

Recall \cref{set:OW} where we consider a multi-asset variant of the Obizhaeva--Wang model \cite{obizhaeva2013optimal}  
(in particular, we take the price impact $\gamma$ and the resilience $\rho$ to be constant in our setting of \cref{sec:setting}).  
In the next example, we illustrate that 
it is not sufficient for the existence of an optimal execution strategy in \cref{set:OW} to assume that $\gamma \in \cS_{>0}^n$ and $\rho \in \cS_{>0}^n$.

\begin{ex}\label{exa:notposdefwhileeachposdef}
	In \cref{set:OW} let $T\in (0,\tfrac25)$, $n=2$, $d=0$,  
	\begin{equation*}
		\gamma = 
		\begin{pmatrix}
			2&1\\
			1&1
		\end{pmatrix}, 
		\quad \text{and} \quad 
		\rho = 
		\begin{pmatrix}
			1&2\\
			2&5
		\end{pmatrix}
		.
	\end{equation*}
	Note that it holds that $\gamma\in\cS_{>0}^2$ and that also $\rho \in \cS_{>0}^2$. 
	We denote $e_1=(1,0)^\top \in \R^2$. 
	For every $k \in \N$ we consider the $\R^2$-valued function $X^k=(X^k(s))_{s\in[0,T]}$ defined by $X^k(0-)=x$, $X^k(T)=0$, and 
	\begin{equation*}
		X^k(s) = 
		x+  k \gamma^{-1} e_1 + s k \gamma^{-1} \rho e_1
		= 
		x + k 
		\begin{pmatrix}
			1\\-1
		\end{pmatrix}
		+ s k 
		\begin{pmatrix}
			-1\\3
		\end{pmatrix}
		, \quad s\in[0,T). 
	\end{equation*} 
	It holds for all $k\in\N$ that the associated  $D^k\equiv D^{X^k}$ given by $D^k(0-)=d=0$ and~\eqref{eq:def_deviation_pm} 
	(or, equivalently, \eqref{eq:deviationdynmultivariate}) 
	satisfies for all $r\in[0,T)$ that 
	$D^k(r)=(k,0)^\top$ and $D^k(T)=-\gamma x - T \rho (k,0)^\top$.
	In particular, it follows for all $k\in\N$ that $X^k$ is an admissible trade execution strategy (in $\cA^{fv}$ and $\cA^{pm}$).  
	Note that when, for $k\in\N$, we use the strategy $X^k$, then we 
	at the initial time buy $k$ units in the first asset and concurrently sell $k$ units in the second asset. 
	The initial block trade for all $k\in\N$ contributes  
	\begin{equation}\label{eq:1536a}
		\frac12 (\Delta X^k(0) )^\top \gamma \Delta X^k(0)
		= \frac{k^2}{2} > 0
	\end{equation}
	to the execution costs $C(X^k)$. 
	During the time interval $(0,T)$ we then for all $k\in\N$ sell in the first asset with rate $k$ and we buy in the second asset with the faster rate $3k$. 
	We thereby exploit the lower price impact and the higher resilience in the second asset with, in some sense, not too adverse side effects on the less liquid first asset.
	Trading during the time interval $(0,T)$ for all $k\in\N$ contributes 
	\begin{equation}\label{eq:1536b}
		\int_{(0,T)} (D^k(s-))^\top dX^k(s)
		= T \begin{pmatrix}
			k & 0
		\end{pmatrix}
		\begin{pmatrix}
			-k \\
			3k
		\end{pmatrix}
		= - T k^2 < 0 
	\end{equation}
	to the execution costs, 
	where we observe that the trading in the second asset does not enter the costs since the deviation in the second asset is kept at~$0$. 
	Finally, for all $k\in\N$ the final trade to close the position is given by
	\begin{equation*}
		\Delta X^k(T) = \begin{pmatrix}
			x_1+(1-T)k\\
			x_2+(3T-1)k
		\end{pmatrix}
	\end{equation*}
	and contributes 
	\begin{equation}\label{eq:1536c}
		\begin{split}
			& (D^k(T-))^\top \Delta X^k(T) + \frac12 (\Delta X^k(T))^\top \Delta X^k(T) \\
			& = 
			\frac{k^2}{2} (5T^2-1)
			+ \frac{k}{2} (2Tx_1+4Tx_2) 
			+ \frac12 (2x_1^2+2x_1x_2+x_2^2)
		\end{split}
	\end{equation}
	to the execution costs. 
	If $k \in\N$ is chosen large enough, then~\eqref{eq:1536c} is negative due to $T<\tfrac25$. 
	We can moreover choose $k\in\N$ large enough such that the cost benefit (negative costs) during the time interval $(0,T]$ overcompensates the costs of the trade at the initial time $0$ so that $C(X^k)<0$.
	Furthermore, we have from \eqref{eq:1536a}, \eqref{eq:1536b}, and~\eqref{eq:1536c} that we can produce arbitrarily large negative costs  $\pmJ(X^{k}) = C(X^k) \to -\infty$ as $k\to\infty$. 
\end{ex}

In general the existence of an optimal execution strategy is ensured under the assumptions of \cref{cor:soln_opt_trade_execution}. 
In this regard, note in particular \cref{assump_convexity}. 
Observe that in 
\cref{set:OW}, we have from 
\cref{rem:sufficient_cond_convex} and  \cref{rem:necessaryforconvex} that 
\cref{assump_convexity} is satisfied 
if and only if $\kap = \frac12 (\gamma^{-\frac12} \rho \gamma^{\frac12} + \gamma^{\frac12} \rho^\top \gamma^{-\frac12} ) \in \cS_{>0}^n$.
In \cref{exa:notposdefwhileeachposdef} this condition is indeed violated, as we can demonstrate that $\kap$ is indefinite in this case.

In the next remark we discuss sufficient conditions which ensure that $\kap$ in the Obizhaeva--Wang model (\cref{set:OW}) is positive definite.  

\begin{remark}\label{rem:conley}
	Consider \cref{set:OW}. 
	The representation 
	$\kap = \tfrac12 \gamma^{-\frac12} ( \rho \gamma + \gamma \rho^\top ) \gamma^{-\frac12}$ 
	and the fact that $\gamma^{-\frac12} \in \cS_{>0}^n$ 
	show that $\kap \in \cS_{>0}^n$ 
	is satisfied if and only if $\rho \gamma + \gamma \rho^\top \in \cS_{>0}^n$. 
	If $\rho \in\cS_{>0}^n$ and $\rho\gamma=\gamma\rho$, then this condition holds true. 
	Moreover, we can use Conley et al.~\cite[Theorem 2.1]{conley2005elliptic} to provide a sufficient criterion without the requirement that~$\gamma$ and~$\rho$ commute:  
	Suppose that $\rho \in \cS_{>0}^n$ and denote by $\eta_1(\rho)$, $\eta_n(\rho)$ a smallest, respectively largest, eigenvalue of $\rho$. 
	Analogously, let $\eta_1(\gamma)$, $\eta_n(\gamma)$ denote a smallest, respectively largest, eigenvalue of $\gamma$. 
	If 
	\begin{equation*}
		\Bigg( \sqrt{\frac{\eta_n(\gamma)}{\eta_1(\gamma)}}-1 \Bigg)
		\Bigg( \sqrt{\frac{\eta_n(\rho)}{\eta_1(\rho)}}-1 \Bigg) < 2, 
	\end{equation*}
	then it holds that $\kap \in \cS_{>0}^n$. 
\end{remark}

We apply this result in the next example.

\begin{ex}\label{ex:application_conley}
	In \cref{set:OW} let $n=2$, $d=0$,  
	\begin{equation*}
		\gamma = 
		\begin{pmatrix}
			2&1\\
			1&1
		\end{pmatrix}, 
		\quad \text{and} \quad 
		\rho = 
		\begin{pmatrix}
			3&2\\
			2&5
		\end{pmatrix}
		.
	\end{equation*}
	Note that the only difference to \cref{exa:notposdefwhileeachposdef} is the entry $3$, respectively $1$, in $\rho$. 
	As in \cref{exa:notposdefwhileeachposdef}, it holds that $\gamma \in \cS_{> 0}^n$ and $\rho \in \cS_{> 0}^n$. Moreover, $\rho$ and $\gamma$ do not commute. 
	The eigenvalues of $\rho$ in the present example are $4+\sqrt{5}$ and $4-\sqrt{5}$. 
	For $\gamma$ we have the eigenvalues $\tfrac12 (3+\sqrt{5})$ and $\tfrac12 (3-\sqrt{5})$. 
	We compute that 
	\begin{equation*}
		\begin{split}
			& \Bigg( \sqrt{\frac{\tfrac12 (3+\sqrt{5})}{\tfrac12 (3-\sqrt{5})}} -1 \Bigg)
			\Bigg( \sqrt{\frac{4+\sqrt{5}}{4-\sqrt{5}}}-1 \Bigg) < 2
		\end{split}
	\end{equation*}
	and conclude via \cref{rem:conley} that $\kap\in\cS_{> 0}^n$. Thus, \cref{cor:soln_OW} applies and in contrast to \cref{exa:notposdefwhileeachposdef} we obtain finite minimal costs and the existence of an optimal strategy.
	Moreover, the present example illustrates that to be able to apply \cref{cor:soln_OW}, it is not necessary to assume that $\rho$ and $\gamma$ commute. 
\end{ex}

Furthermore, note that $\rho$ does not need to be symmetric for \cref{cor:soln_OW} to apply:

\begin{ex}
	In \cref{set:OW} let $n=2$, $d=0$,  
	\begin{equation*}
		\gamma = 
		\begin{pmatrix}
			2&1\\
			1&1
		\end{pmatrix}, 
		\quad \text{and} \quad 
		\rho = 
		\begin{pmatrix}
			4&2\\
			3&5
		\end{pmatrix}
		.
	\end{equation*}
	It then holds that $\gamma \in \cS_{> 0}^n$ and 
	\begin{equation*}
		\rho \gamma + \gamma \rho^\top 
		=
		\begin{pmatrix}
			20&17\\
			17&16
		\end{pmatrix} 
		\in \cS_{> 0}^n, 
	\end{equation*}
	which implies that $\kap \in \cS_{> 0}^n$. 
	Now we can apply \cref{cor:soln_OW} and obtain the existence of a unique optimal strategy. 
\end{ex}

In some subsettings of \cref{set:OW}, the formulas in \cref{cor:soln_OW} simplify.  
For example, consider the condition $\rho\gamma = \gamma \rho^\top$, which within \cref{set:OW} is equivalent to $\mathscr{B}=\mathscr{B}^\top$. 
It then follows that $\KK = - \mathscr{B}$.
We further obtain that $\KK^{-1} \mathscr{B}^\top = - I_n$ and 
$\mathscr{B} \KK^{-1} \mathscr{B}^\top = - \mathscr{B}$. 
If, in addition, $\rho$ is symmetric, then it holds that $\mathscr{B}=-\rho$ and $\kap=\rho$. 
Therefore, we have the following corollary. 

\begin{corollary}\label{cor:OWsymcommute}
	Assume \cref{set:OW}, that $\rho \in \cS_{>0}^n$, and that $\rho\gamma = \gamma \rho$. 
	Then there exists a unique optimal strategy $X^* \in \cA^{pm}$ that minimizes 
	$\pmJ$, 
	and it holds that $X^* \in \cA^{fv}$.  
	The optimal strategy $X^*$ and the associated deviation $D^*$ satisfy 
	\begin{equation*}
		\begin{split}
			X^*(s) & = 
			\bigl( I_n + (T-s) \rho \bigr) 
			\bigl(2I_n + T \rho \bigr)^{-1} \bigl( x - \gamma^{-1} d \bigr)
			, \quad s\in[0,T) , \\
			D^*(s) &= \gamma^{\frac12} \bigl(2I_n + T \rho \bigr)^{-1} \bigl( \gamma^{-\frac12} d - \gamma^{\frac12} x \bigr),
			\quad  s\in[0,T), \\
			D^*(T) &= 2 \gamma^{\frac12} 
			\bigl(2I_n+ T \rho \bigr)^{-1} \bigl( \gamma^{-\frac12} d - \gamma^{\frac12} x \bigr) .
		\end{split}
	\end{equation*}
\end{corollary}

\begin{remark}
	The fact that $\rho$ in \cref{cor:OWsymcommute} is symmetric and commutes with $\gamma$ implies that there exist an orthogonal matrix $\widetilde{O} \in \R^{n\times n}$ and diagonal matrices $\widetilde \rho, \widetilde\gamma \in \R^{n\times n}$  such that $\rho = \widetilde{O}^\top \widetilde{\rho} \widetilde{O}$ and $\gamma = \widetilde{O}^\top \widetilde{\gamma} \widetilde{O}$ (see, for example, \cite[Chapter~4, Theorem~5]{Bellman1970MatrixAnalysis}). 
	The optimal strategy $X^*$ can thus be represented in the form   
	$\widetilde{O} X^*(0-)= \widetilde{O} x$, $\widetilde{O} X^*(T)=0$, and 
	\begin{equation*}
		\begin{split}
			\widetilde{O} X^*(s) & = 
			\bigl( I_n + (T-s) \widetilde\rho \bigr) 
			\bigl(2I_n + T \widetilde\rho \bigr)^{-1}
			\widetilde{O}
			x \\
			& \quad - 
			\bigl( I_n + (T-s) \widetilde\rho \bigr) 
			\bigl(2I_n + T \widetilde\rho \bigr)^{-1} 
			\widetilde{\gamma}^{-1} 
			\widetilde{O} 
			d
			, \quad s\in[0,T) .
		\end{split}
	\end{equation*}
	Note that in this representation the position $X^*$ and the deviation $D^*$ are orthogonally transformed with $\widetilde{O}$ into $\widetilde{O}X^*$ and $\widetilde{O}D^*$. Moreover, observe that the components of $\widetilde{O} X^*$ are the optimal strategies of certain single-asset optimal trade execution problems. Namely, for each $j\in\{1,\ldots,n\}$ we have that $(\widetilde{O} X^*)_j$ is the optimal strategy in the single-asset model that has the resilience $\widetilde{\rho}_{j,j}$, the price impact $\widetilde{\gamma}_{j,j}$, the initial position $(\widetilde{O} x)_j$, and the initial deviation $(\widetilde{O} d)_j$. 
\end{remark}

Observe that if $d=0$ in the setting of\footnote{or in the subsetting of \cref{set:OW} where $\gamma=\tilde{\gamma} I_n$ for some $\tilde \gamma \in (0,\infty)$} 
\cref{cor:OWsymcommute}, 
then the optimal strategy does not depend on the price impact~$\gamma$.  
However, the optimal strategy can still be influenced by cross-effects between the assets. This is due to the resilience~$\rho$. 
Moreover, although the price impact $\gamma$ does not affect the optimal strategy, it has an impact on the optimal costs. 
Furthermore, note that if, in addition, $\rho$ is a diagonal matrix, then we obtain that the optimal strategy is composed of the optimal strategies for the respective single-asset models. In the next example we  consider such a situation and analyze how the execution costs differ depending on whether cross-impact is taken into account.

\begin{ex}\label{ex:simpleres}
	Within \cref{set:OW} suppose that\footnote{The assumption that $\rho=\wt\rho I_2$ is crucial for our analysis.} $n=2$, $x_1\ne 0 \ne x_2$, $d=0$, $\wt{\rho} \in (0,\infty)$, $\rho=\wt\rho I_2$, and 
	\begin{equation}\label{eq:gamma_123}
		\gamma = 
		\begin{pmatrix}
			\gamma_1 & \gamma_3\\
			\gamma_3 & \gamma_2
		\end{pmatrix}
		,
	\end{equation}
	where $\gamma_1,\gamma_2,\gamma_3 \in \R$ are chosen in such a way that $\gamma \in \cS_{>0}^2$.
	For every $j \in \{1,2\}$ let 
	$X_j=(X_j(s))_{s\in[0,T]}$ be defined by $X_j(0-)=x_j$, $X_j(T)=0$, and 
	\begin{equation*}
		X_j(s) = \frac{\big(1+(T-s)\wt\rho \big) x_j}{2+T\wt\rho}
		,\quad s \in[0,T). 
	\end{equation*}
	Note that for all $j\in\{1,2\}$ it holds that $X_j$ is the optimal strategy in the single-asset model with price impact $\gamma_j$ and resilience $\wt\rho$. 
	Here, we moreover have that $X=(X_1,X_2)^\top \in \cA^{fv}$ is the optimal strategy in the multi-asset model (cf.\ \cref{cor:OWsymcommute}).
	Furthermore, it follows from \cref{cor:soln_OW}  
	that $\mathscr{Y}(s) = (2+(T-s)\wt\rho)^{-1} I_2$, $s\in[0,T]$.
	We thus have from \eqref{eq:optcosts_trade_execution_soln} 
	that the costs of the strategy $X$ 
	are given by 
	\begin{equation*}
		\begin{split}
			x^\top \gamma^{\frac12} \mathscr{Y}(0) \gamma^{\frac12} x
			& = \frac{x^\top \gamma x}{2+T\wt\rho} 
			= \frac{ \gamma_1 x_1^2 + 2 \gamma_3 x_1 x_2 + \gamma_2 x_2^2}{2+T\wt\rho} 
			.
		\end{split}
	\end{equation*}
	An agent who believes that there is no cross-impact, that is, $\gamma_3=0$, computes the costs 
	$(\gamma_1 x_1^2 + \gamma_2 x_2^2)(2+T\wt\rho)^{-1}$, which differ from the real costs by the term $2 \gamma_3 x_1 x_2$ in the numerator. 
	In particular, the costs that take cross-impact into account can be higher or smaller than the no-cross-impact costs, depending on the sign of the cross-impact term~$\gamma_3$ between the assets. 
\end{ex}

\subsection{Cross-effects from the resilience}
\label{sec:appendixcrossresilience}

We continue the discussion of \cref{ex:crossingzero}. 

\begin{ex}\label{ex:crossingzeroappendix}
	Resume the setting of \cref{ex:crossingzero} and
	suppose that\footnote{The case $x_1<0$ can be analyzed analogously with obvious changes.} $x_1> 0$. 
	
	To further analyze the strategy in the second asset, 
	note that by using the facts that $\rho_1,\rho_2>0$ and $\rho_1\rho_2-\rho_3^2>0$ we see that the denominator in the expression for $X^*_2$ in~\eqref{eq:1002ex} is always positive. 
	If $\rho_3$ is positive, then it is optimal to, in the second asset, jump to the positive position 
	\begin{equation*}
		X^*_2(0) = 
		\frac{T  \rho_3 x_1 }{(2+T\rho_1)(2+T\rho_2)-T^2\rho_3^2}
	\end{equation*}
	and to subsequently sell at a constant rate while crossing the position $0$ at the time $T/2$, and to end with a block buy trade of the size 
	\begin{equation*}
		\lvert \Delta X^*_2(T)\rvert = 
		\frac{T  \rho_3 x_1 }{(2+T\rho_1)(2+T\rho_2)-T^2\rho_3^2} .
	\end{equation*} 
	In the case that $\rho_3$ is negative (see also \cref{fig:crossingzeronegrho3}), 
	the optimal strategy entails for the second asset 
	to jump to a negative position, to then buy at a constant rate, switching to a positive position after the first half of the trading period, and to close the position at the terminal time by a sell block trade. 
	
	For the first component of the optimal strategy, note that the trading rate during $[0,T)$ is given by 
	\begin{equation}\label{eq:1308ex}
		\left\lvert \frac{\big[-\rho_1(2+T\rho_2)+T\rho_3^2\big] x_1}{(2+T\rho_1)(2+T\rho_2)-T^2\rho_3^2 } \right\rvert 
	\end{equation} 
	and the initial and terminal block trades are given by
	\begin{equation}\label{eq:1311ex}
		\begin{split}
			\Delta X_1^*(0) 
			& = - \frac{(2+T\rho_2)x_1}{(2+T\rho_1)(2+T\rho_2)-T^2\rho_3^2 }
			= 
			\Delta X_1^*(T) .
		\end{split}
	\end{equation}
	The fact that $-\rho_1(2+T\rho_2)+T\rho_3^2=T(\rho_3^2-\rho_1\rho_2)-2\rho_1<0$ therefore shows that the optimal strategy in the first asset is a pure sell strategy. 
	In particular, the position in the first asset does not become negative. 
	Moreover, note that the strategy in the first asset does not depend on the sign of $\rho_3$. 
	
	Observe that \cref{cor:OWsymcommute}, $d=0$,  and the fact that $\rho$ and $\gamma$ commute show that 
	the deviation $D^*$ associated to $X^*$ is given by 
	\begin{equation*}
		\begin{split}
			D^*(s) & = -(2I_2+T\rho)^{-1} \gamma x, \quad s\in[0,T),
			\quad 
			D^*(T) = -2(2I_2+T\rho)^{-1} \gamma x .
		\end{split}
	\end{equation*}
	This and $x=(x_1,0)^\top$ imply for all $s\in[0,T)$ that 
	\begin{equation}\label{eq:1112ex}
		\begin{split}
			D_1^*(s) & = \frac{\big[T\rho_3 \gamma_3-(2+T\rho_2)\gamma_1 \big] x_1}{(2+T\rho_1)(2+T\rho_2)-T^2\rho_3^2 } , \quad
			D_2^*(s) = \frac{\big[T\rho_3 \gamma_1 - (2+T\rho_1)\gamma_3 \big] x_1}{(2+T\rho_1)(2+T\rho_2)-T^2\rho_3^2 } ,
		\end{split}
	\end{equation}
	where we use for $\gamma$ the same notation as in~\eqref{eq:gamma_123}. 
	
	In the following, let us consider the subsetting where $\gamma_3=0$ and $\rho_3>0$. 
	Then, \eqref{eq:1112ex} shows that the deviation in the first asset is a negative constant during $[0,T)$ and that the deviation in the second asset is a positive constant during $[0,T)$. 
	Note that the price in the first asset during $[0,T]$ is even lower than it would have been for $\rho_3=0$. 
	To mitigate the lower price in the first asset, the selling in the first asset during $[0,T)$ happens at a slower rate than in the case $\rho_3=0$ (cf.~\eqref{eq:1308ex} and~\eqref{eq:1311ex}). 
	The block trades at the initial time $0$ and at the terminal time $T$ in the first asset both have a larger size than in the case $\rho_3=0$ (cf.~\eqref{eq:1311ex}).  	
	In the second asset, where without $\rho_3$ there would have been no trading and thus no induced change in the price, we obtain that the initial buy trade drives the price up. The increased price in the second asset is then exploited by selling in the second asset during $[0,T)$. One even ``oversells'' and has to buy back at the terminal time $T$.
	
	\begin{SCfigure}
		\caption{The optimal strategy in the setting of \cref{ex:crossingzero} for the specific values $T=1$, $x=(100,0)^\top$, $\rho_1=2=\rho_2$, $\rho_3=-1$, and $\gamma=I_2$.}
		\label{fig:crossingzeronegrho3}
		\includegraphics[scale=0.75]{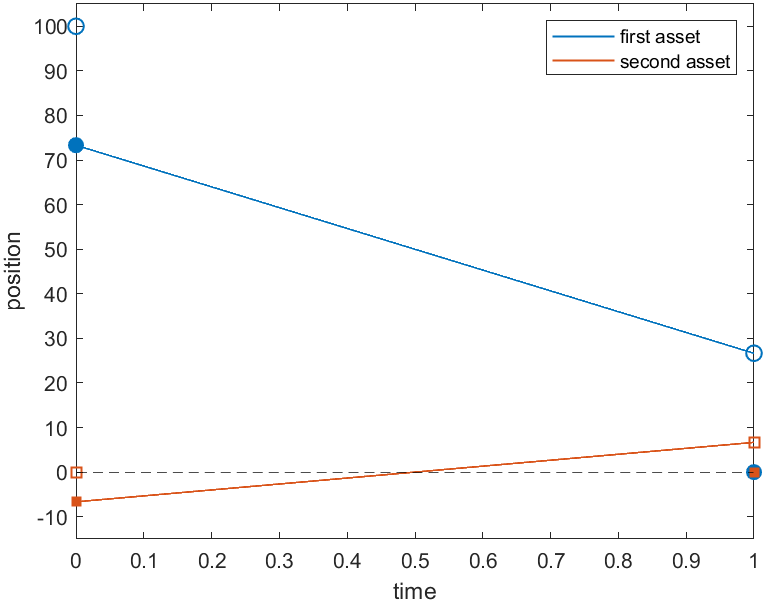}
	\end{SCfigure}
\end{ex}

\begin{remark}
	The effect that trading in an asset can become optimal despite a zero objective in that asset has also been observed by, for example, Abi Jaber et al.~\cite{jaber2024optimal} and Tsoukalas et al.~\cite{TsoukalasGieseckeWang2019}. 
	In particular, note that the optimal position in~\cite[Figure~2]{jaber2024optimal} displays a similar behavior as our optimal position in \cref{fig:crossingzeronegrho3}. 
	Some visible differences come from the fact that Abi Jaber et al.\ optimize over absolutely continuous positions and do not require strict liquidation, but enforce it via a penalization of the terminal position in their cost functional. 
	But also the mechanisms behind the similar observations differ: 
	In our example, the cross-resilience in~$\rho$ is crucial, whereas in their example, they have what corresponds to $\rho=\wt{\rho}I_2$ in our model and the important ingredient is the cross-impact. 
	In the discrete-time multi-asset Obizhaeva--Wang model studied in \cite{TsoukalasGieseckeWang2019}, the cross-effects come from the permanent price impact part and from possible correlations in the fundamental prices, which affect the model via risk-aversion, whereas the resilience and the transient price impact are diagonal. 
	The cross-hedging observation in \cite{TsoukalasGieseckeWang2019} is made in situations where the cross-impact is symmetric and there is risk-aversion, or in situations where the cross-impact is asymmetric (see \cite[Section~4.3.1]{TsoukalasGieseckeWang2019}).
\end{remark}

\subsection{Cross-effects from the risk term}
\label{sec:example_risk}

In this subsection we consider the following 
subsetting of \cref{sec:setting}. 

\begin{setting}\label{set:crossrisk}
	Let $\xi=0$, $\zeta=0$ $\mu = 0$, $\sigma = 0$ and choose a deterministic, constant $\risk \in \cS_{>0}^n$. 
	Let $\rho \in \cS_{>0}^n$ (deterministic, constant) and the orthogonal matrix $\OO \in \R^{n\times n}$ be such that $\rho$ and $\gamma$ commute. 
	Furthermore, assume that $d\ne \gamma(0) x$ and  
	let \cref{assump_filtration} be satisfied.
\end{setting}

Note that in \cref{set:crossrisk} it holds that $\mathscr{Q} \in \cS_{\ge 0}^n$ and that there exists $\delta \in (0,\infty)$ such that $\kap - \delta I_n = \rho - \delta I_n \in \cS_{\ge 0}^n$. 
In particular, \cref{assump_convexity} is satisfied in~\cref{set:crossrisk}. 
We thus have that given \cref{set:crossrisk} the assumptions of \cref{cor:soln_opt_trade_execution} are met. 
Hence, in \cref{set:crossrisk} we can apply \cref{cor:soln_opt_trade_execution} and obtain that there exists a unique optimal strategy in $\cA^{pm}$ for 
$\pmJ$. 
Recall that the optimal strategy is given in terms of the solution of the BSDE~\eqref{eq:BSDE}. 
We have that there exists a unique solution of the matrix Riccati ordinary differential equation 
\begin{equation}\label{eq:ODE_exa_risk}
	\begin{split}
		\frac{dY(s)}{ds} 
		& = - \big( \mathscr{Q} 
		- \big(Y(s)\rho+\mathscr{Q} \big) 
		\big(\mathscr{Q}+\rho\big)^{-1} 
		\big(\rho Y(s) + \mathscr{Q} \big) \big), 
		\quad s \in [0,T],  \quad 
		Y(T) = \tfrac12 I_n
	\end{split}
\end{equation}
(cf., for example, \cite[Corollary~2.10]{yong1999stochasticcontrols}). 
Therefore, in \cref{set:crossrisk} the unique solution of the BSDE~\eqref{eq:BSDE} is given by 
$(\mathscr{Y},0)$ where $\mathscr{Y}\equiv Y$ is the deterministic solution of~\eqref{eq:ODE_exa_risk}. 

In the following example we showcase cross-effects of the risk term on the optimal strategy. 
To this end, we take all other model components in \cref{set:crossrisk} to be rather simple. 

\begin{ex}\label{ex:opt_strat_risk}
	Within \cref{set:crossrisk}, let 
	$n=2$, $d=0$, $x=(100,0)^\top$, $\lambda_1(0)=1=\lambda_2(0)$, $\OO = I_2$, and $\rho=3I_2$. 
	Furthermore, consider the positive definite matrix 
	\begin{equation}\label{eq:opt_strat_risk_def_risk}
		\risk = \begin{pmatrix}
			1 & 0.5\\
			0.5 & 1
		\end{pmatrix}
		.
	\end{equation}
	If the off-diagonal elements in $\risk$ were $0$, then the optimal strategy would be composed of the optimal strategies in the respective single-asset problems. 
	Observe that in the single-asset problem with initial position~$0$, initial deviation~$0$, price impact~$1$, resilience~$3$, and risk coefficient~$1$, the optimal strategy is to stay in the position $0$ (cf., for example, \cref{lem:closing_immed}). 
	Therefore, with $0.5$ replaced by $0$ in \eqref{eq:opt_strat_risk_def_risk}, the optimal strategy would entail no trading in the second asset. 
	However, with $\risk$ chosen as in \eqref{eq:opt_strat_risk_def_risk}, we see from 
	\cref{fig:opt_strat_risk} that it is optimal to 
	take on a negative position in the second asset with an initial sell block trade and to buy back the missing amount of shares over $(0,T]$. In particular, we see how cross-hedging strategies can be implemented in our model through a suitable choice of $\risk$.
	The deviation associated to this optimal strategy is shown in \cref{fig:opt_dev_risk}. 
	\begin{SCfigure}
		\caption{The optimal strategy in the setting of \cref{ex:opt_strat_risk}, where $T=1$, $x=(100,0)^\top$, $\rho=3I_2$, $\gamma=I_2$, and with $\risk$ defined in \eqref{eq:opt_strat_risk_def_risk}.}
		\label{fig:opt_strat_risk}
		\includegraphics[scale=0.75]{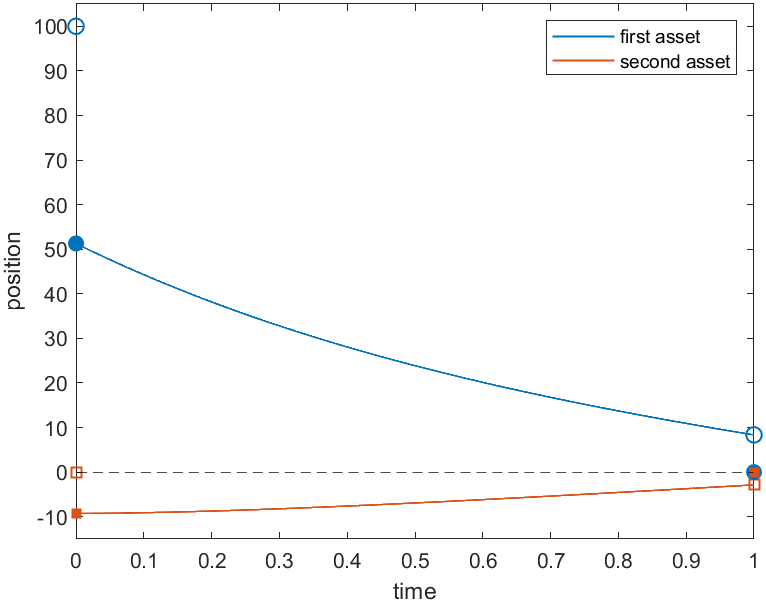}
	\end{SCfigure}
	\begin{SCfigure}
		\caption{The deviation associated to the optimal strategy of \cref{fig:opt_strat_risk} in the same setting as in \cref{fig:opt_strat_risk}.}
		\label{fig:opt_dev_risk}
		\includegraphics[scale=0.75]{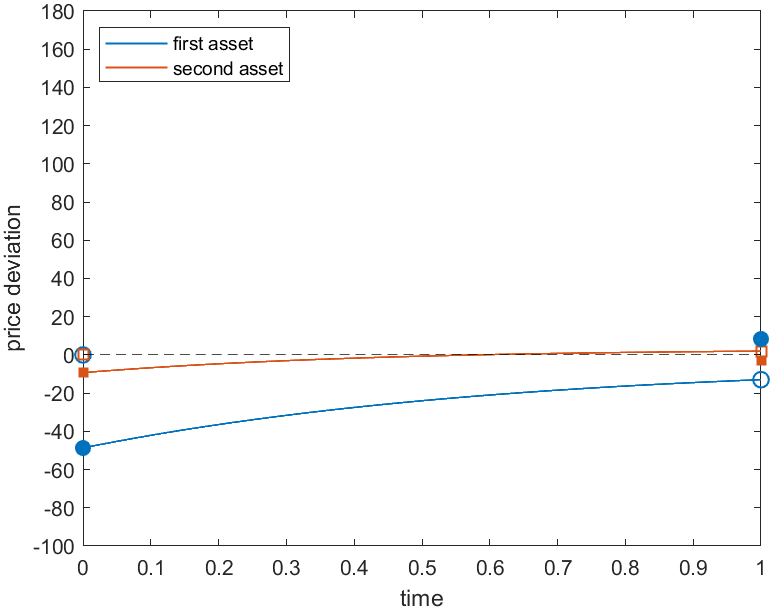}
	\end{SCfigure}
\end{ex}

\subsection{Cross-effects from the price impact}
\label{sec:example_gamma}

In this subsection we consider the following 
subsetting of \cref{sec:setting}. 

\begin{setting}\label{set:crossimpact}
	Let $(\wt{\rho}_1,\dots,\wt{\rho}_n)^\top \in \R^n$, $\mu \in \R^n$, and $\sigma \in \R^{n\times m}$ satisfy for all $j\in\{1,\dots,n\}$ that $2\wt{\rho}_j+\mu_j-\sum_{k=1}^m \sigma_{j,k}^2 >0$, 
	and assume that $\rho=\OO^\top \diag(\wt{\rho}_1,\dots,\wt{\rho}_n) \OO$. 
	Let $\risk=0$, $\xi=0$, and $\zeta=0$. 
	Suppose that \cref{assump_filtration} holds and that $d\ne \gamma(0) x$.
\end{setting}

\begin{lemma}\label{lem:example_cross_impact_assump_satisfied}
	Assume \cref{set:crossimpact}. 
	Then \cref{assump_convexity} is satisfied.  
\end{lemma} 

\begin{proof}
	Note that $\gamma=\OO^\top \diag(\lambda_1,\dots,\lambda_n) \OO$ and $\rho=\OO^\top \diag(\wt{\rho}_1,\dots,\wt{\rho}_n) \OO$ commute. 
	In particular, this yields that $\gamma^{-\frac12} \rho \gamma^{\frac12}$ is $dP\times ds$-a.e.\ bounded 
	and that 
	\begin{equation*}
		\begin{split}
			\kap 
			& = \frac12 \OO^\top \bigg( 2 \diag(\wt{\rho}_1,\dots,\wt{\rho}_n) + \diag(\mu_1,\dots,\mu_n) - \sum_{k=1}^m \diag(\sigma_{1,k}^2,\dots,\sigma_{n,k}^2) \bigg) \OO
		\end{split}
	\end{equation*}
	(cf.~\cref{rem:remark_on_setting_assumptions2}). 
	This representation for $\kap$, the fact that for all $j\in\{1,\dots,n\}$ it holds that $2\wt{\rho}_j+\mu_j-\sum_{k=1}^m \sigma_{j,k}^2 >0$, and $\risk=0$ 
	demonstrate that there exists $\delta \in (0,\infty)$ such that $\KK-\delta I_n = \kap - \delta I_n$ is positive semidefinite. 
	Hence, \cref{assump_convexity} is satisfied (cf.~\cref{rem:sufficient_cond_convex}).
\end{proof}

Due to \cref{lem:example_cross_impact_assump_satisfied} and 
\cref{cor:soln_opt_trade_execution} we obtain that in \cref{set:crossimpact} there exists a unique optimal strategy in $\cA^{pm}$ for 
$\pmJ$. 
Further, note that in \cref{set:crossimpact} the unique solution of the BSDE~\eqref{eq:BSDE} is given by $(\mathscr{Y},0)$ where $\mathscr{Y}\equiv Y$ is the unique solution of the matrix Riccati ordinary differential equation 
\begin{equation}\label{eq:0106}
	\begin{split}
		\frac{dY(s)}{ds} & = 
		- Y(s) \mathscr{A} - \mathscr{A} Y(s) 
		- \sum_{k=1}^m \mathscr{C}^k Y(s) \mathscr{C}^k  
		+ \bigg( Y(s) \mathscr{B} - 2 \sum_{k=1}^m \mathscr{C}^k Y(s) \mathscr{C}^k  \bigg) \\
		& \quad \cdot \bigg( \kap + 4 \sum_{k=1}^m \mathscr{C}^k Y(s) \mathscr{C}^k  \bigg)^{-1} 
		\bigg( \mathscr{B}^\top Y(s) - 2 \sum_{k=1}^m \mathscr{C}^k Y(s) \mathscr{C}^k \bigg), \quad s\in[0,T], \\
		Y(T) & =\tfrac12 I_n 
	\end{split}
\end{equation}
(cf., for example, \cite[Theorem~7.2]{yong1999stochasticcontrols}).

We first present a deterministic example which illustrates cross-effects of the price impact~$\gamma$ on the optimal strategy. 

\begin{ex}\label{ex:opt_strat_gamma}
	Within \cref{set:crossimpact} let 
	$n=2$, $d=0$, $x=(100,0)^\top$, $\lambda_1(0)=1=\lambda_2(0)$, $\wt{\rho}_1=1=\wt{\rho}_2$, $\sigma=0$,  $\mu=(3,1)^\top$, and 
	\begin{equation}\label{eq:OO_in_opt_strat_gamma}
		\OO = \frac15 \begin{pmatrix}
			3& 4\\
			-4& 3 
		\end{pmatrix}
		.
	\end{equation}
	Then, the resilience is given by $\rho=I_2$ and the  price impact is given by 
	\begin{equation}\label{eq:gamma_in_opt_strat_gamma}
		\begin{split}
			\gamma(s)=\OO^\top \diag(e^{3s}, e^{s}) \OO
			& = \frac{1}{25} 
			\begin{pmatrix}
				9 e^{3s} + 16 e^{s} & 12( e^{3s} - e^{s}) \\
				12( e^{3s} - e^{s}) & 16 e^{3s} + 9 e^{s}
			\end{pmatrix}
			, \quad s \in [0,T].
		\end{split}
	\end{equation}
	In particular, $\gamma$ at the initial time is equal to the diagonal matrix~$I_2$, but at later times contains entries on the off-diagonal. 
	We thus have cross-impact. As shown in \cref{fig:opt_strat_gamma}, in the setting of the present example it is  optimal to trade in the second asset despite the initial position $0$ in this asset. 
	\begin{SCfigure}
		\caption{The optimal strategy in the setting of \cref{ex:opt_strat_gamma}, where $T=1$, $x=(100,0)^\top$, $\risk=0$, $\rho=I_2$, and with $\gamma$ from \eqref{eq:gamma_in_opt_strat_gamma}.}
		\label{fig:opt_strat_gamma}
		\includegraphics[scale=0.75]{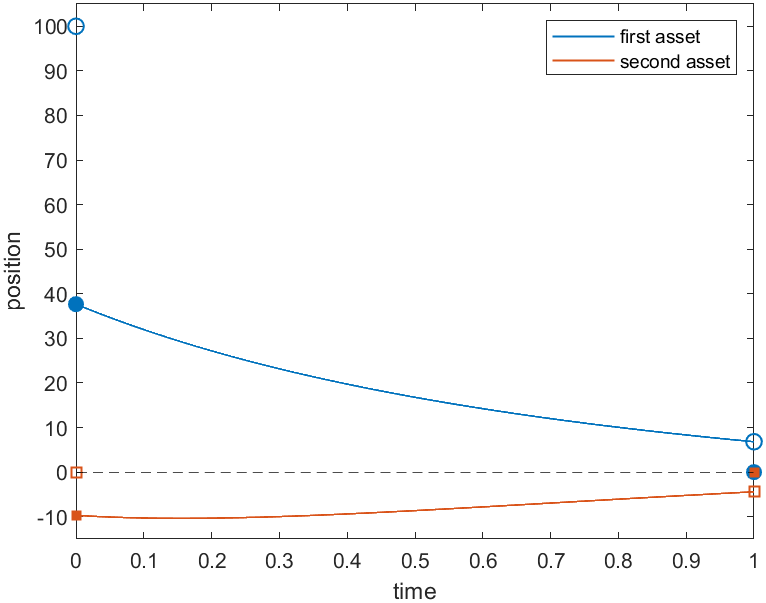}
	\end{SCfigure}
	To explain this behavior, note that after the initial time there is a positive cross-impact between both assets. This means that selling in the first asset leads to lower prices not only in the first asset, but also in the second asset. 
	The lower prices in the second asset are exploited by implementing a buy program in the second asset. 
	Moreover, observe that buying in the second asset increases the prices also in the first asset, which is beneficial for the sell task in the first asset.  According to \cref{fig:opt_dev_gamma}, the optimal strategy leads to stable prices in the sense that the deviation inside the trading interval stays constant.
	\begin{SCfigure}
		\caption{The deviation associated to the optimal strategy of \cref{fig:opt_strat_gamma} in the same setting as in \cref{fig:opt_strat_gamma}.}
		\label{fig:opt_dev_gamma}
		\includegraphics[scale=0.75]{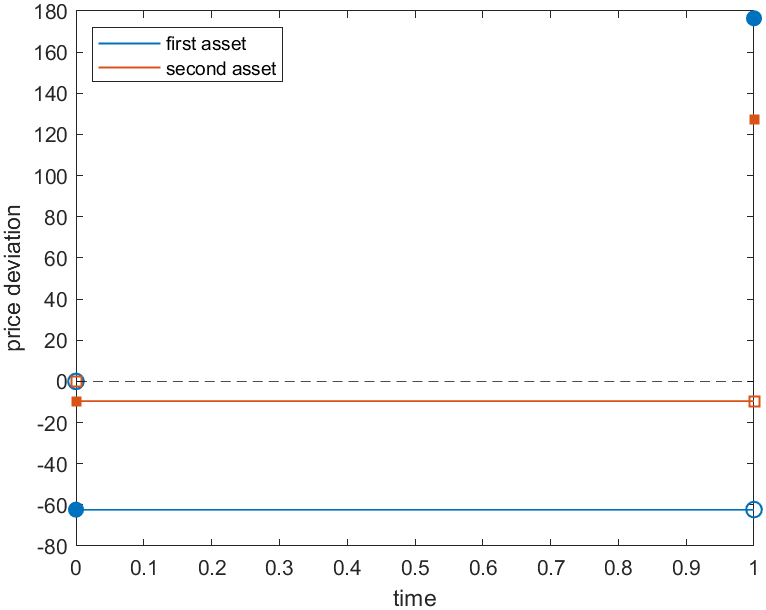}
	\end{SCfigure}
\end{ex}

We can actually show that the optimal deviation is constant on $(0,T)$ in any subsetting of \cref{set:crossimpact} that satisfies $\sigma=0$:

\begin{lemma}\label{lem:example_cross_impact_constant_dev}
	Assume \cref{set:crossimpact} with $\sigma=0$. 
	It then holds that the deviation $D^*$ associated to the optimal strategy $X^*$ is constant on $(0,T)$. 
	Moreover, we have that $X^*\in\cA^{fv}$. 
\end{lemma}

\begin{proof}
	Note that in \cref{set:crossimpact} with $\sigma=0$ we have that $\mathscr{C}^k = 0$, $k\in\{1,\dots,m\}$, 
	$\mathscr{A} = \frac12 \OO^\top \diagmu \OO$, 
	$\KK = \kap = \rho + \mathscr{A}$, 
	and $\mathscr{B} = -\rho - 2 \mathscr{A}$. 
	Therefore, the function $Y$ in~\eqref{eq:0106} satisfies 
	\begin{equation}\label{eq:1417a}
		\frac{dY(s)}{ds}
		= -Y(s) \mathscr{A} - \mathscr{A} Y(s) 
		+ Y(s) (\rho + 2\mathscr{A}) (\rho+\mathscr{A})^{-1} (\rho+2\mathscr{A}) Y(s), \quad s\in[0,T], 
	\end{equation}
	the process $\theta$ defined in \eqref{eq:def_theta} is deterministic and satisfies 
	\begin{equation}\label{eq:1417b}
		\theta(s) = (\rho+\mathscr{A})^{-1} (\rho+2\mathscr{A}) Y(s), \quad s\in[0,T], 
	\end{equation}
	and the optimal state $\scH^*$ is deterministic and satisfies 
	\begin{equation}\label{eq:1417c}
		\frac{d\scH^*(s)}{ds}
		= \big(\mathscr{A}-(\rho+2\mathscr{A})(\rho+\mathscr{A})^{-1}(\rho+2\mathscr{A}) Y(s) \big) \scH^*(s), \quad s\in[0,T].
	\end{equation}
	Moreover, the optimal deviation on $[0,T)$ is given by 
	$D^*(s)=\gamma^{\frac12}(s)\theta(s)\scH^*(s)$, $s\in[0,T)$. 
	Note that $\gamma^{\frac12}$, $\theta$, and $\scH^*$ all are of finite variation in the setting of the present example and therefore $D^*$ is of finite variation. 
	Integration by parts yields for all $s\in(0,T)$ that 
	\begin{equation*}
		\begin{split}
			d D^*(s) & = (d\gamma^{\frac12}(s)) \theta(s) \scH^*(s) 
			+ \gamma^{\frac12}(s) (d\theta(s)) \scH^*(s)
			+ \gamma^{\frac12}(s) \theta(s) d \scH^*(s) .
		\end{split}
	\end{equation*}
	We combine this, \cref{lemma:ito_to_lambda}, \eqref{eq:1417a}, \eqref{eq:1417b}, and \eqref{eq:1417c} to obtain for all $s\in(0,T)$ that
	\begin{equation}\label{eq:1458a}
		\begin{split}
			d D^*(s) & = \OO^\top \lam^{\frac12}(s) \tfrac12 \diagmu \OO \theta(s) \scH^*(s) ds 
			+ \gamma^{\frac12}(s) (\rho+\mathscr{A})^{-1} (\rho+2\mathscr{A}) (dY(s)) \scH^*(s) \\
			& \quad + \gamma^{\frac12}(s) \theta(s) \big(\mathscr{A}-(\rho+2\mathscr{A})(\rho+\mathscr{A})^{-1}(\rho+2\mathscr{A}) Y(s) \big) \scH^*(s) ds \\
			& = \gamma^{\frac12}(s) \mathscr{A} \theta(s) \scH^*(s) ds
			+ \gamma^{\frac12}(s) (\rho+\mathscr{A})^{-1} (\rho+2\mathscr{A}) \\
			& \quad \cdot \big(-Y(s) \mathscr{A} - \mathscr{A} Y(s) 
			+ Y(s) (\rho + 2\mathscr{A}) (\rho+\mathscr{A})^{-1} (\rho+2\mathscr{A}) Y(s)\big) \scH^*(s) ds \\
			& \quad + \gamma^{\frac12}(s) \theta(s) 
			\big(\mathscr{A}-(\rho+2\mathscr{A})(\rho+\mathscr{A})^{-1}(\rho+2\mathscr{A}) Y(s) \big) \scH^*(s) ds \\
			& = \gamma^{\frac12}(s) \mathscr{A} \theta(s) \scH^*(s) ds
			- \gamma^{\frac12}(s) (\rho+\mathscr{A})^{-1} (\rho+2\mathscr{A}) 
			\mathscr{A} Y(s) \scH^*(s) ds \\
			& \quad + \gamma^{\frac12}(s) \theta(s) \big(-\mathscr{A} + (\rho + 2\mathscr{A}) (\rho+\mathscr{A})^{-1} (\rho+2\mathscr{A}) Y(s) \big) \scH^*(s) ds \\
			& \quad + \gamma^{\frac12}(s) \theta(s) 
			\big(\mathscr{A}-(\rho+2\mathscr{A})(\rho+\mathscr{A})^{-1}(\rho+2\mathscr{A}) Y(s) \big) \scH^*(s) ds \\
			& = \gamma^{\frac12}(s) 
			\big( \mathscr{A} (\rho+\mathscr{A})^{-1} (\rho+2\mathscr{A})  
			- (\rho+\mathscr{A})^{-1} (\rho+2\mathscr{A}) 
			\mathscr{A} \big)  Y(s) \scH^*(s) ds . 
		\end{split}
	\end{equation}
	Furthermore, it holds that 
	\begin{equation}\label{eq:1458b}
		\begin{split}
			& \mathscr{A} (\rho+\mathscr{A})^{-1} (\rho+2\mathscr{A})  
			- (\rho+\mathscr{A})^{-1} (\rho+2\mathscr{A}) 
			\mathscr{A} \\
			& = \mathscr{A} + \mathscr{A}  (\rho+\mathscr{A})^{-1} \mathscr{A}   
			- 
			\mathscr{A} 
			- (\rho+\mathscr{A})^{-1} \mathscr{A}  
			\mathscr{A} \\
			& = \mathscr{A}  (\rho+\mathscr{A})^{-1} \mathscr{A} 
			- (\rho+\mathscr{A})^{-1} \mathscr{A}  
			\mathscr{A} . 
		\end{split}
	\end{equation}
	Observe that 
	\begin{equation*}
		\begin{split}
			(\rho+\mathscr{A})^{-1} \mathscr{A} 
			& = \tfrac12 \OO^\top \big( \diag(\wt{\rho}_1,\dots,\wt{\rho}_n) + \tfrac12 \diagmu \big)^{-1} \OO  \OO^\top \diagmu \OO \\
			& =  \tfrac12 \OO^\top \diagmu \big( \diag(\wt{\rho}_1,\dots,\wt{\rho}_n) + \tfrac12 \diagmu \big)^{-1} \OO \\
			& = \mathscr{A} \OO^\top \big( \diag(\wt{\rho}_1,\dots,\wt{\rho}_n) + \tfrac12 \diagmu \big)^{-1} \OO 
			= \mathscr{A} (\rho+\mathscr{A})^{-1} .
		\end{split}
	\end{equation*}
	This, \eqref{eq:1458a}, and \eqref{eq:1458b} demonstrate that $D^*$ is constant on $(0,T)$. 
	Since $\gamma^{-\frac12}$, $\theta$, and $\scH^*$ here all are of finite variation, 
	we obtain that the optimal strategy, which is given by $X^*(0-)=0$, $X^*(T)=0$, and $X^*(s)=\gamma^{-\frac12}(s) (\theta(s) - I_n) \scH^*(s)$, $s\in[0,T)$, has finite variation. 
	Further, \eqref{eq:A1} and~\eqref{eq:A2} are satisfied. 
	We thus conclude that $X^*\in\cA^{fv}$.
\end{proof}

Note that the price impact $\gamma$, the BSDE solution $\mathscr{Y}\equiv Y$, the auxiliary process $\theta$, the optimal state $\scH^*$, the optimal strategy $X^*$, and the optimal deviation $D^*$ in \cref{lem:example_cross_impact_constant_dev} are deterministic.
To point out that our model allows for stochastic liquidity, we next provide a simulation of the optimal strategy in a setting with cross-impact where the price impact $\gamma$ is a (non-deterministic) stochastic process and where the resulting optimal strategy is no longer deterministic. 

\begin{ex}\label{ex:opt_strat_sigma}
	\begin{SCfigure}
		\caption{The optimal strategy in the setting of \cref{ex:opt_strat_sigma}, where $T=1$, $x=(100,0)^\top$, $\risk=0$, $\rho=I_2$, and with $\gamma$ from \eqref{eq:gamma_in_opt_strat_sigma}.}
		\label{fig:opt_strat_sigma}
		\includegraphics[scale=0.75]{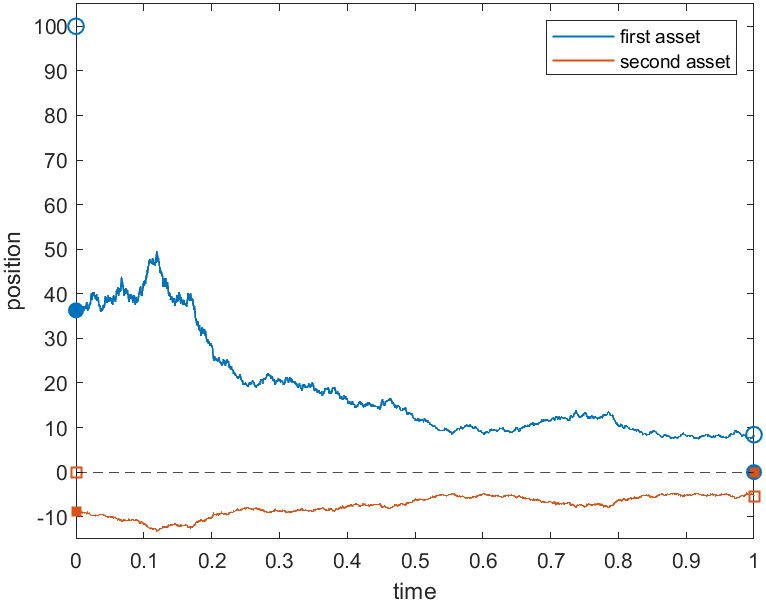}
	\end{SCfigure}
	\begin{SCfigure}
		\caption{The deviation associated to the optimal strategy of \cref{fig:opt_strat_sigma} in the same setting as in \cref{fig:opt_strat_sigma}.}
		\label{fig:opt_dev_sigma}
		\includegraphics[scale=0.75]{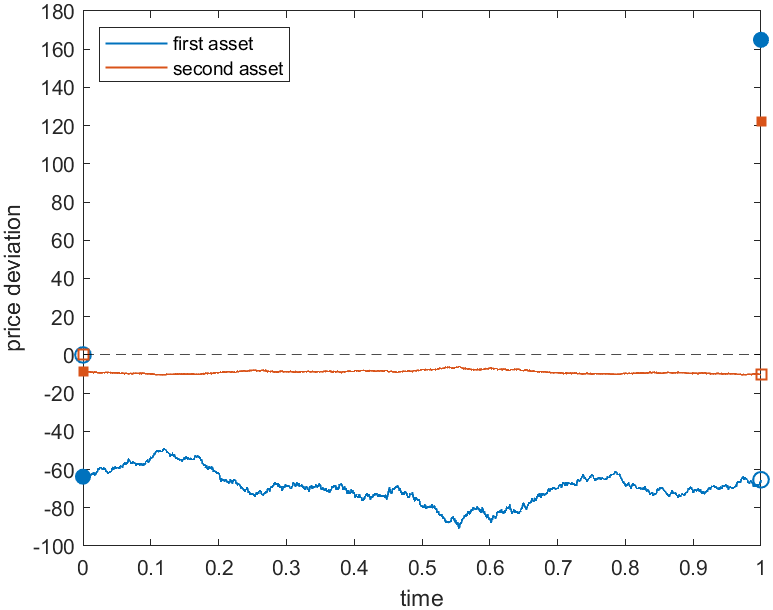}
	\end{SCfigure}
	Within \cref{set:crossimpact} let 
	$n=2$, $m=1$, $d=0$, $x=(100,0)^\top$, $\lambda_1(0)=1=\lambda_2(0)$, $\wt{\rho}_1=1=\wt{\rho}_2$, $\mu=(3,1)^\top$, and $\sigma=(1,1)^\top$. 
	Furthermore, let $\OO$ be defined as in \eqref{eq:OO_in_opt_strat_gamma} in \cref{ex:opt_strat_gamma}. 
	Note that the only difference in the present set-up to the one in \cref{ex:opt_strat_gamma} is that we now include a non-zero $\sigma$. 
	The price impact $\gamma$ is then given by 
	\begin{equation}\label{eq:gamma_in_opt_strat_sigma}
		\begin{split}
			\gamma(s) 
			& = \frac{1}{25}
			\begin{pmatrix}
				9 e^{3 s} + 16 e^{s} & 12 ( e^{3 s} - e^{s} ) \\
				12( e^{3s} - e^{s} ) & 16 e^{3s} + 9 e^{s}
			\end{pmatrix}
			e^{W_1(s)-\frac{s}{2}}
			, \quad s \in [0,T].
		\end{split}
	\end{equation}
	Observe that the price impact $\gamma$ and the optimal state $\scH^*$ are non-deterministic, whereas the solution of the BSDE~\eqref{eq:BSDE} and $\theta$ in \eqref{eq:def_theta} are still deterministic. 
	The optimal strategy (see \cref{fig:opt_strat_sigma}) and the associated deviation process (see \cref{fig:opt_dev_sigma}) are non-deterministic.
	In particular, the deviation process is not constant on $(0,T)$.
\end{ex}

\end{document}